\documentclass[11pt]{article}

\usepackage{amsmath,amssymb,amsthm,amscd,eucal}
\usepackage{times,tikz}
\usetikzlibrary{backgrounds}

\newtheorem{thm}[equation]{Theorem}

\newtheorem{lemma}[equation]{Lemma}
\newtheorem{prop}[equation]{Proposition}

\newtheorem*{conj*}{Conjecture}
\theoremstyle{definition}
\newtheorem{definition}[equation]{Definition}

\newtheorem{remark}[equation]{Remark}
\newtheorem{example}[equation]{Example} 
\newtheorem{examples}[equation]{Examples}

\numberwithin{equation}{section}
 
\usepackage[letterpaper,margin=1.25in]{geometry}

\def\End{ \mathsf{End}}

\def\min{\mathsf {min}}

\def\ot{\otimes}

\newcommand{\proj}{\varepsilon}

\newcommand{\e}{\mathsf{e}}

\newcommand{\ve}{\varepsilon}
\newcommand{\bs}{\mathsf{b}}
\newcommand{\sff}{\mathsf{s}}
\newcommand{\EE}{\mathsf{E}}  
\newcommand{\GG}{\mathsf{G}}

\newcommand{\qs}{\mathsf{q}}   
\newcommand{\rs}{\mathsf{r}}
    
\newcommand{\tf}{\mathsf{t}}  
\newcommand{\us}{\mathsf{u}}  
\newcommand{\vr}{\varrho}

\renewcommand{\S}{\mathsf{S}}

\renewcommand{\P}{\mathsf{P}}

\newcommand{\FF}{\mathbb{F}}  
  
\newcommand{\ZZ}{\mathbb{Z}}  
\newcommand{\CC}{\mathbb{F}}  
   
\newcommand{\MM}{\mathsf{N}}

\newcommand{\para}{\xi} 
\newcommand{\ef}{\mathsf{e}}
\newcommand{\pn}{\mathsf{pn}}

\newcommand{\M}{\mathsf{M}}

\newcommand{\Xie}{\Xi}
\newcommand{\modu}{\mathsf{M}_n}
\newcommand{\VV}{\mathsf{V}}
   
\newcommand{\wtp}{\widetilde{\mathfrak p_1}}

\newcommand{\half}{\frac{1}{2}}
 
\newcommand{\dimm}{\mathsf{dim}\,}
\newcommand{\spann}{\mathsf{span}} 
\newcommand{\im}{\mathsf{im}\,}
\renewcommand{\ker}{\mathsf{ker}\,}
\newcommand{\Proj}{\Xi}

\newcommand{\vertedge}{ \!\!
\begin{array}{c}
\begin{tikzpicture}[xscale=.3,yscale=.3,line width=1.0pt] 
\foreach \i in {1}  { \path (\i,1.25) coordinate (T\i); \path (\i,.25) coordinate (B\i); } 
\draw[blue] (B1) -- (T1);
\foreach \i in {1}  { \filldraw[fill=white,draw=black,line width = 1pt] (T\i) circle (5pt); \filldraw[fill=white,draw=black,line width = 1pt]  (B\i) circle (5pt); } 
\end{tikzpicture}
\end{array}\!\!}
\newcommand{\blvertedge}{ \!\!
\begin{array}{c}
\begin{tikzpicture}[xscale=.3,yscale=.3,line width=1.0pt] 
\foreach \i in {1}  { \path (\i,1.0) coordinate (T\i); \path (\i,0) coordinate (B\i); } 
\draw[blue] (B1) -- (T1);
\foreach \i in {1}  { \filldraw[fill=black,draw=black,line width = 1pt] (T\i) circle (5pt); \filldraw[fill=black,draw=black,line width = 1pt]  (B\i) circle (5pt); } 
\end{tikzpicture}
\end{array}\!\!}
\newcommand{\novertedge}{ \!\!
\begin{array}{c}
\begin{tikzpicture}[xscale=.3,yscale=.3,line width=1.0pt] 
\foreach \i in {1}  { \path (\i,1.25) coordinate (T\i); \path (\i,.25) coordinate (B\i); } 
\foreach \i in {1}  { \filldraw[fill=white,draw=black,line width = 1pt] (T\i) circle (5pt); \filldraw[fill=white,draw=black,line width = 1pt]  (B\i) circle (5pt); } 
\end{tikzpicture}
\end{array}\!}

\def\abcd{{
\begin{tikzpicture}[xscale=.5,yscale=.5,line width=1.25pt] 
\foreach \i in {1,2} 
{ \path (\i,1.25) coordinate (T\i); \path (\i,.25) coordinate (B\i); } 
\filldraw[fill= black!12,draw=black!12,line width=4pt]  (T1) -- (T2) -- (B2) -- (B1) -- (T1);
\draw[blue] (T2) -- (B1);\draw[blue] (T1) -- (B2);\draw[blue] (T1) -- (B1);\draw[blue] (T2) -- (B2);
\draw[blue] (T1) -- (T2);\draw[blue] (B1) -- (B2);
\foreach \i in {1,2} 
{ \fill (T\i) circle (4pt); \fill (B\i) circle (4pt); } 
\end{tikzpicture}}}

\def\abcld{{
\begin{tikzpicture}[xscale=.5,yscale=.5,line width=1.25pt] 
\foreach \i in {1,2} 
{ \path (\i,1.25) coordinate (T\i); \path (\i,.25) coordinate (B\i); } 
\filldraw[fill= black!12,draw=black!12,line width=4pt]  (T1) -- (T2) -- (B2) -- (B1) -- (T1);
\draw[blue] (T1) -- (B2);\draw[blue] (T1) -- (B1);
\draw[blue] (B1) -- (B2);
\foreach \i in {1,2} 
{ \fill (T\i) circle (4pt); \fill (B\i) circle (4pt); } 
\end{tikzpicture}}}

\def\abdlc{{
\begin{tikzpicture}[xscale=.5,yscale=.5,line width=1.25pt] 
\foreach \i in {1,2} 
{ \path (\i,1.25) coordinate (T\i); \path (\i,.25) coordinate (B\i); } 
\filldraw[fill= black!12,draw=black!12,line width=4pt]  (T1) -- (T2) -- (B2) -- (B1) -- (T1);
\draw[blue] (B1) -- (B2);\draw[blue] (T2) -- (B2);
\draw[blue] (B1) -- (T2);
\foreach \i in {1,2} 
{ \fill (T\i) circle (4pt); \fill (B\i) circle (4pt); } 
\end{tikzpicture}}}
 
\def\ablcld{{
\begin{tikzpicture}[xscale=.5,yscale=.5,line width=1.25pt] 
\foreach \i in {1,2} 
{ \path (\i,1.25) coordinate (T\i); \path (\i,.25) coordinate (B\i); } 
\filldraw[fill= black!12,draw=black!12,line width=4pt]  (T1) -- (T2) -- (B2) -- (B1) -- (T1);
\draw[blue] (B1) -- (B2);
\foreach \i in {1,2} 
{ \fill (T\i) circle (4pt); \fill (B\i) circle (4pt); } 
\end{tikzpicture}}}

\def\aclbld{{
\begin{tikzpicture}[xscale=.5,yscale=.5,line width=1.25pt] 
\foreach \i in {1,2} 
{ \path (\i,1.25) coordinate (T\i); \path (\i,.25) coordinate (B\i); } 
\filldraw[fill= black!12,draw=black!12,line width=4pt]  (T1) -- (T2) -- (B2) -- (B1) -- (T1);
\draw[blue] (B1) -- (T1);
\foreach \i in {1,2} 
{ \fill (T\i) circle (4pt); \fill (B\i) circle (4pt); } 
\end{tikzpicture}}}

\def\albcld{{
\begin{tikzpicture}[xscale=.5,yscale=.5,line width=1.25pt] 
\foreach \i in {1,2} 
{ \path (\i,1.25) coordinate (T\i); \path (\i,.25) coordinate (B\i); } 
\filldraw[fill= black!12,draw=black!12,line width=4pt]  (T1) -- (T2) -- (B2) -- (B1) -- (T1);
\draw[blue] (B2) -- (T1);
\foreach \i in {1,2} 
{ \fill (T\i) circle (4pt); \fill (B\i) circle (4pt); } 
\end{tikzpicture}}}

\def\alblcld{{
\begin{tikzpicture}[xscale=.5,yscale=.5,line width=1.25pt] 
\foreach \i in {1,2} 
{ \path (\i,1.25) coordinate (T\i); \path (\i,.25) coordinate (B\i); } 
\filldraw[fill= black!12,draw=black!12,line width=4pt]  (T1) -- (T2) -- (B2) -- (B1) -- (T1);
\foreach \i in {1,2} 
{ \fill (T\i) circle (4pt); \fill (B\i) circle (4pt); } 
\end{tikzpicture}}}

\def\adlblc{{
\begin{tikzpicture}[xscale=.5,yscale=.5,line width=1.25pt] 
\foreach \i in {1,2} 
{ \path (\i,1.25) coordinate (T\i); \path (\i,.25) coordinate (B\i); } 
\filldraw[fill= black!12,draw=black!12,line width=4pt]  (T1) -- (T2) -- (B2) -- (B1) -- (T1);
\draw[blue] (B1) -- (T2);
\foreach \i in {1,2} 
{ \fill (T\i) circle (4pt); \fill (B\i) circle (4pt); } 
\end{tikzpicture}}}

\def\orbitadlblc{{
\begin{tikzpicture}[xscale=.5,yscale=.5,line width=1.25pt] 
\foreach \i in {1,2} 
{ \path (\i,1.25) coordinate (T\i); \path (\i,.25) coordinate (B\i); } 
\filldraw[fill= black!12,draw=black!12,line width=4pt]  (T1) -- (T2) -- (B2) -- (B1) -- (T1);
\draw[blue] (B1) -- (T2);
\foreach \i in {1,2} 
{ \filldraw[fill=white,draw=black,line width = 1pt] (T\i) circle (4pt); \filldraw[fill=white,draw=black,line width = 1pt]  (B\i) circle (4pt); } 
\end{tikzpicture}}}

\def\albdlc{{
\begin{tikzpicture}[xscale=.5,yscale=.5,line width=1.25pt] 
\foreach \i in {1,2} 
{ \path (\i,1.25) coordinate (T\i); \path (\i,.25) coordinate (B\i); } 
\filldraw[fill= black!12,draw=black!12,line width=4pt]  (T1) -- (T2) -- (B2) -- (B1) -- (T1);
\draw[blue] (B2) -- (T2);
\foreach \i in {1,2} 
{ \fill (T\i) circle (4pt); \fill (B\i) circle (4pt); } 
\end{tikzpicture}}}

\def\orbitalblcld{{
\begin{tikzpicture}[xscale=.5,yscale=.5,line width=1.25pt] 
\foreach \i in {1,2} 
{ \path (\i,1.25) coordinate (T\i); \path (\i,.25) coordinate (B\i); } 
\filldraw[fill= black!12,draw=black!12,line width=4pt]  (T1) -- (T2) -- (B2) -- (B1) -- (T1);
\foreach \i in {1,2} 
{ 
\filldraw[fill=white,draw=black,line width = 1pt] (T\i) circle (4pt); \filldraw[fill=white,draw=black,line width = 1pt]  (B\i) circle (4pt); 
} 
\end{tikzpicture}}}

\def\ablcd{{
\begin{tikzpicture}[xscale=.5,yscale=.5,line width=1.25pt] 
\foreach \i in {1,2} 
{ \path (\i,1.25) coordinate (T\i); \path (\i,.25) coordinate (B\i); } 
\filldraw[fill= black!12,draw=black!12,line width=4pt]  (T1) -- (T2) -- (B2) -- (B1) -- (T1);
\draw[blue] (B2) -- (B1);\draw[blue] (T2) -- (T1);
\foreach \i in {1,2} 
{ \fill (T\i) circle (4pt); \fill (B\i) circle (4pt); } 
\end{tikzpicture}}}

\def\aclbd{{
\begin{tikzpicture}[xscale=.5,yscale=.5,line width=1.25pt] 
\foreach \i in {1,2} 
{ \path (\i,1.25) coordinate (T\i); \path (\i,.25) coordinate (B\i); } 
\filldraw[fill= black!12,draw=black!12,line width=4pt]  (T1) -- (T2) -- (B2) -- (B1) -- (T1);
\draw[blue] (T1) -- (B1); 
\draw[blue] (T2) -- (B2); 
\foreach \i in {1,2} 
{ \fill (T\i) circle (4pt); \fill (B\i) circle (4pt); } 
\end{tikzpicture}}}

\def\acdlb{{
\begin{tikzpicture}[xscale=.5,yscale=.5,line width=1.25pt] 
\foreach \i in {1,2} 
{ \path (\i,1.25) coordinate (T\i); \path (\i,.25) coordinate (B\i); } 
\filldraw[fill= black!12,draw=black!12,line width=4pt]  (T1) -- (T2) -- (B2) -- (B1) -- (T1);
\draw[blue] (T2) -- (B1) -- (T1) -- (T2);
\foreach \i in {1,2} 
{ \fill (T\i) circle (4pt); \fill (B\i) circle (4pt); } 
\end{tikzpicture}}}

\def\adlbc{{
\begin{tikzpicture}[xscale=.5,yscale=.5,line width=1.25pt] 
\foreach \i in {1,2} 
{ \path (\i,1.25) coordinate (T\i); \path (\i,.25) coordinate (B\i); } 
\filldraw[fill= black!12,draw=black!12,line width=4pt]  (T1) -- (T2) -- (B2) -- (B1) -- (T1);
\draw[blue] (T2) -- (B1);
\draw[blue] (T1) -- (B2);
\foreach \i in {1,2} 
{ \fill (T\i) circle (4pt); \fill (B\i) circle (4pt); } 
\end{tikzpicture}}}

\def\albcd{{
\begin{tikzpicture}[xscale=.5,yscale=.5,line width=1.25pt] 
\foreach \i in {1,2} 
{ \path (\i,1.25) coordinate (T\i); \path (\i,.25) coordinate (B\i); } 
\filldraw[fill= black!12,draw=black!12,line width=4pt]  (T1) -- (T2) -- (B2) -- (B1) -- (T1);
\draw[blue] (T2) -- (B2)--(T1)--(T2);
\foreach \i in {1,2} 
{ \fill (T\i) circle (4pt); \fill (B\i) circle (4pt); } 
\end{tikzpicture}}}

\def\alblcd{{
\begin{tikzpicture}[xscale=.5,yscale=.5,line width=1.25pt] 
\foreach \i in {1,2} 
{ \path (\i,1.25) coordinate (T\i); \path (\i,.25) coordinate (B\i); } 
\filldraw[fill= black!12,draw=black!12,line width=4pt]  (T1) -- (T2) -- (B2) -- (B1) -- (T1);
\draw[blue] (T2) -- (T1);
\foreach \i in {1,2} 
{ \fill (T\i) circle (4pt); \fill (B\i) circle (4pt); } 
\end{tikzpicture}}}

\tikzstyle{rep}=[circle,
                                    thick,
                                    minimum size=.6cm,
                                    inner sep=0mm,
                                    draw= black,  
                                     fill=black!10] 
                                    
\tikzstyle{norep}=[circle,
                                    thick,
                                    minimum size=.2cm,
                                    draw= white,  
                                    fill=white]      
                                  
\title{Partition algebras $\P_k(n)$ with $2k>n$ \\
and the fundamental theorems of invariant theory \\
for the symmetric group $\S_n$}

\author{Georgia Benkart \\
 {\small University of Wisconsin-Madison}
 \\{\small Madison, WI 53706, USA}\\
\texttt{\small benkart@math.wisc.edu}
\and Tom Halverson\footnote{
The second author gratefully acknowledges partial support from Simons Foundation grant 283311.}\\
{\small Macalester College} \\{\small Saint Paul, MN 55105, USA}\\
\texttt{\small halverson@macalester.edu}
} 
 
\begin{document}

\date{}

\maketitle   

\begin{abstract}
Assume $\modu$ is the $n$-dimensional permutation module for the symmetric group $\S_n$,  and let $\modu^{\otimes k}$ be its $k$-fold tensor power.   The partition algebra $\P_k(n)$ maps surjectively onto the centralizer algebra $\End_{\S_n}(\modu^{\otimes k})$ for all $k, n \in \ZZ_{\ge 1}$ and  isomorphically when $n \ge 2k$.  We describe the image of the surjection $\Phi_{k,n}:\P_k(n) \to \End_{\S_n}(\modu^{\otimes k})$  explicitly in terms of the orbit basis of $\P_k(n)$ and  show that when $2k > n$ the kernel of $\Phi_{k,n}$ is generated by a single  essential idempotent  $\ef_{k,n}$, which is an orbit basis element.   We obtain a presentation for $\End_{\S_n}(\modu^{\otimes k})$ by imposing one additional relation, $\ef_{k,n} = 0$,  to the standard presentation of the partition algebra $\P_k(n)$ when $2k > n$.   
As a consequence, we obtain the fundamental theorems of invariant theory for the symmetric group $\S_n$. 
We show under the natural embedding of the partition algebra $\P_n(n)$ into $\P_k(n)$ for $k \ge n$ that  the essential idempotent $\ef_{n,n}$ generates the kernel of  $\Phi_{k,n}$.  Therefore, the relation $\ef_{n,n} = 0$ can replace  $\ef_{k,n} = 0$ when $k \ge n$. 
\end{abstract}


 \medskip  
\section{Introduction}   \emph{We assume throughout that $\FF$ is a field of characteristic 0.}   
The partition algebras $\P_k(\para)$, $\para \in \CC \setminus\{0\}$,  were introduced by Martin (\cite{M1},\cite{M2},\cite{M3})   
 to study the  Potts lattice model of interacting spins in statistical mechanics.    The set partitions
of the set $\{1,2,\ldots,2k\}$ index a $\CC$-basis for the partition algebra $\P_k(\para)$, and thus,  $\P_k(\para)$ has dimension equal to the Bell number $\mathsf{B}(2k)$.
We let $\Pi_{2k}$ be the set of set partitions of $\{1,2, \ldots, 2k\}$.   For example,  
$\big\{1,8,9,10 \,|\,  2,3 \,|\,  4,7\,|\,  5,6,11,12,14\,|\, 13\big\}$ is a set partition in
$ \Pi_{14}$ with 5 blocks (subsets).      The algebra $\P_k(\para)$ has 
two distinguished bases --  the diagram basis $\{ d_\pi \mid \pi \in \Pi_{2k}\}$ and the orbit basis $\{ x_\pi \mid \pi \in \Pi_{2k}\}$.
The diagram basis elements $d_\pi$ in $\P_{k+1}(\para)$ corresponding to set partitions $\pi$ having  $k+1$ and $2(k+1)$ in the same block form a subalgebra  
$\P_{k+\half}(\para)$ of $\P_{k+1}(\para)$  under diagram multiplication.
Identifying $\P_k(\para)$ with the span of the diagrams in $\P_{k+\half}(\para)$ which have a block consisting solely
of the two elements $k+1,2(k+1)$ gives a tower of algebras,

\begin{equation}\label{eq:tower} \CC  = \P_0(\para) \cong  \P_\half (\para) \subset  \P_1(\para) \subset \cdots \subset \P_k(\para) \subset \P_{k+\half}(\para) \subset \P_{k+1}(\para) \subset  \cdots \  . \end{equation}

As shown by  Jones \cite{J}, when $\xi = n \in \ZZ_{\ge 1}$,  there is a Schur-Weyl duality between the 
partition algebra $\P_k(n)$ and the symmetric group $\S_n$ acting as centralizers
of one another on the $k$-fold tensor power $\modu^{\ot k}$  of the $n$-dimensional permutation module
$\modu$ for $\S_n$ over  $\CC$.   The surjective algebra homomorphism given in \cite{J} (see also \cite[Thm.\,3.6]{HR} and Section \ref{sec:SWduality} below), 
\begin{equation}\label{eq:phirep}
\Phi_{k,n}: \P_k(n) \rightarrow  \End_{\S_n}(\modu^{\ot k}) = \{T \in \End(\modu^{\ot k}) \mid
T \sigma = \sigma T\ \  \forall  \sigma \in \S_n\},
\end{equation}
is an isomorphism when $n \geq 2k$.   
The partition algebra $\P_k(n)$ acts naturally on $\modu^{\ot k}$ via the representation $\Phi_{k,n}$ in
\eqref{eq:phirep}, and  
 if we regard  $\modu$ as a module for the symmetric group $\S_{n-1}$ by restriction,  there is a surjective algebra homomorphism $\Phi_{k+\half}(n):\P_{k+\half}(n) \rightarrow \End_{\S_{n-1}}(\modu^{\ot {k}})$,  which is an isomorphism if $n \geq 2k+1$ (the details can be found in Remark \ref{SWhalf}).    The intermediate algebras $\P_{k+\half}(n)$  have 
 played an important role in \cite{MR,HR} in 
understanding the structure and representation theory  of partition algebras.

In studying the partition algebras $\P_k(n)$, for example, in determining their representation theory \cite{M2,M3,M4,HR} and character theory \cite{H}, it is helpful to fix $k$ and assume $n \ge 2k$,  so that $\P_k(n)$ is isomorphic to $\End_{\S_n}(\modu^{\ot k})$.  Under that assumption, $\P_k(n)$ is a semisimple associative algebra, and the full power of Schur-Weyl duality can be applied to use the representation theory of $\S_n$ to derive results about $\P_k(n)$.   However, if
proving results about the symmetric group $\S_n$ and its representations  is the goal, as it is  in studies of Kronecker coefficients and symmetric functions in \cite{BDO1,BDO2,  OZ}, then fixing $n$ and letting $k$ grow arbitrarily large is the more natural setting.   This  amounts to considering increasingly large tensor powers $\M_n^{\otimes k}$ of $\M_n$.  When $2k > n$,  it is critical to determine the kernel of the homomorphism $\Phi_{k,n}$  and the dependence relations that are imposed in the image of $\Phi_{k,n}$.   The orbit basis,
its relation to the diagram basis, and the rule for multiplying elements in the orbit basis
are essential ingredients for describing the kernel and image of $\Phi_{k,n}$.  

In this paper,  we
\begin{itemize} 
\item [{\rm(i)}]\, describe the change of basis matrix between the diagram basis and the orbit basis in terms of the M\"obius function  of the set partition lattice \ \ (\emph {Section \ref{S:change}});
\item [{\rm(ii)}]\, prove a rule for multiplication in the orbit basis,  originally stated by Halverson and Ram in unpublished notes \ \ (\emph{Theorem \ref{C:mult},  see also 
Lemma \ref{T:mult} and Examples \ref{exs:mult}});
\item [{\rm(iii)}]\,  show for $k \in \half  \ZZ_{\ge 1}$ and $2k >n$  that  the kernel of the surjection $\Phi_{k,n}$
 is generated as a two-sided ideal of $\P_k(n)$ by  a single essential idempotent element $\ef_{k,n}$, which is the orbit 
 basis element
 defined in \eqref{eq:ef} for $k \in \ZZ_{\ge 1}$,  and in \eqref{eq:efhalf} for $k \in  \half\ZZ_{\ge 1}\setminus \ZZ_{\ge 1}$   \ \  (\emph{Theorem \ref{thm:generator}, see also Theorem \ref{T:secfund})};
 \item[{\rm(v)}]\,  establish the Second Fundamental Theorem of Invariant Theory for the symmetric group $\S_n$ \ \ (\emph{Theorem \ref{T:2ndfund}});
\item [{\rm(vi)}]\, construct  for $k \in \ZZ_{\ge 1}$ a certain rational linear combination $\Xie_{k,n}$  of orbit basis elements
in $\P_k(n)$ (see \eqref{eq:Up})  and prove that  $\Xie_{k,n} = \Psi_{k,n}(\ve_{[n-k,k]})$ when $n \geq 2k$, 
where $\Psi_{k,n}(\ve_{[n-k,k]})$ is the
 image of  the primitive central idempotent  $\ve_{[n-k,k]}$ in the group algebra $\CC\S_n$ corresponding to the  two-part integer partition $[n-k,k]$ of $n$ under the representation $\Psi_{k,n}: \CC\S_n \rightarrow \End(\modu^{\ot k})$\ \ (\emph{Theorem \ref{IdempotentOrbitBasis}}).  The primitive central idempotent $\Xie_{k,n}$ of $\P_k(n)$ corresponds to the one-dimensional
 irreducible $\P_k(n)$-module $\P_{k,n}^{[n-k,k]}$;   
\item [{\rm(vii)}]\, construct for $k \in \ZZ_{\ge 1}$ a certain rational linear combination $\Xie_{k+\half,n}$  of orbit basis elements in $\P_{k+\half}(n)$ (see \eqref{eq:Xihalffirst} and also \eqref{eq:Xihalf})
and prove that  $\Xi_{k+\half,n} = \Psi_{k+\half,n}(\ve_{[n-1-k,k]})$ when $n \geq 2k+1$, 
where $\Psi_{k+\half,n}(\ve_{[n-1-k,k]})$ is the
 image of  the primitive central idempotent  $\ve_{[n-1-k,k]}$ in $\CC\S_{n-1}$ corresponding to the integer partition $[n-1-k,k]$ of $n-1$ under the representation $\Psi_{k+\half,n}: \CC\S_{n-1} \rightarrow \End(\modu^{\ot k})$ \ \ (\emph{Theorem \ref{T:IdOB})};   
\item [{\rm(viii)}]\, prove for $k \in \ZZ_{\ge 1}$  that the kernels of the surjections  
\begin{align*} &\Phi_{k,2k-1}:  \P_k(2k-1) \rightarrow \End_{\S_{2k-1}}(\mathsf{M}_{2k-1}^{\ot k}) \\
&\Phi_{k+\half,2k}: \P_{k+\half}(2k) \rightarrow \End_{\S_{2k}}(\M_{2k}^{\ot k}) \end{align*} 
are one-dimensional; more precisely,   $\ker \Phi_{k,2k-1}$ is spanned by $\Xie_{k,2k-1}$ which equals $\ef_{k,2k-1}$ up to a scalar multiple  (\emph{Theorem \ref{T:Xief}\,(a)}),  and $\ker \Phi_{k+\half, 2k}$ is spanned by $\Xi_{k+\half,2k}$ which equals $\ef_{k+\half,2k}$ up to a scalar multiple  (\emph{Theorem \ref{T:Xief}\,(b))}.  The largest value of $n$ for which the
kernel is nontrivial is $n=2k-1$ for $k \in \ZZ_{\ge 1}$ and $n = 2k$ for $k+\half$.   Moreover, the essential idempotents  $\ef_{k,2k-1}$ and $\ef_{k+\half, 2k}$ 
are the only central ones among the kernel generators $\ef_{k,n}$  (\emph{Remark} \ref{R:central}), so these elements are noteworthy for their  exceptional behavior.  \end{itemize}

In the classical invariant theory of a group $\mathsf{G}$ via endomorphism algebras, there is a surjection $\Psi:\mathsf{A} \rightarrow \End_{\GG}(\VV^{\ot k})$
from a finite-dimensional associative algebra $\mathsf{A}$ onto the centralizer algebra $\End_{\GG}(\VV^{\ot k})$
of endomorphisms that commute with the action of $\GG$ on tensor powers of its natural module $\VV$.    The generators and relations of $\mathsf{A}$ afford the First Fundamental
Theorem of Invariant Theory for $\GG$.   The Second Fundamental Theorem of Invariant Theory for $\GG$ describes generators for the kernel of  $\Psi$.
As a $\GG$-module,  $\End_\GG(\VV^{\ot k})$ is isomorphic to the space of $\GG$-invariants in $(\VV \ot \VV^*)^{\ot k}$,
which is isomorphic to the $\GG$-invariants in $\VV^{\ot 2k}$ when $\VV$ is isomorphic to its dual $\VV^*$ as a $\GG$-module.

For the general linear group $\GG = \mathsf{GL}_n$ and the tensor
power $\VV^{\ot k}$ of its defining module $\VV = \CC^n$,   the algebra $\mathsf{A}$ is the group algebra $\CC \S_k$ of the symmetric group $\S_k$, where the endomorphisms in  $\Psi(\S_k)$ act by place permutation of the tensor factors of $\VV^{\ot k}$.  The standard generators and 
relations for $\S_k$ provide the  first fundamental theorem in this setting, and the second fundamental theorem states
that a  generator of the kernel is the essential idempotent $e = \sum_{\sigma \in \S_{n+1}} (-1)^{\mathsf{sgn}(\sigma)} \sigma$
for all $k \ge n+1$.   

 In \cite{Br},  Brauer  introduced
diagrammatic algebras, now known as \emph{Brauer algebras}, that centralize the action of the orthogonal group $\mathsf{O}_n$ and symplectic group $\mathsf{Sp}_{2n}$ on
tensor powers of their defining modules.  The surjective algebra homomorphisms
$\mathsf{B}_k(n) \to \End_{\mathsf{O}_n}(\VV^{\ot k})$ ($\dimm \VV=n$) and  $\mathsf{B}_k(-2n) \to
\End_{\mathsf{Sp}_{2n}}(\VV^{\ot k})$ ($\dimm \VV = 2n$) defined in \cite{Br} provide 
the First Fundamental Theorem of 
Invariant Theory for these groups.  (See, for example, \cite[Sec.~4.3.2]{GW} for an exposition 
of these results.)  Generators for the kernels of these surjections give the Second Fundamental
Theorem of Invariant Theory for the orthogonal and symplectic groups.  
As shown in the work of Hu and Xiao \cite{HX}, Lehrer and Zhang \cite{LZ1, LZ2}, and Rubey and Westbury \cite{RW1}, the kernels of these surjections are  principally generated by a single idempotent when $k \ge n+1$.  The recent work of Bowman, Enyang, and Goodman \cite{BEG} adopts a cellular 
basis approach to describing the kernels in the orthogonal and symplectic cases, as well as in the case
of the general linear group $\mathsf{GL}_n$ acting on mixed tensor powers $\VV^{\ot k} \ot (\VV^*)^{\ot \ell}$
of its natural $n$-dimensional module $\VV$ and its dual $\VV^*$.   A surjection of the walled Brauer
algebra $\mathsf{B}_{k,\ell}(n) \to \End_{\mathsf{GL}_n}(\VV^{\ot k} \ot (\VV^*)^{\ot \ell})$  is used for this purpose. (The algebra $\mathsf{B}_{k,\ell}(n)$ and some of its representation-theoretic properties including its action on $\VV^{\ot k} \ot (\VV^*)^{\ot \ell}$ can be found, for example,  in \cite{BCHLLS}.)

In \cite[Sec.~7.4]{RW1} (see also \cite{RW2}),  Rubey and Westbury consider the Brauer algebra $\mathsf{B}_k(-2n)$, the related Brauer
diagram category, and the commuting actions of $\mathsf{B}_k(-2n)$ and the symplectic group
$\mathsf{Sp}_{2n}$ afforded by the above surjection.    They show that the central idempotent $E = \frac{1}{(n+1)!} \sum_{d \in \mathcal{Br}_{n+1}} d$, obtained by summing all the Brauer diagrams $d$
with $2n+2$ vertices, corresponds  to the
one-dimensional trivial $\mathsf{B}_{n+1}(-2n)$-module and generates the kernel of
the surjection $\mathsf{B}_k(-2n) \to \End_{\mathsf{Sp}_{2n}} (\VV^{\ot k})$ for all $k \ge n+1$.   As a result,
they obtain the fundamental theorems of invariant theory for the symplectic groups from  
Brauer algebra considerations. The Brauer diagram category also is a key ingredient in the papers  of Hu and Xiao \cite{HX}  and of  Lehrer and Zhang
 \cite{LZ1, LZ2}  in proving  that the kernel of the surjection is principally
generated and in establishing analogous results for the quantum version of the Brauer algebra, the Birman-Murakami-Wenzl algebra.   
  
Theorem \ref{T:present} in Section 2 below
gives a presentation for $\P_k(n)$ by generators and relations, and Theorem \ref{T:Phi} (due originally to Jones \cite{J}, see also \cite[Thm.~3.6]{HR}) describes
a basis for the image and kernel of $\Phi_{k,n}: \P_k(n) \to \End_{\S_n}(\M_n^{\ot k})$.  These results  combine to provide the First Fundamental Theorem of Invariant Theory for the symmetric group $\S_n$.  The following theorem, which is Theorem 
\ref{T:2ndfund},  
 gives  the  second fundamental theorem.  
  
\begin{thm}\label{T:SecondFund} (Second Fundamental Theorem of Invariant Theory for $\S_n$) \, For all $k,n\in \ZZ_{\ge 1}$,   $\im \Phi_{k,n}
= \End_{\S_n}(\modu^{\ot k})$ is generated by the partition algebra generators and relations in Theorem \ref{T:present} (a)-(c) together with the one additional relation $\ef_{k,n} = 0$ in the case that $2k > n$.   When $k \ge n$,  the relation $\ef_{k,n} = 0$ can be replaced with  $\ef_{n,n} = 0$.
\end{thm} 
 
The last sentence in Theorem \ref{T:SecondFund}  is the counterpart of the results for the classical groups, and it  comes from identifying 
 $\ef_{n,n}$ with its embedded image in $\P_k(n)$ for $k \ge n$ (see Theorem \ref{T:gensk>n}).

The  essential idempotent  $\e_{k,n}$ is the orbit basis element corresponding to the set partition $\pi_{k,n}$
in \eqref{eq:ef}.    In terms of the diagram basis of $\P_k(n)$, it has the expression
$\ef_{k,n} = \sum_{\vr}  \mu_{2k}(\pi_{k,n},\varrho) d_{\varrho},$  where $\mu_{2k}(\pi_{k,n},\varrho)$ 
is the M\"obius function of the set-partition lattice, and the sum is over the $\vr \in \Pi_{2k}$ with $\pi_{k,n} \preceq \varrho$.    
In the special case that $n=2k-1$,  the corresponding set partition  $\pi = \pi_{k,2k-1}$ has $2k$ singleton blocks.   All the diagram basis elements of $\P_k(2k-1)$ occur in the expression
$\ef_{k,2k-1} = \sum_{\varrho \in \Pi_{2k}}  \mu_{2k}(\pi,\varrho) d_{\varrho}$ in that case,  and 
$\big(\ef_{k,2k-1}\big)^2 \allowbreak = \frac{(-1)^k}{k!}\,\,\ef_{k,2k-1}$ (further details can be found in Section \ref{S:change}).

In \cite[Sec.~8]{RW1}, Rubey  and Westbury discuss the representation theory of the symmetric
groups from the viewpoint of the partition category and the combinatorics of set partitions.
For $n \ge 2k$, they consider a certain one-dimensional representation  for the 
algebra $\End_{\S_n}(\mathsf{M}_n^{\ot k})$ and let $E(k)$ denote the associated central idempotent. 
When  $n \ge 2k$,   $\End_{\S_n}(\mathsf{M}_n^{\ot k}) \cong \P_k(n)$,  and the representation is given by the action on the one-dimensional  $\P_k(n)$-module  $\P_{k,n}^{[n-k,k]}$ indexed by the partition $[n-k,k]$ (in our notation).
When $n \ge 2k+1$  and the one-dimensional  representation for the 
$\End_{\S_{n}}(\mathsf{M}_n^{\ot k+1}) \cong \P_{k+1}(n)$ is restricted 
to the diagrams having $k+1$ and $2(k+1)$ in the same block, the result  is
the one-dimensional module $\P_{k+\half,n}^{[n-1-k,k]}$ of $\P_{k+\half}(n)$,  with corresponding central 
idempotent, which they denote $E'(k+1)$.  They conjecture (see \cite[Conjecture~8.4.9]{RW1}) that the idempotents $E(k)$ and $E'(k+1)$ satisfy certain recurrence relations.  We do not prove 
their recurrence relation conjecture,  but rather  in \eqref{eq:Ek} and \eqref{eq:Ek+1},  we give exact expressions for $E(k)$ and $E'(k+1)$ by 
showing that  $E(k) = \Xi_{k,n}$  when $n \ge 2k$ and $E'(k+1)=\Xi_{k+\half,n}$  when $n \ge 2k+1$.  In the final section of this paper, we  
explain these connections more fully (see especially  Remark \ref{R:nottrue}).

\section{Two Bases for $\P_k(\para)$} \label{sec:bases}

\subsection{Set partition notation}
For $k \in \ZZ_{> 0}$, we consider the set partitions of  $[1,2k]:=\{1,2, \ldots, 2k\}$ into disjoint nonempty subsets, referred to as \emph{blocks} in this context,  and define
\begin{equation}
\begin{array}{rcl}
\Pi_{2k} &=& \left\{ \text{\,set partitions of $[1,2k]$\,} \right\}, \\
\Pi_{2k-1} &=& \left\{ \pi \in \Pi_{2k} \mid k \text{ and } 2k \text{\,are in the same block of\,} \pi \right\}.
\end{array}
\end{equation}
For $\pi \in \Pi_{2k}$, we let $|\pi|$ equal the number of blocks of $\pi$. 
For example, if
\begin{equation}\label{ex:setparts}
\begin{array}{rcl}
\pi &=& \big\{1,8,9,10 \,|\,  2,3 \,|\, 4,7\,|\, 5,6,11,12,14\,|\, 13\big\} \in \Pi_{14}, \\
\varrho &=& \big\{1,8,9,10\,|\,  2,3 \,|\,4 \,|\,  5,6,11,12 \,|\, 7,13, 14\big\} \in \Pi_{13} \subseteq \Pi_{14},
\end{array}
\end{equation}
then $\pi \not \in \Pi_{13}$ and  $|\pi| = |\varrho|=5$.
 
\subsection{The diagram basis}

For  $k \in \ZZ_{\ge 1}$  and  $\pi \in \Pi_{2k}$,  the diagram $d_\pi$  of $\pi$ has two rows of $k$ vertices each,
with the bottom vertices indexed by $1,2,\dots, k$ and the top vertices indexed by $k+1,k+2, \dots, 2k$
from left to right.    Vertices are connected by an edge if they lie in the same block of $\pi$.   Thus, to the set partition $\pi$ in \eqref{ex:setparts}, we associate the diagram
$$
d_{\pi} = \begin{array}{c} 
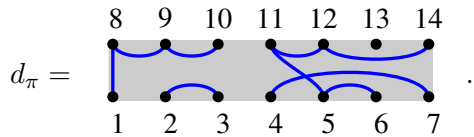
{\begin{tikzpicture}[scale=.7,line width=1.25pt] 
\foreach \i in {1,...,7} 
{ \path (\i,1) coordinate (T\i); \path (\i,0) coordinate (B\i); } 
\filldraw[fill= gray!40,draw=gray!40,line width=3.2pt]  (T1) -- (T7) -- (B7) -- (B1) -- (T1);
\draw[blue] (T1) .. controls +(.1,-.30) and +(-.1,-.30) .. (T2);
\draw[blue] (T2) .. controls +(.1,-.30) and +(-.1,-.30) .. (T3);
\draw[blue] (T1) .. controls +(0,-.30) and  +(0,.30) .. (B1);
\draw[blue] (T4) .. controls +(.1,-.30) and +(-.1,-.30) .. (T5) ;
\draw[blue] (T4) .. controls +(0,-.30) and +(0,.30) .. (B5) ;
\draw[blue] (T5) .. controls +(.1,-.38) and +(-.1,-.38) .. (T7) ;
\draw[blue] (B2) .. controls +(.1,.30) and +(-.1,.30) .. (B3) ;
\draw[blue] (B4) .. controls +(.1,.62) and +(-.1,.62) .. (B7) ;
\draw[blue] (B5) .. controls +(.1,.30) and +(-.1,.30) .. (B6) ;
\foreach \i in {1,...,7} 
{ \fill (T\i) circle (3pt); \fill (B\i) circle (3pt); } 
\draw (1,1.5) node  {{ 8}}; 
\draw (2,1.5) node {9}; 
\draw (3,1.5) node   {{10}};  
\draw (4,1.5) node  {{11}}; 
\draw (5,1.5) node  {{12}}; 
\draw (6,1.5) node  {{13}}; 
\draw (7,1.5) node  {{14}}; 
\draw (1.1,-.5) node  {{1}}; 
\draw (2.1,-.5) node  {2}; 
\draw (3.1,-.5) node  {{3}}; 
\draw (4.1,-.5) node  {{4}}; 
\draw (5.1,-.5) node  {{5}}; 
\draw (6.1,-.5) node  {{6}}; 
\draw (7.1,-.5) node {{7}}; 
\end{tikzpicture}}
\end{array}.
$$ 
The way the edges are drawn is immaterial,  what matters is that the connected components of the diagram $d_\pi$ 
correspond to the blocks of the set partition $\pi$. Thus,  $d_\pi$  represents the equivalence class of all diagrams with connected components equal to the blocks of   $\pi$.

Multiplication of two diagrams $d_{\pi_1}$, $d_{\pi_2}$ is accomplished by placing $d_{\pi_1}$ above $d_{\pi_2}$;
identifying the vertices in the bottom row of $d_{\pi_1}$ with those in the top row of $d_{\pi_2}$; 
concatenating  the edges; deleting all connected components that lie entirely in the middle row of the joined diagrams; and
multiplying by a factor of $\para$ for each such middle-row component.  
For example,  if
\begin{eqnarray*} d_{\pi_1} &=& {\begin{array}{c}
{\begin{tikzpicture}[scale=.7,line width=1.25pt] 
\foreach \i in {1,...,7} 
{ \path (\i,1) coordinate (T\i); \path (\i,0) coordinate (B\i); } 
\filldraw[fill= gray!40,draw=gray!40,line width=3.2pt]  (T1) -- (T7) -- (B7) -- (B1) -- (T1);
\draw[blue] (T1) .. controls +(.1,-.30) and +(-.1,-.30) .. (T2);
\draw[blue] (T2) .. controls +(.1,-.30) and +(-.1,-.30) .. (T3);
\draw[blue] (T1) .. controls +(0,-.30) and  +(0,.30) .. (B1);
\draw[blue] (T4) .. controls +(.1,-.30) and +(-.1,-.30) .. (T5) ;
\draw[blue] (T4) .. controls +(0,-.30) and +(0,.30) .. (B5) ;
\draw[blue] (T5) .. controls +(.1,-.38) and +(-.1,-.38) .. (T7) ;
\draw[blue] (B2) .. controls +(.1,.30) and +(-.1,.30) .. (B3) ;
\draw[blue] (B4) .. controls +(.1,.62) and +(-.1,.62) .. (B7) ;
\draw[blue] (B5) .. controls +(.1,.30) and +(-.1,.30) .. (B6) ;
\foreach \i in {1,...,7}  { \fill (T\i) circle (3pt); \fill (B\i) circle (3pt); } 
\end{tikzpicture}}\end{array}} 
\\
d_{\pi_2} &=& 
{\begin{array}{c}{\begin{tikzpicture}[scale=.7,line width=1.25pt] 
\foreach \i in {1,...,7} { \path (\i,1) coordinate (T\i); \path (\i,0) coordinate (B\i); } 
\filldraw[fill=gray!40,draw=gray!40,line width=3.2pt]  (T1) -- (T7) -- (B7) -- (B1) -- (T1);
\draw[blue] (T2) .. controls +(.1,-.30) and +(-.1,-.30) .. (T3);
\draw[blue] (T1) .. controls +(0,-.30) and  +(0,.30) .. (B2);
\draw[blue] (T4) .. controls +(.1,-.55) and +(.1,-.55) .. (T7) ;
\draw[blue] (T6) .. controls +(.1,-.30) and +(-.1,+.30) .. (B7) ;
\draw[blue] (B2) .. controls +(.1,.30) and +(-.1,.30) .. (B3) ;
\draw[blue] (B3) .. controls +(.1,.40) and +(-.1,.40) .. (B5) ;
\draw[blue] (B5) .. controls +(.1,.30) and +(-.1,.30) .. (B6) ;
\foreach \i in {1,...,7}  { \fill (T\i) circle (3pt); \fill (B\i) circle (3pt); } 
\end{tikzpicture}} \end{array}} 
\end{eqnarray*}  
then
\begin{eqnarray}\label{DiagramMultiplication} \hspace{1.5cm} d_{\pi_1} d_{\pi_2} &=&\para^2 
{\begin{array}{c}  
{\begin{tikzpicture}[scale=.7,line width=1.25pt] 
\foreach \i in {1,...,7} 
{ \path (\i,1) coordinate (T\i); \path (\i,0) coordinate (B\i); } 
\filldraw[fill= gray!40,draw=gray!40,line width=3.2pt]  (T1) -- (T7) -- (B7) -- (B1) -- (T1);
\draw[blue] (T1) .. controls +(0,-.30) and +(0,-.30) .. (T2);
\draw[blue] (T1) .. controls +(.1,-.30) and +(-.1,+.30) .. (B2);
\draw[blue] (T2) .. controls +(.1,-.30) and +(-.1,-.30) .. (T3);
\draw[blue] (T4) .. controls +(.1,-.30) and +(-.1,-.30) .. (T5) ;
\draw[blue] (T5) .. controls +(.1,-.3) and +(.1,-.3) .. (T7) ;
\draw[blue] (T7) .. controls +(0,-.30) and +(0,+.30) .. (B7) ;
\draw[blue] (B2) .. controls +(.1,.30) and +(-.1,.30) .. (B3) ;
\draw[blue] (B3) .. controls +(.1,.40) and +(-.1,.40) .. (B5) ;
\draw[blue] (B5) .. controls +(.1,.30) and +(-.1,.30) .. (B6) ;
\foreach \i in {1,...,7} { \fill (T\i) circle (3pt); \fill (B\i) circle (3pt); } 
\end{tikzpicture}}
 \end{array} =  \para^2 d_{\pi_1 \ast \pi_2},} \end{eqnarray} 
where $\pi_1 \ast \pi_2$ is the set partition obtained by the concatenation of $\pi_1$ and $\pi_2$ in this process.   It is easy to confirm that the product depends only on the underlying set partition and is independent of the diagram chosen to represent $\pi$.
For any two set partitions $\pi_1,\pi_2 \in \Pi_{2k}$, we let $[\pi_1 \ast \pi_2]$ denote the number of blocks removed from the middle of the product $d_{\pi_1}d_{\pi_2}$, so that the product is given by 
\begin{equation}\label{eq:diagmult} d_{\pi_1} d_{\pi_2} = \para^{[\pi_1 \ast \pi_2]} d_{\pi_1 \ast \pi_2}. \end{equation} 

For  $\para \in \CC\setminus\{0\}$ and $k \in\ZZ_{\ge 1}$,  define the partition algebra $\P_k(\para)$ to be the $\CC$-span of $\{d_{\pi} \mid \pi \in \Pi_{2k}\}$ under the diagram multiplication in \eqref{eq:diagmult}.
We refer to  $\{d_{\pi} \mid \pi \in \Pi_{2k}\}$ as the \emph{diagram basis}.  Diagram multiplication is easily seen to be associative with identity 
element  $\mathsf{I}_k$ corresponding to the set partition,  $\big\{ 1,k+1 \mid 2, k+2 \mid \cdots \mid k,2k\big\}$,
where
\begin{equation}\label{id-def}
\mathsf{I}_k   =  
\begin{array}{c}\begin{tikzpicture}[scale=.7,line width=1.25pt] 
\foreach \i in {1,...,8} 
{ \path (\i,1) coordinate (T\i); \path (\i,0) coordinate (B\i); } 
\filldraw[fill= gray!40,draw=gray!40,line width=3.2pt]  (T1) -- (T8) -- (B8) -- (B1) -- (T1);
\draw[blue] (T1) -- (B1);\draw[blue] (T2) -- (B2);\draw[blue] (T3) -- (B3);\draw[blue] (T4) -- (B4);
\draw[blue] (T6) -- (B6);
\draw[blue] (T7) -- (B7);
\draw[blue] (T8) -- (B8);
\foreach \i in {1,2,3,4,6,7,8} { \fill (T\i) circle (3pt); \fill (B\i) circle (3pt); } 
\draw (T5) node  {$\qquad \cdots \qquad $ }; \draw (B5) node  { $\qquad \cdots\qquad $ }; 
\end{tikzpicture}\end{array}.
\hskip1.15in
\end{equation}   

{For $k \in \ZZ_{\ge 1}$,  the partition algebra $\P_k(\xi)$  has a presentation by the generators} 
\begin{align}
\label{s-gen}
\mathfrak{s}_i &=  
\begin{array}{c}\begin{tikzpicture}[scale=.7,line width=1.25pt] 
\foreach \i in {1,...,8} 
{ \path (\i,1) coordinate (T\i); \path (\i,0) coordinate (B\i); } 
\filldraw[fill= gray!40,draw=gray!40,line width=3.2pt]  (T1) -- (T8) -- (B8) -- (B1) -- (T1);
\draw[blue] (T1) -- (B1);
\draw[blue] (T3) -- (B3);
\draw[blue] (T4) -- (B5);
\draw[blue] (T5) -- (B4);
\draw[blue] (T6) -- (B6);
\draw[blue] (T8) -- (B8);
\foreach \i in {1,3,4,5,6,8} { \fill (T\i) circle (3pt); \fill (B\i) circle (3pt); } 
\draw (T2) node  {$\cdots$}; \draw (B2) node  {$\cdots$}; \draw (T7) node  {$\cdots$}; \draw (B7) node  {$\cdots$}; 
\draw  (T4)  node[black,above=0.05cm]{$\scriptstyle{i}$};
\draw  (T5)  node[black,above=0.0cm]{$\scriptstyle{i+1}$};
\end{tikzpicture}\end{array} \qquad 1 \le i \le k-1,
\\
\label{p-gen}
\mathfrak{p}_i &=  \frac{1}{\para}
\begin{array}{c}\begin{tikzpicture}[scale=.7,line width=1.25pt] 
\foreach \i in {1,...,8} 
{ \path (\i,1) coordinate (T\i); \path (\i,0) coordinate (B\i); } 
\filldraw[fill= gray!40,draw=gray!40,line width=3.2pt]  (T1) -- (T8) -- (B8) -- (B1) -- (T1);
\draw[blue] (T1) -- (B1);
\draw[blue] (T3) -- (B3);
\draw[blue] (T5) -- (B5);
\draw[blue] (T6) -- (B6);
\draw[blue] (T8) -- (B8);
\foreach \i in {1,3,4,5,6,8} { \fill (T\i) circle (3pt); \fill (B\i) circle (3pt); } 
\draw (T2) node  {$\cdots$}; \draw (B2) node  {$\cdots$}; \draw (T7) node  {$\cdots$}; \draw (B7) node  {$\cdots$}; 
\draw  (T4)  node[black,above=0.05cm]{$\scriptstyle{i}$};
\end{tikzpicture}\end{array} \qquad 1 \le i \le k, \\
\label{b-gen}
\mathfrak{b}_i &=  
\begin{array}{c}\begin{tikzpicture}[scale=.7,line width=1.25pt] 
\foreach \i in {1,...,8} 
{ \path (\i,1) coordinate (T\i); \path (\i,0) coordinate (B\i); } 
\filldraw[fill= gray!40,draw=gray!40,line width=3.2pt]  (T1) -- (T8) -- (B8) -- (B1) -- (T1);
\draw[blue] (T1) -- (B1);
\draw[blue] (T3) -- (B3);
\draw[blue] (T4) -- (B4) -- (B5) -- (T5) -- (T4);
\draw[blue] (T6) -- (B6);
\draw[blue] (T8) -- (B8);
\foreach \i in {1,3,4,5,6,8} { \fill (T\i) circle (3pt); \fill (B\i) circle (3pt); } 
\draw (T2) node  {$\cdots$}; \draw (B2) node  {$\cdots$}; \draw (T7) node  {$\cdots$}; \draw (B7) node  {$\cdots$}; 
\draw  (T4)  node[black,above=0.05cm]{$\scriptstyle{i}$};
\draw  (T5)  node[black,above=0.0cm]{$\scriptstyle{i+1}$};
\end{tikzpicture}\end{array} \qquad 1 \le i \le k-1,
\end{align}
and the relations in the next result.   
\begin{thm}{\rm \cite[Thm.~1.11]{HR}}\label{T:present}  Assume $k \in \ZZ_{\ge 1}$,  and set $\mathfrak{p}_{i+\half} = \mathfrak{b}_i \ \  (1 \le i \le k-1)$.     Then $\P_k(\xi)$ has a presentation as a unital associative algebra by  
generators $\mathfrak{s}_i \ \ (1 \leq i \le k-1)$,  $\mathfrak{p}_\ell \ \  (\ell \in \half \ZZ_{\ge 1},  \  1 \le \ell \le k)$,  and the following relations:
{\rm
\begin{itemize}
\item[{\rm(a)}]  $ \mathfrak{s}_i^2 = \mathsf{I}_k, \quad   \mathfrak{s}_i \mathfrak{s}_j = \mathfrak{s}_j\mathfrak{s}_i \quad  (|i-j| > 1)$, \quad    $\mathfrak{s}_i \mathfrak{s}_{i+1} \mathfrak{s}_i =  \mathfrak{s}_{i+1}
 \mathfrak{s}_i \mathfrak{s}_{i+1} \quad (1 \leq i \leq k-2)$;
\item[{\rm(b)}]  $\mathfrak{p}_\ell^2 = \mathfrak{p}_\ell,  \quad   \mathfrak{p}_\ell \mathfrak{p}_m = 
 \mathfrak{p}_m \mathfrak{p}_\ell \quad (m \ne \ell\pm \half), \quad   \mathfrak{p}_\ell \mathfrak{p}_{\ell\pm \half}  \mathfrak{p}_\ell =\mathfrak{p}_\ell  \quad   ( \mathfrak{p}_\half := \mathsf{I}_k =: \mathfrak{p}_{k+\half})$;
 \item[{\rm(c)}]  $\mathfrak{s}_i \mathfrak{p}_i \mathfrak{p}_{i +1} =  \mathfrak{p}_i \mathfrak{p}_{i +1},  \   \quad  
  \mathfrak{s}_i \mathfrak{p}_i \mathfrak{s}_i =  \mathfrak{p}_{i+1},   \ \quad   \mathfrak{s}_i \mathfrak{p}_{i+\half}= 
 \mathfrak{p}_{i+\half} \mathfrak{s}_i = \mathfrak{p}_{i+\half}  \ \ \ \, (1 \leq i \leq k-1), \newline 
 \mathfrak{s}_i  \mathfrak{s}_{i+1}  \mathfrak{p}_{i+\half} \mathfrak{s}_{i+1}\mathfrak{s}_i = \mathfrak{p}_{i+\frac{3}{2}} \ \ \ (1 \le i \le k-2), \quad  \mathfrak{s}_i \mathfrak{p}_\ell = \mathfrak{p}_\ell \mathfrak{s}_i \quad  (\ell \neq i-\half, i,i+\half, i+1, i+\frac{3}{2})$.
 \end{itemize}
 }

 \end{thm}   
 
\begin{remark} It is easily seen from the relations that $\P_k(\para)$ is generated by $\mathfrak{s}_i\ (1 \le i \le k-1), \mathfrak{p}_1$,  and $\mathfrak{b}_1 = \mathfrak{p}_{1+\half}.$    It was observed by East \cite[Remark 37]{E} that the relation
$ \mathfrak{p}_{i}\mathfrak{p}_{i+1}\mathfrak{s_i} =  \mathfrak{p}_i \mathfrak{p}_{i +1}$, which
was included in \cite[Thm.~1.11]{HR}, can be derived from the others.    Several different presentations for $\P_k(n)$ can be found in \cite{E}.
\end{remark}

If $\pi_1, \pi_2 \in \Pi_{2k-1}$, so that $k$ and $2k$ are in the same block in both $\pi_1$ and $\pi_2$, then $k$ and $2k$ are also in the same block of $\pi_1 \ast \pi_2$. Thus, for $k \in \ZZ_{\ge 1}$, we define $\P_{k-\frac{1}{2}}(\para) \subset \P_{k}(\para)$ to be the $\CC$-span of $\{d_\pi \mid \pi \in \Pi_{2k-1} \subset \Pi_{2k}\}$. There is also an embedding  $\P_{k}(\para) \subset  \P_{k+\frac{1}{2}}(\para)$ given by adding a top and a bottom node to the right of any diagram in $\P_k(\para)$ and a vertical edge connecting them. Setting $\P_0(\para) = \CC$, we have a tower of embeddings
\begin{equation}
\P_0(\para)  \cong  \P_{\frac{1}{2}}(\para) \subset  \P_{1}(\para) \subset   \P_{1\frac{1}{2}}(\para) \subset  \P_{2}(\para)\subset  \P_{2\frac{1}{2}}(\para) \subset  \cdots
\end{equation}
with $\dimm \P_k(\para) = |\Pi_{2k}| =  \mathsf{B}(2k)$ (the $2k$-th Bell number) for each $k \in \frac{1}{2}\ZZ_{\ge 1}$.

For $\pi \in \Pi_{2k}$, the  {\it propagating number} $\pn(\pi)$ is  the number of blocks of $\pi$ which intersect both the bottom row $\{1,2, \ldots, k\}$ and the top row $\{k+1, \ldots, 2k\}$. It is straightforward to verify that
\begin{equation}\label{eq:pn}
\pn(\pi_1 \ast \pi_2) \le \min(\pn(\pi_1) , \pn(\pi_2)),
\end{equation}
and thus, for each $0 \le \ell \le k$, $J_\ell :=  \mathsf{span}_\CC\{ d_\pi \mid \pi \in \Pi_{2k}, \pn(\pi) \le \ell\}$ is a two-sided ideal of $\P_k(\para)$. These ideals are important for the Jones basic construction of $\P_k(\para)$ (see, for example, \cite{HR}). 
 
\subsection{The orbit basis}

For $k \in \ZZ_{\ge 1}$, the set partitions  $\Pi_{2k}$ of $[1,2k]$ form a  lattice (a partially ordered set (poset) for which each pair has a least upper bound and a greatest lower bound) under the partial order given by
\begin{equation}
\pi \preceq \vr \quad \hbox{ if every block of $\pi$ is contained in a block of $\vr$.}
\end{equation}
In this case we say that $\pi$ is a {\it refinement} of $\vr$, and that $\vr$ is a {\it coarsening} of $\pi$, so that $\Pi_{2k}$ is partially ordered by refinement.  

 For $k \in \half \ZZ_{\ge 1}$, there is  a second basis $\left\{x_\pi  \mid \pi \in \Pi_{2k} \right\}$ of $\P_k(\para)$, called the \emph{orbit basis}, that is 
defined by the following coarsening relation with respect to the diagram basis:
\begin{equation}\label{refinement-relation}
d_\pi = \sum_{\pi \preceq \vr} x_\vr.
\end{equation}
Thus, the diagram basis element $d_\pi$ is the sum of all orbit basis elements $x_\vr$ for which $\vr$ is  coarser than $\pi$. 
For the remainder of the paper,  we adopt the following convention: 
\begin{equation}
\begin{array}{l} \hbox{\emph{Diagrams with white vertices indicate orbit basis elements, and}} \\
\hbox{\emph{those with black vertices indicate diagram basis elements.}} 
\end{array}
\end{equation}
For example, the expression below writes the diagram $d_{14|2|3}$ in $\P_2(\xi)$ in terms of the orbit basis,
\begin{equation*}
\begin{array}{c}
\begin{tikzpicture}[xscale=.5,yscale=.5,line width=1.25pt] 
\foreach \i in {1,2}  { \path (\i,1.25) coordinate (T\i); \path (\i,.25) coordinate (B\i); } 
\filldraw[fill= black!12,draw=black!12,line width=4pt]  (T1) -- (T2) -- (B2) -- (B1) -- (T1);
\draw[blue] (B1) -- (T2);
\foreach \i in {1,2}  { \fill (T\i) circle (4pt); \fill  (B\i) circle (4pt); } 
\end{tikzpicture}
\end{array}
=
\begin{array}{c}
\begin{tikzpicture}[xscale=.5,yscale=.5,line width=1.25pt] 
\foreach \i in {1,2}  { \path (\i,1.25) coordinate (T\i); \path (\i,.25) coordinate (B\i); } 
\filldraw[fill= black!12,draw=black!12,line width=3pt]  (T1) -- (T2) -- (B2) -- (B1) -- (T1);
\draw[blue] (B1) -- (T2);
\foreach \i in {1,2}  { \filldraw[fill=white,draw=black,line width = 1pt] (T\i) circle (4pt); \filldraw[fill=white,draw=black,line width = 1pt]  (B\i) circle (4pt); } 
\end{tikzpicture}
\end{array}
+\begin{array}{c}
\begin{tikzpicture}[xscale=.5,yscale=.5,line width=1.25pt] 
\foreach \i in {1,2}  { \path (\i,1.25) coordinate (T\i); \path (\i,.25) coordinate (B\i); } 
\filldraw[fill= black!12,draw=black!12,line width=3pt]  (T1) -- (T2) -- (B2) -- (B1) -- (T1);
\draw[blue] (B1) -- (T2);
\draw[blue] (B1) -- (B2);
\draw[blue] (T2) -- (B2);
\foreach \i in {1,2}  { \filldraw[fill=white,draw=black,line width = 1pt] (T\i) circle (4pt); \filldraw[fill=white,draw=black,line width = 1pt]  (B\i) circle (4pt); } 
\end{tikzpicture}
\end{array}
+\begin{array}{c}
\begin{tikzpicture}[xscale=.5,yscale=.5,line width=1.25pt] 
\foreach \i in {1,2}  { \path (\i,1.25) coordinate (T\i); \path (\i,.25) coordinate (B\i); } 
\filldraw[fill= black!12,draw=black!12,line width=3pt]  (T1) -- (T2) -- (B2) -- (B1) -- (T1);
\draw[blue] (B1) -- (T2);
\draw[blue] (T1) -- (T2);
\draw[blue] (T1) -- (B1);
\foreach \i in {1,2}  { \filldraw[fill=white,draw=black,line width = 1pt] (T\i) circle (4pt); \filldraw[fill=white,draw=black,line width = 1pt]  (B\i) circle (4pt); } 
\end{tikzpicture}
\end{array}
+\begin{array}{c}
\begin{tikzpicture}[xscale=.5,yscale=.5,line width=1.25pt] 
\foreach \i in {1,2}  { \path (\i,1.25) coordinate (T\i); \path (\i,.25) coordinate (B\i); } 
\filldraw[fill= black!12,draw=black!12,line width=3pt]  (T1) -- (T2) -- (B2) -- (B1) -- (T1);
\draw[blue] (B1) -- (T2);
\draw[blue] (T1) -- (B2);
\foreach \i in {1,2}  { \filldraw[fill=white,draw=black,line width = 1pt] (T\i) circle (4pt); \filldraw[fill=white,draw=black,line width = 1pt]  (B\i) circle (4pt); } 
\end{tikzpicture}
\end{array}
+\begin{array}{c}
\begin{tikzpicture}[xscale=.5,yscale=.5,line width=1.25pt] 
\foreach \i in {1,2}  { \path (\i,1.25) coordinate (T\i); \path (\i,.25) coordinate (B\i); } 
\filldraw[fill= black!12,draw=black!12,line width=3pt]  (T1) -- (T2) -- (B2) -- (B1) -- (T1);
\draw[blue] (B1) -- (T2);
\draw[blue] (T1) -- (T2);\draw[blue] (B1) -- (B2);
\draw[blue] (T1) -- (B1);
\draw[blue] (T2) -- (B2);
\draw[blue] (T1) -- (B2);
\foreach \i in {1,2}  { \filldraw[fill=white,draw=black,line width = 1pt] (T\i) circle (4pt); \filldraw[fill=white,draw=black,line width = 1pt]  (B\i) circle (4pt); } 
\end{tikzpicture}
\end{array}.
\end{equation*}

We define the propagating number of the diagrams  $d_\pi$ and $x_{\pi}$ to be 
the propagating number of the set partition $\pi$,  so  $\pn(d_\pi)= \pn(\pi) = \pn(x_\pi)$. 

\begin{remark}
We refer to the basis $\{x_\pi  \mid \pi \in \Pi_{2k}\}$ as the orbit basis,  because when $\para = n$, the elements in this basis act on the tensor space $\modu^{\ot k}$ in a natural way that corresponds to $\S_n$-orbits on simple tensors (see \eqref{eq:Phi-coeffs-orbit}). In fact, in Jones' original definition of the partition algebra \cite{J}, the orbit basis appears first, and the diagram basis is defined later using the refinement relation \eqref{refinement-relation}.   However, as we see in Section \ref{sec:Mult}, multiplication is simpler and more natural in the diagram basis, and for this reason the diagram basis is most commonly used when working with $\P_k(n)$.
\end{remark}

\subsection{Change of basis}\label{S:change}   
  
The transition matrix between the diagram basis  and the orbit basis determined by \eqref{refinement-relation}  is the matrix $\zeta_{2k}$, called the \emph{zeta matrix} of the poset $\Pi_{2k}$. It is unitriangular with respect to any extension to a linear order, and thus it is invertible, confirming that indeed the elements $x_{\pi}, \pi \in \Pi_{2k}$,  form a basis of $\P_k(\para)$. 
The inverse of $\zeta_{2k}$ is the matrix $\mu_{2k}$ of the M\"obius function of the set partition lattice, and it satisfies
\begin{equation}\label{eq:mobiusa}
x_\pi = \sum_{\pi \preceq \vr} \mu_{2k}(\pi,\vr) d_\vr,
\end{equation}
where $\mu_{2k}(\pi,\vr)$ is the $(\pi,\vr)$ entry of $\mu_{2k}$. The M\"obius function of the set partition lattice has an easily computed formula.  If  $\pi \preceq \vr$,  and $\vr$ consists of $\ell$ blocks such that the $i$th block of $\vr$ is the union of $b_i$ blocks  of $\pi$, then (see, for example, \cite[p.~30]{St}), 
\begin{equation}\label{eq:mobiusb}
 \mu_{2k}(\pi,\vr) = \prod_{i=1}^\ell (-1)^{b_i-1}(b_i-1)!.
\end{equation}
 
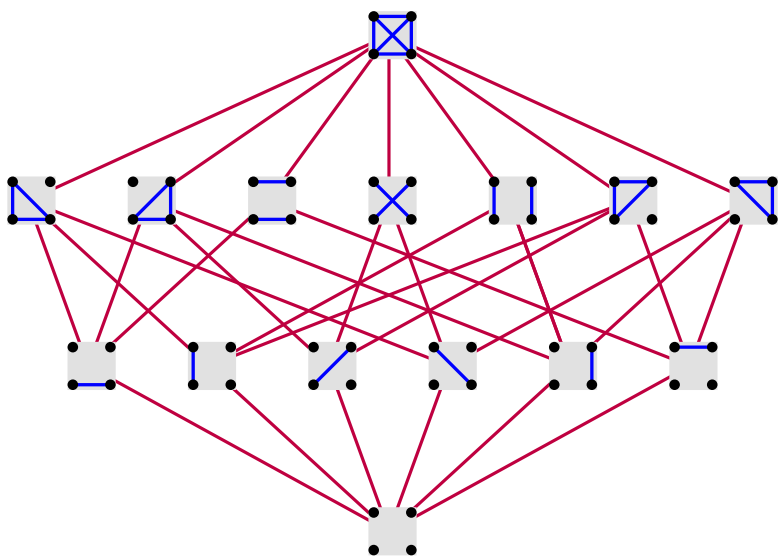
\begin{figure}[h!]
$$
  \begin{tikzpicture}[xscale=.08,yscale=.22,line width=1.25pt] 
  nodes={draw=white, 
      minimum width=.2cm, minimum height=.5cm
    }

\path (70,30) coordinate (m-1234); 

\path (10,20) coordinate (m-123-4); 
\path (30,20) coordinate (m-124-3); 
\path (50,20) coordinate (m-12-34); 
\path (70,20) coordinate (m-14-23); 
\path (90,20) coordinate (m-13-24); 
\path (110,20) coordinate (m-134-2); 
\path (130,20) coordinate (m-1-234); 

\path (20,10) coordinate (m-12-3-4); 
\path (40,10) coordinate (m-13-2-4); 
\path (60,10) coordinate (m-14-2-3); 
\path (80,10) coordinate (m-23-1-4); 
\path (100,10) coordinate (m-1-3-24); 
\path (120,10) coordinate (m-1-2-34); 

\path (70,0) coordinate (m-1-2-3-4);

\draw (m-1234) node  {\abcd}; 

\draw (m-123-4) node  {\abcld}; 
\draw (m-124-3) node  {\abdlc}; 
\draw (m-12-34) node  {\ablcd}; 
\draw (m-134-2) node  {\acdlb}; 
\draw (m-13-24) node  {\aclbd}; 
\draw (m-14-23) node  {\adlbc}; 
\draw (m-1-234) node  {\albcd}; 

\draw (m-12-3-4) node  {\ablcld}; 
\draw (m-13-2-4) node  {\aclbld}; 
\draw (m-14-2-3) node  {\adlblc}; 
\draw (m-23-1-4) node  {\albcld}; 
\draw (m-1-3-24) node  {\albdlc}; 
\draw (m-1-2-34) node  {\alblcd};

\draw (m-1-2-3-4) node  {\alblcld}; 

\begin{scope}[on background layer]
 
\path  (m-1234) edge[purple,line width=1.2pt] (m-123-4);
\path  (m-1234) edge[purple,line width=1.2pt] (m-124-3);
\path  (m-1234) edge[purple,line width=1.2pt] (m-12-34);
\path  (m-1234) edge[purple,line width=1.2pt] (m-134-2);
\path  (m-1234) edge[purple,line width=1.2pt] (m-13-24);
\path  (m-1234) edge[purple,line width=1.2pt] (m-14-23);
\path  (m-1234) edge[purple,line width=1.2pt] (m-1-234);

\path  (m-12-3-4) edge[purple,line width=1.2pt] (m-123-4);
\path  (m-13-2-4) edge[purple,line width=1.2pt] (m-123-4);
\path  (m-23-1-4) edge[purple,line width=1.2pt] (m-123-4);
\path  (m-12-3-4) edge[purple,line width=1.2pt] (m-124-3);
\path  (m-14-2-3) edge[purple,line width=1.2pt] (m-124-3);
\path  (m-1-3-24) edge[purple,line width=1.2pt] (m-124-3);
\path  (m-12-3-4) edge[purple,line width=1.2pt] (m-12-34);
\path  (m-1-2-34) edge[purple,line width=1.2pt] (m-12-34);
\path  (m-13-2-4) edge[purple,line width=1.2pt] (m-134-2);
\path  (m-14-2-3) edge[purple,line width=1.2pt] (m-134-2);
\path  (m-1-2-34) edge[purple,line width=1.2pt] (m-134-2);
\path  (m-13-2-4) edge[purple,line width=1.2pt] (m-13-24);
\path  (m-1-3-24) edge[purple,line width=1.2pt] (m-13-24);
\path  (m-1-3-24) edge[purple,line width=1.2pt] (m-13-24);
\path  (m-14-2-3) edge[purple,line width=1.2pt] (m-14-23);
\path  (m-23-1-4) edge[purple,line width=1.2pt] (m-14-23);
\path  (m-1-2-34) edge[purple,line width=1.2pt] (m-1-234);
\path  (m-1-3-24) edge[purple,line width=1.2pt] (m-1-234);
\path  (m-23-1-4) edge[purple,line width=1.2pt] (m-1-234);

\path  (m-12-3-4) edge[purple,line width=1.2pt] (m-1-2-3-4);
\path  (m-13-2-4) edge[purple,line width=1.2pt] (m-1-2-3-4);
\path  (m-14-2-3) edge[purple,line width=1.2pt] (m-1-2-3-4);
\path  (m-23-1-4) edge[purple,line width=1.2pt] (m-1-2-3-4);
\path  (m-1-3-24) edge[purple,line width=1.2pt] (m-1-2-3-4);
\path  (m-1-2-34) edge[purple,line width=1.2pt] (m-1-2-3-4);
    \end{scope}
  \end{tikzpicture}
$$  
\caption{Hasse diagram of the partition lattice $\Pi_4$ in the refinement ordering.\label{fig:Hasse}}
\end{figure}

The Hasse diagram of the partition lattice of $\Pi_4$ is shown in Figure \ref{fig:Hasse}. In the change of basis between the orbit basis and the diagram basis (in either direction), each basis element is an integer linear combination of the basis elements above or equal to it in the Hasse diagram. 
If we apply  formula \eqref{eq:mobiusa} to express $x_{1 \,|\,  2\,|\,  3\,|\,  4}$ and $x_{14\,|\,  2\,|\, 3}$ in the orbit basis in terms of the diagram basis in $\P_2(\para)$, we get

\begin{align}\begin{split}\label{eq:orbit2diag}
\begin{array}{c}\orbitalblcld\end{array} =&  \begin{array}{c}\alblcld\end{array}  - \begin{array}{c}\ablcld\end{array} - \begin{array}{c}\aclbld\end{array} - 
 \begin{array}{c}\adlblc\end{array} - \begin{array}{c}\albcld\end{array} - \begin{array}{c}\albdlc\end{array} - \begin{array}{c}\alblcd\end{array}
 +2 \begin{array}{c}\abcld\end{array} \\ &+2 \begin{array}{c}\abdlc\end{array} + \begin{array}{c}\ablcd\end{array}+ \begin{array}{c}\adlbc\end{array}
+ \begin{array}{c}\aclbd\end{array} + 2 \begin{array}{c}\acdlb\end{array} +2 \begin{array}{c}\albcd\end{array}  - 6  \begin{array}{c}\abcd\end{array}, \\
\begin{array}{c}\orbitadlblc\end{array} = & \begin{array}{c}\adlblc\end{array} - \begin{array}{c}\adlbc\end{array} - \begin{array}{c}\acdlb\end{array} - \begin{array}{c}\abdlc\end{array}  + 2  \begin{array}{c}\abcd \end{array}. 
\end{split} \end{align}  

\begin{remark}\label{R:Hasse}
If $\pi \in \Pi_{2k-1} \subset \Pi_{2k}$, then $k$ and $2k$ are in the same block of $\pi$. When expressing $x_\pi$ in terms of the diagram basis or $d_\pi$ in terms of the orbit basis, the sum is over coarsenings of $\pi$, so the expression will always involve only set partitions also in $\Pi_{2k-1}$. Thus, the expressions in \eqref{refinement-relation} and \eqref{eq:mobiusa} apply equally well to the algebras $\P_{k-\frac{1}{2}}(\para)$. The corresponding Hasse diagram is given by the sublattice  of partitions greater than or equal to $\{1\,|\,2\,| \cdots |\, k-1\,|\, k+1\,|\, k+2\,| \cdots |\, 2k-1\,|\, k, 2k \}$.  For example, the Hasse diagram for $\Pi_3$ is found inside that of $\Pi_4$ in Figure  \ref{fig:Hasse} as those partitions greater than or equal to $\begin{array}{c}\begin{tikzpicture}[scale=.4,line width=1.25pt] 
\foreach \i in {1,2}  { \path (\i,1) coordinate (T\i); \path (\i,0) coordinate (B\i); } 
\filldraw[fill= gray!40,draw=gray!40,line width=3.2pt]  (T1) -- (T2) -- (B2) -- (B1) -- (T1);
\draw[blue] (T2) -- (B2);
\foreach \i in {1,2} { \fill (T\i) circle (4pt); \fill (B\i) circle (4pt); } 
\end{tikzpicture}\end{array}$.  
\end{remark} 

\begin{remark}\label{R:orbid} The orbit diagram 
$\mathsf{I}_k^{\boldsymbol \circ} :=\begin{array}{c} 
\begin{tikzpicture}[xscale=.35,yscale=.35,line width=1.0pt] 
\foreach \i in {1,2,3,4}  { \path (\i,1.25) coordinate (T\i); \path (\i,.25) coordinate (B\i); } 
\filldraw[fill= black!12,draw=black!12,line width=4pt]  (T1) -- (T4) -- (B4) -- (B1) -- (T1);
\draw[blue] (T1) -- (B1);
\draw[blue] (T2) -- (B2);
\draw[blue] (T4) -- (B4);
\draw (T3) node {$\cdots$};\draw (B3) node {$\cdots$};
\foreach \i in {1,2,4}  { \filldraw[fill=white,draw=black,line width = 1pt] (T\i) circle (4pt); \filldraw[fill=white,draw=black,line width = 1pt]  (B\i) circle (4pt); } 
\end{tikzpicture} 
\end{array}$ is \emph{not} the identity element  in $\P_k(n)$. To get the identity element $\mathsf{I}_k$,  we must add all coarsenings to $\mathsf{I}_k^{\boldsymbol \circ}$,
as in the first line below:
\begin{align*}
\mathsf{I}_3 &= \begin{array}{c}
\begin{tikzpicture}[xscale=.45,yscale=.45,line width=1.0pt] 
\foreach \i in {1,2,3}  { \path (\i,1.25) coordinate (T\i); \path (\i,.25) coordinate (B\i); } 
\filldraw[fill= black!12,draw=black!12,line width=4pt]  (T1) -- (T3) -- (B3) -- (B1) -- (T1);
\draw[blue] (T1) -- (B1);
\draw[blue] (T2) -- (B2);
\draw[blue] (T3) -- (B3);
\foreach \i in {1,2,3}  { \filldraw[fill=black,draw=black,line width = 1pt] (T\i) circle (4pt); \filldraw[fill=black,draw=black,line width = 1pt]  (B\i) circle (4pt); } 
\end{tikzpicture}\end{array}
=
\begin{array}{c}\begin{tikzpicture}[xscale=.45,yscale=.45,line width=1.0pt] 
\foreach \i in {1,2,3}  { \path (\i,1.25) coordinate (T\i); \path (\i,.25) coordinate (B\i); } 
\filldraw[fill= black!12,draw=black!12,line width=4pt]  (T1) -- (T3) -- (B3) -- (B1) -- (T1);
\draw[blue] (T1) -- (B1);
\draw[blue] (T2) -- (B2);
\draw[blue] (T3) -- (B3);
\foreach \i in {1,2,3}  { \filldraw[fill=white,draw=black,line width = 1pt] (T\i) circle (4pt); \filldraw[fill=white,draw=black,line width = 1pt]  (B\i) circle (4pt); } 
\end{tikzpicture}\end{array}
+
\begin{array}{c}\begin{tikzpicture}[xscale=.45,yscale=.45,line width=1.0pt] 
\foreach \i in {1,2,3}  { \path (\i,1.25) coordinate (T\i); \path (\i,.25) coordinate (B\i); } 
\filldraw[fill= black!12,draw=black!12,line width=4pt]  (T1) -- (T3) -- (B3) -- (B1) -- (T1);
\draw[blue] (T1) -- (B1)--(B2);
\draw[blue] (T1)--(T2) -- (B2);
\draw[blue] (T3) -- (B3);
\foreach \i in {1,2,3}  { \filldraw[fill=white,draw=black,line width = 1pt] (T\i) circle (4pt); \filldraw[fill=white,draw=black,line width = 1pt]  (B\i) circle (4pt); } 
\end{tikzpicture}\end{array}
+
\begin{array}{c}\begin{tikzpicture}[xscale=.45,yscale=.45,line width=1.0pt] 
\foreach \i in {1,2,3}  { \path (\i,1.25) coordinate (T\i); \path (\i,.25) coordinate (B\i); } 
\filldraw[fill= black!12,draw=black!12,line width=4pt]  (T1) -- (T3) -- (B3) -- (B1) -- (T1);
\draw[blue] (T1) -- (B1);
\draw[blue] (T2) -- (B2);
\draw[blue] (T3) -- (B3);
\draw[blue] (T1) .. controls +(0,-.50) and +(0,-.50) .. (T3);
\draw[blue] (B1) .. controls +(0,.50) and +(0,.50) .. (B3);
\foreach \i in {1,2,3}  { \filldraw[fill=white,draw=black,line width = 1pt] (T\i) circle (4pt); \filldraw[fill=white,draw=black,line width = 1pt]  (B\i) circle (4pt); } 
\end{tikzpicture}\end{array}
+
\begin{array}{c}\begin{tikzpicture}[xscale=.45,yscale=.45,line width=1.0pt] 
\foreach \i in {1,2,3}  { \path (\i,1.25) coordinate (T\i); \path (\i,.25) coordinate (B\i); } 
\filldraw[fill= black!12,draw=black!12,line width=4pt]  (T1) -- (T3) -- (B3) -- (B1) -- (T1);
\draw[blue] (T1) -- (B1);
\draw[blue] (T2) -- (B2)--(B3);
\draw[blue] (T2)--(T3) -- (B3);
\foreach \i in {1,2,3}  { \filldraw[fill=white,draw=black,line width = 1pt] (T\i) circle (4pt); \filldraw[fill=white,draw=black,line width = 1pt]  (B\i) circle (4pt); } 
\end{tikzpicture}\end{array}
+
\begin{array}{c}\begin{tikzpicture}[xscale=.45,yscale=.45,line width=1.0pt] 
\foreach \i in {1,2,3}  { \path (\i,1.25) coordinate (T\i); \path (\i,.25) coordinate (B\i); } 
\filldraw[fill= black!12,draw=black!12,line width=4pt]  (T1) -- (T3) -- (B3) -- (B1) -- (T1);
\draw[blue] (T1) -- (B1) -- (B2);
\draw[blue] (T1)--(T2) -- (B2) -- (B3);
\draw[blue] (T2)--(T3) -- (B3);
\foreach \i in {1,2,3}  { \filldraw[fill=white,draw=black,line width = 1pt] (T\i) circle (4pt); \filldraw[fill=white,draw=black,line width = 1pt]  (B\i) circle (4pt); } 
\end{tikzpicture}\end{array}  \qquad \hbox{and} \\
\mathsf{I}_3^{\boldsymbol \circ}&=
\begin{array}{c}\begin{tikzpicture}[xscale=.45,yscale=.45,line width=1.0pt] 
\foreach \i in {1,2,3}  { \path (\i,1.25) coordinate (T\i); \path (\i,.25) coordinate (B\i); } 
\filldraw[fill= black!12,draw=black!12,line width=4pt]  (T1) -- (T3) -- (B3) -- (B1) -- (T1);
\draw[blue] (T1) -- (B1);
\draw[blue] (T2) -- (B2);
\draw[blue] (T3) -- (B3);
\foreach \i in {1,2,3}  { \filldraw[fill=white,draw=black,line width = 1pt] (T\i) circle (4pt); \filldraw[fill=white,draw=black,line width = 1pt]  (B\i) circle (4pt); } 
\end{tikzpicture}\end{array}
=
\begin{array}{c}\begin{tikzpicture}[xscale=.45,yscale=.45,line width=1.0pt] 
\foreach \i in {1,2,3}  { \path (\i,1.25) coordinate (T\i); \path (\i,.25) coordinate (B\i); } 
\filldraw[fill= black!12,draw=black!12,line width=4pt]  (T1) -- (T3) -- (B3) -- (B1) -- (T1);
\draw[blue] (T1) -- (B1);
\draw[blue] (T2) -- (B2);
\draw[blue] (T3) -- (B3);
\foreach \i in {1,2,3}  { \filldraw[fill=black,draw=black,line width = 1pt] (T\i) circle (4pt); \filldraw[fill=black,draw=black,line width = 1pt]  (B\i) circle (4pt); } 
\end{tikzpicture}\end{array}
-
\begin{array}{c}\begin{tikzpicture}[xscale=.45,yscale=.45,line width=1.0pt] 
\foreach \i in {1,2,3}  { \path (\i,1.25) coordinate (T\i); \path (\i,.25) coordinate (B\i); } 
\filldraw[fill= black!12,draw=black!12,line width=4pt]  (T1) -- (T3) -- (B3) -- (B1) -- (T1);
\draw[blue] (T1) -- (B1)--(B2);
\draw[blue] (T1)--(T2) -- (B2);
\draw[blue] (T3) -- (B3);
\foreach \i in {1,2,3}  { \filldraw[fill=black,draw=black,line width = 1pt] (T\i) circle (4pt); \filldraw[fill=black,draw=black,line width = 1pt]  (B\i) circle (4pt); } 
\end{tikzpicture}\end{array}
-
\begin{array}{c}\begin{tikzpicture}[xscale=.45,yscale=.45,line width=1.0pt] 
\foreach \i in {1,2,3}  { \path (\i,1.25) coordinate (T\i); \path (\i,.25) coordinate (B\i); } 
\filldraw[fill= black!12,draw=black!12,line width=4pt]  (T1) -- (T3) -- (B3) -- (B1) -- (T1);
\draw[blue] (T1) -- (B1);
\draw[blue] (T2) -- (B2);
\draw[blue] (T3) -- (B3);
\draw[blue] (T1) .. controls +(0,-.50) and +(0,-.50) .. (T3);
\draw[blue] (B1) .. controls +(0,.50) and +(0,.50) .. (B3);
\foreach \i in {1,2,3}  { \filldraw[fill=black,draw=black,line width = 1pt] (T\i) circle (4pt); \filldraw[fill=black,draw=black,line width = 1pt]  (B\i) circle (4pt); } 
\end{tikzpicture}\end{array}
-
\begin{array}{c}\begin{tikzpicture}[xscale=.45,yscale=.45,line width=1.0pt] 
\foreach \i in {1,2,3}  { \path (\i,1.25) coordinate (T\i); \path (\i,.25) coordinate (B\i); } 
\filldraw[fill= black!12,draw=black!12,line width=4pt]  (T1) -- (T3) -- (B3) -- (B1) -- (T1);
\draw[blue] (T1) -- (B1);
\draw[blue] (T2) -- (B2)--(B3);
\draw[blue] (T2)--(T3) -- (B3);
\foreach \i in {1,2,3}  { \filldraw[fill=black,draw=black,line width = 1pt] (T\i) circle (4pt); \filldraw[fill=black,draw=black,line width = 1pt]  (B\i) circle (4pt); } 
\end{tikzpicture}\end{array}
+2
\begin{array}{c}\begin{tikzpicture}[xscale=.45,yscale=.45,line width=1.0pt] 
\foreach \i in {1,2,3}  { \path (\i,1.25) coordinate (T\i); \path (\i,.25) coordinate (B\i); } 
\filldraw[fill= black!12,draw=black!12,line width=4pt]  (T1) -- (T3) -- (B3) -- (B1) -- (T1);
\draw[blue] (T1) -- (B1) -- (B2);
\draw[blue] (T1)--(T2) -- (B2) -- (B3);
\draw[blue] (T2)--(T3) -- (B3);
\foreach \i in {1,2,3}  { \filldraw[fill=black,draw=black,line width = 1pt] (T\i) circle (4pt); \filldraw[fill=black,draw=black,line width = 1pt]  (B\i) circle (4pt); } 
\end{tikzpicture}\end{array}.
\end{align*}
\end{remark}  
\section{Representation of $\P_k(n)$ on the Tensor Space $\modu^{\ot k}$}    

\subsection{Schur-Weyl duality}\label{sec:SWduality}

Assume $k,n \in \ZZ_{\ge 1}$.  Let $\{\us_j \mid  1\le j \le n\}$ be a basis for the permutation module $\modu$ of $\S_n$ so that
$\sigma.\us_j  = \us_{\sigma(j)}$ for all $\sigma \in \S_n$ and all $1 \le j \le n$.   For  $\rs = (r_1,\dots, r_k) \in [1,n]^k = \{1,2,\dots, n\}^k$, define $\us_\rs = \us_{r_1} \ot \cdots \ot \us_{r_k}$.  The elements $\us_\rs$ form a basis for the $\S_n$-module $\modu^{\ot k}$, where the $\S_n$-action is
given by the diagonal action  $\sigma. \us_\rs = \us_{\sigma(\rs)} := \us_{\sigma(r_1)} \ot \cdots \ot \us_{\sigma(r_n)}.$

Suppose $A \in \End(\modu^{\ot k})$  and $A  = \sum_{\rs, \sff \in [1,n]^k}  A_{\rs}^{\sff}\, \EE_{\rs}^\sff$, 
where $\{\EE_{\rs}^\sff\}$ is a basis for  $\End(\modu^{\ot k})$ of matrix units so that $\EE_{\rs}^\sff \us_{\tf} = \delta_{\rs,\tf} \us_{\sff}$,   $\delta_{\rs,\tf}$ being the Kronecker delta.   Then for any subgroup $\GG \subseteq \S_n$ and for 
$\End_{\GG}(\modu^{\ot k}) = \{ T \in \End(\modu^{\ot k}) \mid T \sigma = \sigma T$ for all $\sigma \in \GG$\}, we have
\begin{align}\label{eq:centcond}
 A\in \End_{\GG}(\modu^{\ot k}) & \iff   \sigma A  = A \sigma  \ \ \hbox{\rm for all} \ \ \sigma \in \GG \nonumber  \\
& \iff  \sum_{\sff \in [1,n]^k}  A_{\rs}^{\sff} \us_{\sigma(\sff)} = \sum_{\sff \in [1,n]^k}  A_{\sigma(\rs)}^{\sff} \us_{\sff}  \quad\hbox{\rm for all} \ \ \rs \in [1,n]^k \nonumber  \\
& \iff   A_{\rs}^{\sff}  = A_{\sigma(\rs)}^{\sigma(\sff)} \quad \hbox{\rm for all} \ \ \rs, \sff  \in [1,n]^k,  \sigma \in \GG. 
\end{align} 
 
 We adopt the shorthand notation  $(\rs|\rs') = (r_1,\dots, r_{2k}) \in [1,n]^{2k}$ when $\rs = (r_1,\dots,r_k) \in [1,n]^k$ and $\rs' = (r_{k+1},\dots, r_{2k}) \in [1,n]^k$.     
 Then for $\pi \in \Pi_{2k}$ and for all  $(\rs|\rs') \in [1,n]^{2k}$, we define 
 \begin{equation}\label{eq:Phi-coeffs-orbit}
  \Phi_{k,n}(x_\pi)_{\rs}^{\rs'} = 
  \begin{cases}  1  & \quad \text{if $r_a = r_b$ \emph{if and only if} $a$ and $b$ are in the same block of $\pi$,}  \\
 0  & \quad \text{otherwise},
  \end{cases} 
  \end{equation} 
 and set  $\Phi_{k,n}(x_\pi) = \sum_{(\rs|\rs')}  \Phi_{k,n}(x_{\pi})_\rs^{\rs'} \,  \EE_{\rs}^{\rs'}$.
 As $\{x_\pi \mid \pi \in \Pi_{2k}\}$ is a basis for $\P_k(n)$, we can extend $\Phi_{k,n}$ linearly to get a transformation
 $\Phi_{k,n}: \P_k(n) \rightarrow \End(\modu^{\ot k})$.   
 Observe that
 \begin{equation}
 \Phi_{k,n}(x_\pi) = 0  \quad \hbox{if $\pi$ has more than $n$ blocks,}
 \end{equation}
 for in that case,  there are not enough distinct values $r_a \in [1,n]$ to assign to each of the blocks.
 Furthermore, for $\sigma \in \S_n$, 
\begin{equation}\label{ref:orbitaction}
 \Phi_{k,n}(x_\pi)_{\sigma(\rs)}^{\sigma(\rs')}  =  \Phi_{k,n}(x_\pi)_{\rs}^{\rs'} \quad \text {for all \ $(\rs|\rs') \in \Pi_{2k}$.}
 \end{equation}
 which, together with  \eqref{eq:centcond}, implies that
 the image of $\Phi_{k,n}$ commutes with $\S_n$ and thus lies in $\End_{\S_n}(\modu^{\ot k})$.

Since the diagram basis  $\{d_\pi \mid \pi \in \Pi_{2k}\}$ is related to the orbit basis $\{x_\pi \mid \pi \in \Pi_{2k}\}$ by 
 the refinement relation \eqref{refinement-relation}, we have as an immediate consequence, 
 \begin{equation}\label{eq:Phi-coeffs-diagram}
  \Phi_{k,n}(d_\pi)_{\rs}^{\rs'} = 
  \begin{cases}  1  & \quad \text{if $r_a = r_b$ \emph{when} $a$ and $b$ are in the same block of $\pi$,}  \\
 0  & \quad \text{otherwise}.
  \end{cases} 
  \end{equation} 
  Note that we could, equivalently, define the map $\Phi_{k,n}$ on the diagram basis using \eqref{eq:Phi-coeffs-diagram} and argue that the action on the orbit basis \eqref{eq:Phi-coeffs-orbit} is forced by (the inverse of) the refinement relation \eqref{refinement-relation}, but it is more obvious in direction presented above.
 
 \begin{prop}\label{P:Phirepn} For $k,n  \in \ZZ_{\ge 1}$, the mapping $\Phi_{k,n}: \P_k(n) \rightarrow \End(\modu^{\ot k})$ affords a representation
 of $\P_k(n)$. \end{prop} 
 
 \begin{proof}  This result will hold, once we establish that $\Phi_{k,n}$ is an algebra homomorphism.  
 For $\pi \in \Pi_{2k}$, we set 
\begin{equation} \MM_k(\pi) = \{(\rs|\rs') \in [1,n]^{2k} \mid    \text{$r_a = r_b$ when $a$ and $b$ are in the same block of $\pi$}\}.
\end{equation} 
Then $\Phi_{k,n}(d_\pi) = \sum_{(\rs \mid \rs') \in \MM_k(\pi)}  \EE_{\rs}^{\rs'}$ by \eqref{eq:Phi-coeffs-diagram}, and for $\pi_1, \pi_2 \in \Pi_{2k}$,  we have
 \begin{align*}  \Phi_{k,n}(d_{\pi_1})  \Phi_{k,n}(d_{\pi_2}) & =  \left ( \sum_{(\rs \mid \rs') \in \MM_k(\pi_1)}  \EE_{\rs}^{\rs'}\right)
  \left ( \sum_{(\sff \mid \sff') \in \MM_k(\pi_2)}  \EE_{\sff}^{\sff'}\right)  = n^{[\pi_1 \ast \pi_2]}  \sum_{(\sff| \rs') \in \MM_k(\pi_1 \ast \pi_2)}  \EE_{\sff}^{\rs'} \\
  & =  n^{[\pi_1 \ast \pi_2]} \Phi_{k,n}( d_{\pi_1 \ast \pi_2})  = \Phi_{k,n}(d_{\pi_1} d_{\pi_2}), \end{align*}
where $[\pi_1 \ast \pi_2]$ is as in \eqref{eq:diagmult}.  As a result,  $\Phi_{k,n}$ is an algebra homomorphism.   \end{proof}  
 
 \newpage
\begin{thm}\label{T:Phi} {\rm\cite{J}, \cite[Thm.~3.6]{HR}}  Assume  $n \in \mathbb Z_{\ge 1}$  and   $\{x_\pi \mid \pi \in \Pi_{2k}\}$ is the orbit basis
for $\P_k(n)$.   
\begin{itemize}
\item[{\rm (a)}]  For $k \in \ZZ_{\ge 1}$, the representation  $\Phi_{k,n}: \P_k(n) \rightarrow \End(\modu^{\ot k})$ has
\begin{align*} \im \Phi_{k,n} &= \End_{\S_n}(\modu^{\ot k}) = \spann_{\CC}\{\Phi_{k,n}(x_\pi) \mid \pi \in \Pi_{2k} \ \text{has $\le n$ blocks}\}\ \ \hbox{\rm and}  \\ 
\ker \Phi_{k,n} &= \spann_{\CC}\{x_\pi \mid \pi \in \Pi_{2k} \ \text{has more than $n$ blocks}\}.\end{align*} 
Consequently, $\End_{\S_n}(\modu^{\ot k})$ is isomorphic to $\P_k(n)$ for $n \geq 2k$.  
\item[{\rm (b)}] For $k \in \ZZ_{\ge 0}$,  the representation   $\Phi_{k+\half}: \P_{k+\half}(n) \rightarrow \End(\modu^{\ot k})$ has
\begin{align*} \im \Phi_{k+\half,n} &= \End_{\S_{n-1}}(\modu^{\ot k}) = \spann_{\CC}\{\Phi_{k+\half,n}(x_\pi) \mid \pi \in \Pi_{2k+1} \ \text{has $\le n$ blocks}\} \ \ \hbox{\rm and} \\
\ker \Phi_{k+\half,n} &= \spann_{\CC}\{x_\pi \mid \pi \in \Pi_{2k+1} \ \text{has more than $n$ blocks}\}.\end{align*}
Consequently, $\End_{\S_{n-1}}(\modu^{\ot k})$ is isomorphic to $ \P_{k+\half}(n)$ for $n \geq 2k+1$.   \end{itemize}
\end{thm}  
 
\begin{remark} The assertion that the maps are isomorphisms 
for $k \in \half \ZZ_{\ge 0}\setminus \{0\}$ when $n \ge 2k$  
holds because no set partition $\pi \in \Pi_{2k}$ has more
than $n$ blocks in that case. 
\end{remark}       
 
\begin{remark} \label{SWhalf} In part (b) of Theorem \ref{T:Phi},  we are identifying $\S_{n-1}$ with the subgroup of $\S_n$ of permutations
that fix $n$ and making the identification $\modu^{\ot k} \cong \modu^{\ot k} \ot \us_n  \subseteq \modu^{\ot k+1}$, so that $\modu^{\ot k}$ is a submodule for both $\S_{n-1}$ and $\P_{k+\frac{1}{2}}(n) \subset \P_{k+1}(n)$. 
Then for tuples $\tilde \rs, \tilde \sff \in [1,n]^{k+1}$  having  $r_{k+1} = n = s_{k+1}$,  condition \eqref{eq:centcond} for $\GG= \S_{n-1}$ becomes
$$
A_{\tilde \rs}^{\tilde \sff} = A_{\sigma(\tilde \rs)}^{\sigma(\tilde \sff)} \quad 
\hbox{\rm for all} \ \ \tilde \rs, \tilde \sff  \in [1,n]^{k+1},  \sigma \in \S_{n-1}.
$$
Thus, the matrix units for $\GG=\S_{n-1}$ in \eqref{eq:Phi-coeffs-orbit} correspond to set partitions in  $\Pi_{2k+1}$; that is, set partitions of $\{1,2,\ldots,2(k+1)\}$ having $k+1$ and $2(k+1)$ in the same block.  The proof that $\Phi_{k+\half,n}: \P_{k+\half}(n) \rightarrow \End(\modu^{\ot k} \ot \mathsf{u}_n)$ is 
a representation is completely analogous to the proof of Proposition \ref{P:Phirepn}.
\end{remark}

\subsection{Labeled diagrams}\label{sec:LabeledDiagrams}

In computing $\Phi_{k,n}(x_\pi)_{\rs}^{\rs'}$ and $\Phi_{k,n}(d_\pi)_{\rs}^{\rs'}$ for $(\rs|\rs') \in \{1,n\}^k$,  it is helpful to think of the values of $\rs$ and $\rs'$ as labeling the vertices on the bottom row and top row, respectively, of the corresponding diagram of $\pi$. For example, when $\pi = \{1,4,10 \,|\,  2,6, 8, 9 \,|\,  3 \,|\,  5, 7 \}$, we have
\begin{align}\label{ex:LabeledOrbitDiagram}
 \Phi_{k,n}(x_\pi)_\rs^{\rs'} & = 
\!\!
\begin{array}{c}{\begin{tikzpicture}[scale=.7,line width=1.25pt] 
\foreach \i in {1,...,5} 
{ \path (\i,1) coordinate (T\i); \path (\i,0) coordinate (B\i); } 
\filldraw[fill= gray!50,draw=gray!50,line width=4pt]  (T1) -- (T5) -- (B5) -- (B1) -- (T1);
\draw[blue] (T1) .. controls +(.1,-.40) and +(-.1,-.40) .. (T3);
\draw[blue] (T3) .. controls +(.1,-.30) and +(-.1,-.30) .. (T4);
\draw[blue] (T1) .. controls +(0,-.30) and  +(0,.30) .. (B2); 
\draw[blue] (T5) .. controls +(.1,-.30) and +(0,.30) .. (B4) ; 
\draw[blue] (T2) .. controls +(0,-.30) and +(0,.30) .. (B5) ; 
\draw[blue] (B1) .. controls +(.1,.45) and +(-.1,.45) .. (B4) ;
\draw  (B1)  node[black,below=0.1cm]{${r_1}$};
\draw  (B2)  node[black,below=0.1cm]{${r_2}$};
\draw  (B3)  node[black,below=0.1cm]{${r_3}$};
\draw  (B4)  node[black,below=0.1cm]{${r_4}$};
\draw  (B5)  node[black,below=0.1cm]{${r_5}$};
\draw  (T1)  node[black,above=0.1cm]{${r_6}$};
\draw  (T2)  node[black,above=0.1cm]{${r_7}$};
\draw  (T3)  node[black,above=0.1cm]{${r_8}$};
\draw  (T4)  node[black,above=0.1cm]{${r_9}$};
\draw  (T5)  node[black,above=0.1cm]{${r_{10}}$}; 
\foreach \i in {1,...,5}
{ \filldraw[fill=white,draw=black,line width = 1pt] (T\i) circle (3pt);  \filldraw[fill=white,draw=black,line width = 1pt]  (B\i) circle (3pt); } 
\end{tikzpicture}}
\end{array}\!\!
= \begin{cases} 
1 & \quad  \hbox{if  for distinct $a,b,c,d \in [1,n]$,} \\  
& \quad \ \  r_1 = r_4 = r_{10} =a; \ \ r_2 = r_6 = r_8 = r_9 = b;   \\  
& \quad \ \ r_3 = c; \ \ r_5 = r_7=d;  \\
0 & \quad \hbox{\rm otherwise}.\end{cases}  \\
 \Phi_{k,n}(d_\pi)_\rs^{\rs'} & = \label{ex:LabeledDiagram}
\!\!
\begin{array}{c}{\begin{tikzpicture}[scale=.7,line width=1.25pt] 
\foreach \i in {1,...,5} 
{ \path (\i,1) coordinate (T\i); \path (\i,0) coordinate (B\i); } 
\filldraw[fill= gray!50,draw=gray!50,line width=4pt]  (T1) -- (T5) -- (B5) -- (B1) -- (T1);
\draw[blue] (T1) .. controls +(.1,-.40) and +(-.1,-.40) .. (T3);
\draw[blue] (T3) .. controls +(.1,-.30) and +(-.1,-.30) .. (T4);
\draw[blue] (T1) .. controls +(0,-.30) and  +(0,.30) .. (B2); 
\draw[blue] (T5) .. controls +(.1,-.30) and +(0,.30) .. (B4) ; 
\draw[blue] (T2) .. controls +(0,-.30) and +(0,.30) .. (B5) ; 
\draw[blue] (B1) .. controls +(.1,.45) and +(-.1,.45) .. (B4) ;
\draw  (B1)  node[black,below=0.1cm]{${r_1}$};
\draw  (B2)  node[black,below=0.1cm]{${r_2}$};
\draw  (B3)  node[black,below=0.1cm]{${r_3}$};
\draw  (B4)  node[black,below=0.1cm]{${r_4}$};
\draw  (B5)  node[black,below=0.1cm]{${r_5}$};
\draw  (T1)  node[black,above=0.1cm]{${r_6}$};
\draw  (T2)  node[black,above=0.1cm]{${r_7}$};
\draw  (T3)  node[black,above=0.1cm]{${r_8}$};
\draw  (T4)  node[black,above=0.1cm]{${r_9}$};
\draw  (T5)  node[black,above=0.1cm]{${r_{10}}$}; 
\foreach \i in {1,...,5}
{ \filldraw[fill=black,draw=black,line width = 1pt] (T\i) circle (3pt);  \filldraw[fill=black,draw=black,line width = 1pt]  (B\i) circle (3pt); } 
\end{tikzpicture}}
\end{array}\!\!
= \begin{cases} 
1  & \quad  \hbox{if} \ \ r_1 = r_4 = r_{10}; \\  
& \qquad  r_2 = r_6 = r_8 = r_9 ; \  \ r_5 = r_7;   \\  
0 & \quad  \hbox{otherwise}.\end{cases}  
\end{align}
Thus, 
$$
\begin{array}{lllll}
\begin{array}{c}\begin{tikzpicture}[scale=.5,line width=1.25pt] 
\foreach \i in {1,...,5} { \path (\i,1) coordinate (T\i); \path (\i,0) coordinate (B\i); } 
\filldraw[fill= gray!50,draw=gray!50,line width=4pt]  (T1) -- (T5) -- (B5) -- (B1) -- (T1);
\draw[blue] (T1) .. controls +(.1,-.40) and +(-.1,-.40) .. (T3);
\draw[blue] (T3) .. controls +(.1,-.30) and +(-.1,-.30) .. (T4);
\draw[blue] (T1) .. controls +(0,-.30) and  +(0,.30) .. (B2); 
\draw[blue] (T5) .. controls +(.1,-.30) and +(0,.30) .. (B4) ; 
\draw[blue] (T2) .. controls +(0,-.30) and +(0,.30) .. (B5) ; 
\draw[blue] (B1) .. controls +(.1,.45) and +(-.1,.45) .. (B4) ;
\draw  (B1)  node[black,below=0.1cm]{\small 3};
\draw  (B2)  node[black,below=0.1cm]{\small 2};
\draw  (B3)  node[black,below=0.1cm]{\small 5};
\draw  (B4)  node[black,below=0.1cm]{\small 3};
\draw  (B5)  node[black,below=0.1cm]{\small 1};
\draw  (T1)  node[black,above=0.1cm]{\small 2};
\draw  (T2)  node[black,above=0.1cm]{\small 1};
\draw  (T3)  node[black,above=0.1cm]{\small 2};
\draw  (T4)  node[black,above=0.1cm]{\small 2};
\draw  (T5)  node[black,above=0.1cm]{\small 3}; 
\foreach \i in {1,...,5}
{ \filldraw[fill=white,draw=black,line width = 1pt] (T\i) circle (3pt);  \filldraw[fill=white,draw=black,line width = 1pt]  (B\i) circle (3pt); } 
\end{tikzpicture} \end{array} \!\!\!= 1,
&
\begin{array}{c}\begin{tikzpicture}[scale=.5,line width=1.25pt] 
\foreach \i in {1,...,5} { \path (\i,1) coordinate (T\i); \path (\i,0) coordinate (B\i); } 
\filldraw[fill= gray!50,draw=gray!50,line width=4pt]  (T1) -- (T5) -- (B5) -- (B1) -- (T1);
\draw[blue] (T1) .. controls +(.1,-.40) and +(-.1,-.40) .. (T3);
\draw[blue] (T3) .. controls +(.1,-.30) and +(-.1,-.30) .. (T4);
\draw[blue] (T1) .. controls +(0,-.30) and  +(0,.30) .. (B2); 
\draw[blue] (T5) .. controls +(.1,-.30) and +(0,.30) .. (B4) ; 
\draw[blue] (T2) .. controls +(0,-.30) and +(0,.30) .. (B5) ; 
\draw[blue] (B1) .. controls +(.1,.45) and +(-.1,.45) .. (B4) ;
\draw  (B1)  node[black,below=0.1cm]{\small 3};
\draw  (B2)  node[black,below=0.1cm]{\small 2};
\draw  (B3)  node[black,below=0.1cm]{\small 5};
\draw  (B4)  node[black,below=0.1cm]{\small 3};
\draw  (B5)  node[black,below=0.1cm]{\small 1};
\draw  (T1)  node[black,above=0.1cm]{\small 2};
\draw  (T2)  node[black,above=0.1cm]{\small 1};
\draw  (T3)  node[black,above=0.1cm]{\small 2};
\draw  (T4)  node[black,above=0.1cm]{\small 4};
\draw  (T5)  node[black,above=0.1cm]{\small 3}; 
\foreach \i in {1,...,5}
{ \filldraw[fill=white,draw=black,line width = 1pt] (T\i) circle (3pt);  \filldraw[fill=white,draw=black,line width = 1pt]  (B\i) circle (3pt); } 
\end{tikzpicture} \end{array} \!\!\!= 0,
&
\begin{array}{c}\begin{tikzpicture}[scale=.5,line width=1.25pt] 
\foreach \i in {1,...,5} { \path (\i,1) coordinate (T\i); \path (\i,0) coordinate (B\i); } 
\filldraw[fill= gray!50,draw=gray!50,line width=4pt]  (T1) -- (T5) -- (B5) -- (B1) -- (T1);
\draw[blue] (T1) .. controls +(.1,-.40) and +(-.1,-.40) .. (T3);
\draw[blue] (T3) .. controls +(.1,-.30) and +(-.1,-.30) .. (T4);
\draw[blue] (T1) .. controls +(0,-.30) and  +(0,.30) .. (B2); 
\draw[blue] (T5) .. controls +(.1,-.30) and +(0,.30) .. (B4) ; 
\draw[blue] (T2) .. controls +(0,-.30) and +(0,.30) .. (B5) ; 
\draw[blue] (B1) .. controls +(.1,.45) and +(-.1,.45) .. (B4) ;
\draw  (B1)  node[black,below=0.1cm]{\small 3};
\draw  (B2)  node[black,below=0.1cm]{\small 2};
\draw  (B3)  node[black,below=0.1cm]{\small 5};
\draw  (B4)  node[black,below=0.1cm]{\small 3};
\draw  (B5)  node[black,below=0.1cm]{\small 5};
\draw  (T1)  node[black,above=0.1cm]{\small 2};
\draw  (T2)  node[black,above=0.1cm]{\small 5};
\draw  (T3)  node[black,above=0.1cm]{\small 2};
\draw  (T4)  node[black,above=0.1cm]{\small 2};
\draw  (T5)  node[black,above=0.1cm]{\small 3}; 
\foreach \i in {1,...,5}
{ \filldraw[fill=white,draw=black,line width = 1pt] (T\i) circle (3pt);  \filldraw[fill=white,draw=black,line width = 1pt]  (B\i) circle (3pt); } 
\end{tikzpicture} \end{array}\!\!\! = 0,
&
\begin{array}{c}\begin{tikzpicture}[scale=.5,line width=1.25pt] 
\foreach \i in {1,...,5} { \path (\i,1) coordinate (T\i); \path (\i,0) coordinate (B\i); } 
\filldraw[fill= gray!50,draw=gray!50,line width=4pt]  (T1) -- (T5) -- (B5) -- (B1) -- (T1);
\draw[blue] (T1) .. controls +(.1,-.40) and +(-.1,-.40) .. (T3);
\draw[blue] (T3) .. controls +(.1,-.30) and +(-.1,-.30) .. (T4);
\draw[blue] (T1) .. controls +(0,-.30) and  +(0,.30) .. (B2); 
\draw[blue] (T5) .. controls +(.1,-.30) and +(0,.30) .. (B4) ; 
\draw[blue] (T2) .. controls +(0,-.30) and +(0,.30) .. (B5) ; 
\draw[blue] (B1) .. controls +(.1,.45) and +(-.1,.45) .. (B4) ;
\draw  (B1)  node[black,below=0.1cm]{\small 3};
\draw  (B2)  node[black,below=0.1cm]{\small 2};
\draw  (B3)  node[black,below=0.1cm]{\small 5};
\draw  (B4)  node[black,below=0.1cm]{\small 3};
\draw  (B5)  node[black,below=0.1cm]{\small 5};
\draw  (T1)  node[black,above=0.1cm]{\small 2};
\draw  (T2)  node[black,above=0.1cm]{\small 5};
\draw  (T3)  node[black,above=0.1cm]{\small 2};
\draw  (T4)  node[black,above=0.1cm]{\small 2};
\draw  (T5)  node[black,above=0.1cm]{\small 3}; 
\foreach \i in {1,...,5}
{ \filldraw[fill=black,draw=black,line width = 1pt] (T\i) circle (3pt);  \filldraw[fill=black,draw=black,line width = 1pt]  (B\i) circle (3pt); } 
\end{tikzpicture} \end{array}\!\!\! = 1.
\end{array}
$$  

For $\pi \in \Pi_{2k}$, we designate a special labeling associated to $\pi$ as follows. 
\begin{definition}\label{D:std}   Let $\mathsf{B}_1$ be the block of $\pi$ containing 1, and for $1< j \le |\pi|$, let $\mathsf{B}_j$ be the block of $\pi$ containing the smallest number not in $\mathsf{B}_1 \cup \mathsf{B}_2 \cup \cdots \cup \mathsf{B}_{j-1}$.   
The \emph{standard labeling of $\pi$}  is  $(\bs_\pi | \bs_{\pi}')$,  where 
$\bs_\pi = (\textrm{b}_1,\dots, \textrm{b}_k)$ and $\bs_\pi' = (\textrm{b}_{k+1}, \dots,\textrm{b}_{2k})$ in $[1,n]^k$,  and 
\begin{align}\begin{split}\label{eq:beta}  
&\textrm{b}_\ell = j \ \ \hbox{\rm if} \ \  \ell \in \mathsf{B}_j\ \ \hbox{\rm for} \ \  \ell \in [1,2k].  \end{split} \end{align}  
\end{definition}

For example, when $\pi = \{1,4,10\,|\, 2,6, 8, 9\,|\, 3\,|\, 5, 7 \}$,  then  $\mathsf{B}_1 = \{1,4,10\}, \ \mathsf{B_2} = \{2,6,8,9\}, \newline \mathsf{B_3} = \{3\}$,  $\mathsf{B}_4 =
\{5,7\}$,  $\bs_\pi = (1,2,3,1,4)$, and $\bs_\pi' = (2,4,2,2,1)$.  We label vertex $\ell$ with $\textrm{b}_\ell$ so that
$$\Phi_{k,n}(x_\pi)_{\bs_\pi}^{\bs_\pi'} = \!\!\begin{array}{c}{\begin{tikzpicture}[scale=.6,line width=1.2pt] 
\foreach \i in {1,...,5} 
{ \path (\i,1) coordinate (T\i); \path (\i,0) coordinate (B\i); } 
\filldraw[fill= gray!50,draw=gray!50,line width=4pt]  (T1) -- (T5) -- (B5) -- (B1) -- (T1);
\draw[blue] (T1) .. controls +(.1,-.30) and +(-.1,-.30) .. (T3);
\draw[blue] (T3) .. controls +(.1,-.30) and +(-.1,-.30) .. (T4);
\draw[blue] (T1) .. controls +(0,-.30) and  +(0,.30) .. (B2);
\draw[blue] (T5) .. controls +(.1,-.30) and +(0,.30) .. (B4) ;
\draw[blue] (T2) .. controls +(0,-.30) and +(0,.30) .. (B5) ; 
\draw[blue] (B1) .. controls +(.1,.45) and +(-.1,.45) .. (B4) ;
\draw  (B1)  node[black,below=0.1cm]{\small ${1}$};
\draw  (B2)  node[black,below=0.1cm]{\small ${2}$};
\draw  (B3)  node[black,below=0.1cm]{\small ${3}$};
\draw  (B4)  node[black,below=0.1cm]{\small ${1}$};
\draw  (B5)  node[black,below=0.1cm]{\small ${4}$};
\draw  (T1)  node[black,above=0.1cm]{\small ${2}$};
\draw  (T2)  node[black,above=0.1cm]{\small ${4}$};
\draw  (T3)  node[black,above=0.1cm]{\small ${2}$};
\draw  (T4)  node[black,above=0.1cm]{\small ${2}$};
\draw  (T5)  node[black,above=0.1cm]{\small ${1}$}; 
\foreach \i in {1,...,5}
{ \filldraw[fill=white,draw=black,line width = 1pt] (T\i) circle (3pt);  \filldraw[fill=white,draw=black,line width = 1pt]  (B\i) circle (3pt); } 
\end{tikzpicture}}
\end{array}\!\!\!=1.
$$ 

The condition in \eqref{eq:Phi-coeffs-orbit} holds exactly when there is a $\sigma \in \S_n$ such that $\rs = \sigma(\bs_\pi)$ and $\rs' = \sigma(\bs_\pi')$. For $(\rs| \rs')$ and $(\sff|\sff')$, write $(\rs|\rs') \sim_{\S_n} (\sff|\sff')$ if  $(\rs | \rs') = (\sigma(\sff)|\sigma(\sff'))$  for
some $\sigma \in \S_n$.    Then the image of the orbit basis element $x_\pi$ under the representation $\Phi_{k,n}$ is given by 
\begin{equation}
\Phi_{k,n}(x_\pi) = \sum_{(\rs|\rs') \sim_{\S_n}  (\bs_\pi| \bs_\pi')} \EE_\rs^{\rs'},
 \end{equation}
which is the sum of  matrix units over distinct elements in the $\S_n$-orbit of the standard labeling of $\pi$. In this way, the endomorphisms $\Phi_{k,n}(x_\pi)$, for $\pi \in \Pi_{2k}$, are the indicator functions for the $\S_n$-orbits on $[1,n]^{2k}$, and this is why the term ``orbit basis" is used.

As described in Remark \ref{SWhalf},  the partition algebra $\P_{k+\frac{1}{2}}(n)$ acts on $\modu^{\ot k}$ by identifying  
$\modu^{\ot k}$ with $\modu^{\ot k}\ot \us_n \subseteq \modu^{\ot k+1}$. The basis then consists of the simple tensors  $\us_{r_1} \ot \cdots \ot \us_{r_k} \ot \us_n$ with $(r_1, \dots, r_k) \in [1,n]^k$. If $\pi \in \Pi_{2k+1}$, then $\pi$ has  $k+1$ and $2(k+1)$ in the same block, and the matrix of $\Phi_{k + \frac{1}{2},n}(x_\pi)$ is the same as $\Phi_{k + 1,n}(x_\pi)$ but restricted to indices of the form $\tilde \rs = (r_1, \ldots, r_k, n)$ and
$\tilde \rs' = (r_1', \ldots, r_k', n)$. Thus, for example, if $k+\half= 5\frac{1}{2}$ and $n \geq |\pi| +1 = 5$ for $\pi = \{1,4,11,12 \,|\,  2,7,9,10 \,|\,  3 \,|\,  5,8\}$, we have
$$
\Phi_{k+\frac{1}{2},n}(x_\pi)_{\tilde \rs}^{\tilde \rs'} = 
\begin{array}{c}{\begin{tikzpicture}[scale=.7,line width=1.25pt] 
\foreach \i in {1,...,6} 
{ \path (\i,1) coordinate (T\i); \path (\i,0) coordinate (B\i); } 
\filldraw[fill= gray!50,draw=gray!50,line width=4pt]  (T1) -- (T6) -- (B6) -- (B1) -- (T1);
\draw[blue] (T1) .. controls +(.1,-.40) and +(-.1,-.40) .. (T3);
\draw[blue] (T3) .. controls +(.1,-.30) and +(-.1,-.30) .. (T4);
\draw[blue] (T1) .. controls +(0,-.30) and  +(0,.30) .. (B2); 
\draw[blue] (T5) .. controls +(.1,-.30) and +(0,.30) .. (B4) ; 
\draw[blue] (T2) .. controls +(0,-.30) and +(0,.30) .. (B5) ; 
\draw[blue] (B1) .. controls +(.1,.45) and +(-.1,.45) .. (B4) ;
\draw[blue] (T5) .. controls +(.1,-.40) and +(-.1,-.40) .. (T6);
\draw[blue] (T6) -- (B6);
\draw  (B1)  node[black,below=0.1cm]{${r_1}$};
\draw  (B2)  node[black,below=0.1cm]{${r_2}$};
\draw  (B3)  node[black,below=0.1cm]{${r_3}$};
\draw  (B4)  node[black,below=0.1cm]{${r_4}$};
\draw  (B5)  node[black,below=0.1cm]{${r_5}$};
\draw  (T1)  node[black,above=0.1cm]{${r_6}$};
\draw  (T2)  node[black,above=0.1cm]{${r_7}$};
\draw  (T3)  node[black,above=0.1cm]{${r_8}$};
\draw  (T4)  node[black,above=0.1cm]{${r_9}$};
\draw  (T5)  node[black,above=0.1cm]{${r_{10}}$}; 
\draw  (T6)  node[black,above=0.1cm]{$n$}; 
\draw  (B6)  node[black,below=0.1cm]{$n$}; 
\foreach \i in {1,...,6}
{ \filldraw[fill=white,draw=black,line width = 1pt] (T\i) circle (3pt);  \filldraw[fill=white,draw=black,line width = 1pt]  (B\i) circle (3pt); } 
\end{tikzpicture}}
\end{array}.
$$
 In the example above, the labels on the bottom row then would be
$n,2,3,n,4,n$.    
The standard labeling of diagrams on the half-integer levels is the same as the integer levels except that the block containing $k+1$ and
$2(k+1)$ is always labeled by $n$.

\section{Multiplication in the Orbit Basis}\label{sec:Mult}

For $k,n \in \ZZ_{\ge 1}$, let $\Pi_{2k}(n)$ be the subset of $\Pi_{2k}$ of set partitions having at most $n$ blocks. Then $\Pi_{2k}(n) = \Pi_{2k}$ if and only if $n \ge 2k$.
For $\Phi_{k,n}:  \P_k(n) \rightarrow \End_{\S_n}(\modu^{\ot k})$,  define
\begin{equation}
X_{\pi} = \Phi_{k,n}(x_\pi), \qquad \pi \in \Pi_{2k},
\end{equation}
and observe that $X_\pi = 0$ if $\pi$ has more than $n$ blocks. Theorem \ref{T:Phi}\,(a) tells us that $\{ X_{\pi} \mid \pi \in \Pi_{2k}\}$ spans $\im \Phi_{k,n}$,  and $\{ X_{\pi} \mid \pi \in \Pi_{2k}(n)\}$ is a basis for $\im \Phi_{k,n}$. 
In this section, we first prove the formula in Lemma \ref{T:mult} for the product of two transformations $X_\pi$ in $\End_{\S_n}(\M_n^{\ot k})$, which is isomorphic to $\P_k(n)$ when $n\ge 2k$.   Theorem \ref{C:mult} extends this
result to the orbit basis of any partition algebra $\P_k(\xi)$. 

In stating these results,  we apply  the following conventions:
For $\xi \in \CC$ and $\ell \in \ZZ_{\ge 0}$,
\begin{equation}
(\xi)_\ell = \xi(\xi-1) \cdots (\xi-\ell+1).
\end{equation} Thus,  when $m$ is a nonnegative integer, $(m)_\ell = 0$ if $\ell >m $;   $(m)_\ell =  m!/(m-\ell)!$ if $m \ge \ell$; and $(m)_0 = 1$.  
If $\pi \in \Pi_{2k}$,  then by restriction, $\pi$ induces a set partition on the bottom row $\{1, 2, \ldots, k\}$ and a set partition on the top row $\{k+1, k+2, \ldots, 2k\}$. If 
$\pi_1, \pi_2 \in \Pi_{2k}$, then we say $\pi_1 \ast \pi_2$ \emph{exactly matches in the middle} if the set partition that $\pi_1$ induces on its bottom row equals the set partition that $\pi_2$ induces on the top row modulo $k$.  When that happens,  $\pi_1 \ast \pi_2$ is the
concatenation of the two diagrams.   For example, if $k = 4$, then $\pi_1 = \{ 1, 4, 5\,|\, 2, 8 \,|\, 3, 6, 7\}$ induces the set partition $\{1,4 \,|\, 2 \,|\, 3\}$ on the bottom row of $\pi_1$, and $\pi_2 = \{ 1,5,8 \,|\, 2,6\,|\, 3\,|\, 4,7 \}$ induces the set partition $\{5,8\,|\, 6\,|\, 7\} \equiv \{1,4 \,|\, 2 \,|\, 3\}\, \mathsf{mod}\, 4$ on the top row of $\pi_2$. Thus, $\pi_1 \ast \pi_2$ exactly matches in the middle. This definition is easy to see in terms of the diagrams. 

  In the product expression
below,  $X_{\varrho} = 0$ whenever $\varrho$ has more than $n$ blocks.   Recall that $[\pi_1 \ast \pi_2]$ is the number of
blocks in the middle row of $\pi_1 \ast \pi_2$.  

\begin{lemma}\label{T:mult}  Multiplication in the basis $\{ X_\pi\}_{\pi \in \Pi_{2k}(n)}$ of $\End_{\S_n} (\modu^{\otimes k})$ is given by
$$X_{\pi_1}X_{\pi_2} = 
\begin{cases}
\displaystyle{\sum_{\vr} (n-|\vr|)_{[\pi_1 \ast \pi_2]} \, X_\varrho,} &\quad \hbox{if $\pi_1 \ast \pi_2$ exactly matches in the middle,}\\
0 &\quad \hbox{otherwise,} 
\end{cases}
$$
where  the sum is over all coarsenings $\vr$ of $\pi_1 \ast \pi_2$ obtained by connecting  blocks that lie entirely in the top row of $\pi_1$
 to blocks that lie entirely in the bottom row of $\pi_2$.     \end{lemma}

 \begin{remark} By \eqref{eq:pn},  $\pn(\pi_1 \ast \pi_2) \le \min(\pn(\pi_1) , \pn(\pi_2))$. 
However, the  set partitions $\vr$ occurring in the product in
Lemma \ref{T:mult},  satisfy $\pn(\varrho) \ge \pn(\pi_1 \ast \pi_2)$, and so it is possible that
 $\pn(\vr) > \pn(\pi_1)$ or $\pn(\pi_2)$ (or both)  for some $\vr$.  This happens  in several of the examples below.   \end{remark}    

\begin{examples}\label{exs:mult} Before proving Lemma  \ref{T:mult}, we give some examples to illustrate multiplication
in the basis $\{X_\pi \mid \pi \in \Pi_{2k}(n)\}$ of $\End_{\S_n}(\M_n^{\ot k})$. In these examples, the edges added for the coarsenings are displayed in red. In each case, we assume $n$ is at least equal to the number of parts in the diagrams being multiplied.
\bigskip

\noindent (1) \  Suppose $k=3$,  $n \ge 2$, and $\pi = \{1,2,3 \,|\,  4,5,6\}\in \Pi_6$. Then,
in terms of matrix units,   $X_\pi = \sum_{i\neq j \in [1,n]}  \EE_{iii}^{jjj}$ so that  
$$X_\pi^2  =  (n-2)\sum_{i \neq j \in [1,n]}  \EE_{iii}^{jjj} +
  (n-1)\sum_{i \in [1,n]}  \EE_{iii}^{iii}.$$  Writing the corresponding orbit diagrams, we have
$$
\begin{array}{c}
\begin{tikzpicture}[scale=.6,line width=1.25pt] 
\foreach \i in {1,...,3} 
{ \path (\i,1) coordinate (T\i); \path (\i,0) coordinate (B\i); } 
\filldraw[fill= gray!50,draw=gray!50,line width=4pt]  (T1) -- (T3) -- (B3) -- (B1) -- (T1);
\draw[blue] (B1) .. controls +(0,+.3) and +(0,.3) .. (B2).. controls +(0,+.3) and +(0,.3) .. (B3);
\draw[blue] (T1) .. controls +(0,-.3) and +(0,-.3) .. (T2) .. controls +(0,-.3) and +(0,-.3) .. (T3);
\foreach \i in {1,...,3} 
{  \filldraw[fill=white,draw=black,line width = 1pt]  (T\i) circle (3pt);  \filldraw[fill=white,draw=black,line width = 1pt] (B\i) circle (3pt); } 
\end{tikzpicture} \\
\begin{tikzpicture}[scale=.6,line width=1.25pt] 
\foreach \i in {1,...,3} 
{ \path (\i,1) coordinate (T\i); \path (\i,0) coordinate (B\i); } 
\filldraw[fill= gray!50,draw=gray!50,line width=4pt]  (T1) -- (T3) -- (B3) -- (B1) -- (T1);
\draw[blue] (B1) .. controls +(0,+.3) and +(0,.3) .. (B2).. controls +(0,+.3) and +(0,.3) .. (B3);
\draw[blue] (T1) .. controls +(0,-.3) and +(0,-.3) .. (T2) .. controls +(0,-.3) and +(0,-.3) .. (T3);
\foreach \i in {1,...,3} 
{ \filldraw[fill=white,draw=black,line width = 1pt] (T\i) circle (3pt); \filldraw[fill=white,draw=black,line width = 1pt] (B\i) circle (3pt); } 
\end{tikzpicture}
\end{array} = (n-2)
\begin{array}{c}
\begin{tikzpicture}[scale=.6,line width=1.25pt] 
\foreach \i in {1,...,3} 
{ \path (\i,1) coordinate (T\i); \path (\i,0) coordinate (B\i); } 
\filldraw[fill= gray!50,draw=gray!50,line width=4pt]  (T1) -- (T3) -- (B3) -- (B1) -- (T1);
\draw[blue] (B1) .. controls +(0,+.3) and +(0,.3) .. (B2).. controls +(0,+.3) and +(0,.3) .. (B3);
\draw[blue] (T1) .. controls +(0,-.3) and +(0,-.3) .. (T2) .. controls +(0,-.3) and +(0,-.3) .. (T3);
\foreach \i in {1,...,3} 
{ \filldraw[fill=white,draw=black,line width = 1pt] (T\i) circle (3pt);  \filldraw[fill=white,draw=black,line width = 1pt]  (B\i) circle (3pt); } 
\end{tikzpicture}\end{array} + (n-1)
\begin{array}{c}
\begin{tikzpicture}[scale=.6,line width=1.25pt] 
\foreach \i in {1,...,3} 
{ \path (\i,1) coordinate (T\i); \path (\i,0) coordinate (B\i); } 
\filldraw[fill= gray!50,draw=gray!50,line width=4pt]  (T1) -- (T3) -- (B3) -- (B1) -- (T1);
\draw[blue] (B1) .. controls +(0,+.30) and +(0,.30) .. (B2).. controls +(0,+.30) and +(0,.30) .. (B3);
\draw[blue] (T1) .. controls +(0,-.30) and +(0,-.30) .. (T2) .. controls +(0,-.30) and +(0,-.30) .. (T3);
\draw[purple] (T1) -- (B1);
\foreach \i in {1,...,3} 
{ \filldraw[fill=white,draw=black,line width = 1pt]  (T\i) circle (3pt); \filldraw[fill=white,draw=black,line width = 1pt]  (B\i) circle (3pt); } 
\end{tikzpicture}\end{array},
\qquad\hbox{$n \ge 2$}.
$$
Thus, $X_\pi^2 = (n-2) X_\pi + (n-1) X_\varrho$,  where $\varrho = \{1,2,3,4,5,6\}$, as predicted by 
Lemma \ref{T:mult}. \bigskip
  
\noindent (2)  \   Here $k = 4$, $n \ge 4$, and $[\pi_1 \ast \pi_2] = 2$ (two blocks are removed upon concatenation
of  ${\pi_1}$ and ${\pi_2}$). 
$$
\begin{array}{l}
\begin{array}{c}
\begin{tikzpicture}[scale=.6,line width=1.25pt] 
\foreach \i in {1,...,4} 
{ \path (\i,1) coordinate (T\i); \path (\i,0) coordinate (B\i); } 
\filldraw[fill= gray!50,draw=gray!50,line width=4pt]  (T1) -- (T4) -- (B4) -- (B1) -- (T1);
\draw[blue] (T2) .. controls +(0,-.30) and +(0,-.30) .. (T3);
\draw[blue] (B3) .. controls +(0,+.30) and +(0,+.30) .. (B4);
\draw[blue] (T4) -- (B2);
\foreach \i in {1,...,4} 
{\filldraw[fill=white,draw=black,line width = 1pt] (T\i) circle (3pt); \filldraw[fill=white,draw=black,line width = 1pt](B\i) circle (3pt); } 
\end{tikzpicture} \\
\begin{tikzpicture}[scale=.6,line width=1.25pt] 
\foreach \i in {1,...,4} 
{ \path (\i,1) coordinate (T\i); \path (\i,0) coordinate (B\i); } 
\filldraw[fill= gray!50,draw=gray!50,line width=4pt]  (T1) -- (T4) -- (B4) -- (B1) -- (T1);
\draw[blue] (T2) .. controls +(0,-.30) and +(0,-.30) .. (B2);
\draw[blue] (T3) .. controls +(0,-.30) and +(0,-.30) .. (T4);
\draw[blue] (B1) .. controls +(0,+.40) and +(0,+.40) .. (B3);
\foreach \i in {1,...,4} 
{\filldraw[fill=white,draw=black,line width = 1pt](T\i) circle (3pt);\filldraw[fill=white,draw=black,line width = 1pt](B\i) circle (3pt); } 
\end{tikzpicture}
\end{array} =  (n-5)(n-6)
\begin{array}{c}
\begin{tikzpicture}[scale=.6,line width=1.25pt] 
\foreach \i in {1,...,4} 
{ \path (\i,1) coordinate (T\i); \path (\i,0) coordinate (B\i); } 
\filldraw[fill= gray!50,draw=gray!50,line width=4pt]  (T1) -- (T4) -- (B4) -- (B1) -- (T1);
\draw[blue] (T2) .. controls +(0,-.30) and +(0,-.30) .. (T3);
\draw[blue] (B1) .. controls +(0,+.40) and +(0,+.40) .. (B3);
\draw[blue] (T4) -- (B2);
\foreach \i in {1,...,4} 
{\filldraw[fill=white,draw=black,line width = 1pt] (T\i) circle (3pt); \filldraw[fill=white,draw=black,line width = 1pt] (B\i) circle (3pt); } 
\end{tikzpicture}
\end{array}  \\
\hskip.25in+ (n-4)(n-5) \Big(
\begin{array}{c}
\begin{tikzpicture}[scale=.6,line width=1.25pt] 
\foreach \i in {1,...,4} 
{ \path (\i,1) coordinate (T\i); \path (\i,0) coordinate (B\i); } 
\filldraw[fill= gray!50,draw=gray!50,line width=4pt]  (T1) -- (T4) -- (B4) -- (B1) -- (T1);
\draw[blue] (T2) .. controls +(0,-.30) and +(0,-.30) .. (T3);
\draw[blue] (B1) .. controls +(0,+.40) and +(0,+.40) .. (B3);
\draw[purple] (T1)--(B1);
\draw[blue] (T4) -- (B2);
\foreach \i in {1,...,4} 
{\filldraw[fill=white,draw=black,line width = 1pt] (T\i) circle (3pt); \filldraw[fill=white,draw=black,line width = 1pt] (B\i) circle (3pt); } 
\end{tikzpicture}
\end{array}+
\begin{array}{c}
\begin{tikzpicture}[scale=.6,line width=1.25pt] 
\foreach \i in {1,...,4} 
{ \path (\i,1) coordinate (T\i); \path (\i,0) coordinate (B\i); } 
\filldraw[fill= gray!50,draw=gray!50,line width=4pt]  (T1) -- (T4) -- (B4) -- (B1) -- (T1);
\draw[blue] (T2) .. controls +(0,-.30) and +(0,-.30) .. (T3);
\draw[blue] (B1) .. controls +(0,+.40) and +(0,+.40) .. (B3);
\draw[purple] (T1)--(B4);
\draw[blue] (T4) -- (B2);
\foreach \i in {1,...,4} 
{\filldraw[fill=white,draw=black,line width = 1pt](T\i) circle (3pt); \filldraw[fill=white,draw=black,line width = 1pt](B\i) circle (3pt); } 
\end{tikzpicture}
\end{array}
+
\begin{array}{c}
\begin{tikzpicture}[scale=.6,line width=1.25pt] 
\foreach \i in {1,...,4} 
{ \path (\i,1) coordinate (T\i); \path (\i,0) coordinate (B\i); } 
\filldraw[fill= gray!50,draw=gray!50,line width=4pt]  (T1) -- (T4) -- (B4) -- (B1) -- (T1);
\draw[blue] (T2) .. controls +(0,-.30) and +(0,-.30) .. (T3);
\draw[blue] (B1) .. controls +(0,+.40) and +(0,+.40) .. (B3);
\draw[purple] (T2).. controls +(0,-.30) and +(0,+.30) ..(B1);
\draw[blue] (T4) -- (B2);
\foreach \i in {1,...,4} 
{ \filldraw[fill=white,draw=black,line width = 1pt] (T\i) circle (3pt); \filldraw[fill=white,draw=black,line width = 1pt](B\i) circle (3pt); } 
\end{tikzpicture}
\end{array}+
\begin{array}{c}
\begin{tikzpicture}[scale=.6,line width=1.25pt] 
\foreach \i in {1,...,4} 
{ \path (\i,1) coordinate (T\i); \path (\i,0) coordinate (B\i); } 
\filldraw[fill= gray!50,draw=gray!50,line width=4pt]  (T1) -- (T4) -- (B4) -- (B1) -- (T1);
\draw[blue] (T2) .. controls +(0,-.30) and +(0,-.30) .. (T3);
\draw[blue] (B1) .. controls +(0,+.40) and +(0,+.40) .. (B3);
\draw[blue] (T4) -- (B2);
\draw[purple] (T3).. controls +(0,-.30) and +(0,+.30) ..(B4);
\foreach \i in {1,...,4} 
{\filldraw[fill=white,draw=black,line width = 1pt](T\i) circle (3pt); \filldraw[fill=white,draw=black,line width = 1pt](B\i) circle (3pt); } 
\end{tikzpicture}
\end{array}\Big) \\
\hskip.25in+(n-3)(n-4) \Big(
\begin{array}{c}
\begin{tikzpicture}[scale=.6,line width=1.25pt] 
\foreach \i in {1,...,4} 
{ \path (\i,1) coordinate (T\i); \path (\i,0) coordinate (B\i); } 
\filldraw[fill= gray!50,draw=gray!50,line width=4pt]  (T1) -- (T4) -- (B4) -- (B1) -- (T1);
\draw[blue] (T2) .. controls +(0,-.30) and +(0,-.30) .. (T3);
\draw[blue] (B1) .. controls +(0,+.40) and +(0,+.40) .. (B3);
\draw[purple] (T1)--(B1);
\draw[purple] (T3).. controls +(0,-.30) and +(0,+.30) ..(B4);;
\draw[blue] (T4) -- (B2);
\foreach \i in {1,...,4} 
{\filldraw[fill=white,draw=black,line width = 1pt](T\i) circle (3pt); \filldraw[fill=white,draw=black,line width = 1pt] (B\i) circle (3pt); } 
\end{tikzpicture}
\end{array}+
\begin{array}{c}
\begin{tikzpicture}[scale=.6,line width=1.25pt] 
\foreach \i in {1,...,4} 
{ \path (\i,1) coordinate (T\i); \path (\i,0) coordinate (B\i); } 
\filldraw[fill= gray!50,draw=gray!50,line width=4pt]  (T1) -- (T4) -- (B4) -- (B1) -- (T1);
\draw[blue] (T2) .. controls +(0,-.30) and +(0,-.30) .. (T3);
\draw[blue] (B1) .. controls +(0,+.40) and +(0,+.40) .. (B3);
\draw[purple] (T1)--(B4);
\draw[purple] (T2).. controls +(0,-.30) and +(0,+.30) ..(B1);
\draw[blue] (T4) -- (B2);
\foreach \i in {1,...,4} 
{\filldraw[fill=white,draw=black,line width = 1pt] (T\i) circle (3pt); \filldraw[fill=white,draw=black,line width = 1pt] (B\i) circle (3pt); } 
\end{tikzpicture}
\end{array}\Big).
\end{array}
$$ \medskip

\noindent (3) \   Here $k = 4$, $n \ge 6$, and $[\pi_1 \ast \pi_2] = 1$. In this case, the second vertex in the top row of $\pi_1$ is in a block that is not  entirely in the top row,  so it is not allowed to be connected to a block in the bottom row of $\pi_2$.
$$
\begin{array}{c}
\begin{tikzpicture}[scale=.6,line width=1.25pt] 
\foreach \i in {1,...,4} 
{ \path (\i,1) coordinate (T\i); \path (\i,0) coordinate (B\i); } 
\filldraw[fill= gray!50,draw=gray!50,line width=4pt]  (T1) -- (T4) -- (B4) -- (B1) -- (T1);
\draw[blue] (T1) -- (B1);
\draw[blue] (T2) -- (B2);
\draw[blue] (T4) -- (B3);
\foreach \i in {1,...,4} 
{\filldraw[fill=white,draw=black,line width = 1pt] (T\i) circle (3pt); \filldraw[fill=white,draw=black,line width = 1pt](B\i) circle (3pt); } 
\end{tikzpicture} \\
\begin{tikzpicture}[scale=.6,line width=1.25pt] 
\foreach \i in {1,...,3} 
{ \path (\i,1) coordinate (T\i); \path (\i,0) coordinate (B\i); } 
\filldraw[fill= gray!50,draw=gray!50,line width=4pt]  (T1) -- (T4) -- (B4) -- (B1) -- (T1);
\draw[blue] (T1) -- (B2);
\draw[blue] (T3) -- (B1);
\foreach \i in {1,...,4} 
{\filldraw[fill=white,draw=black,line width = 1pt](T\i) circle (3pt);\filldraw[fill=white,draw=black,line width = 1pt](B\i) circle (3pt); } 
\end{tikzpicture}
\end{array} 
=  (n-6)
\begin{array}{c}
\begin{tikzpicture}[scale=.6,line width=1.25pt] 
\foreach \i in {1,...,4} 
{ \path (\i,1) coordinate (T\i); \path (\i,0) coordinate (B\i); } 
\filldraw[fill= gray!50,draw=gray!50,line width=4pt]  (T1) -- (T4) -- (B4) -- (B1) -- (T1);
\draw[blue] (T1) -- (B2);
\draw[blue] (T4) -- (B1);
\foreach \i in {1,...,4} 
{\filldraw[fill=white,draw=black,line width = 1pt](T\i) circle (3pt);\filldraw[fill=white,draw=black,line width = 1pt](B\i) circle (3pt); } 
\end{tikzpicture}
\end{array}
+  (n-5) \left(
\begin{array}{c}
\begin{tikzpicture}[scale=.6,line width=1.25pt] 
\foreach \i in {1,...,4} 
{ \path (\i,1) coordinate (T\i); \path (\i,0) coordinate (B\i); } 
\filldraw[fill= gray!50,draw=gray!50,line width=4pt]  (T1) -- (T4) -- (B4) -- (B1) -- (T1);
\draw[blue] (T1) -- (B2);
\draw[blue] (T4) -- (B1);
\draw[purple] (T3) -- (B3);
\foreach \i in {1,...,4} 
{\filldraw[fill=white,draw=black,line width = 1pt](T\i) circle (3pt);\filldraw[fill=white,draw=black,line width = 1pt](B\i) circle (3pt); } 
\end{tikzpicture}
\end{array}+
\begin{array}{c}
\begin{tikzpicture}[scale=.6,line width=1.25pt] 
\foreach \i in {1,...,4} 
{ \path (\i,1) coordinate (T\i); \path (\i,0) coordinate (B\i); } 
\filldraw[fill= gray!50,draw=gray!50,line width=4pt]  (T1) -- (T4) -- (B4) -- (B1) -- (T1);
\draw[blue] (T1) -- (B2);
\draw[blue] (T4) -- (B1);
\draw[purple] (T3) -- (B4);
\foreach \i in {1,...,4} 
{\filldraw[fill=white,draw=black,line width = 1pt](T\i) circle (3pt);\filldraw[fill=white,draw=black,line width = 1pt](B\i) circle (3pt); } 
\end{tikzpicture}
\end{array}\right).$$ \medskip

{\noindent (4)} \   The product on the left is 0, since the diagrams do  not exactly match in the middle. In the product on the right, $[\pi_1 \ast \pi_2] = 0$, and 
the bottom diagram does not have any blocks entirely in the bottom row, so no coarsenings are possible.

$$
\begin{array}{cl}
\begin{array}{c}
\begin{array}{c}\begin{tikzpicture}[scale=.6,line width=1.25pt] 
\foreach \i in {1,...,6} 
{ \path (\i,1) coordinate (T\i); \path (\i,0) coordinate (B\i); } 
\filldraw[fill= gray!50,draw=gray!50,line width=4pt]  (T1) -- (T6) -- (B6) -- (B1) -- (T1);
\draw[blue] (T2) .. controls +(0,-.50) and +(0,-.50) .. (T4);
\draw[blue] (T3) .. controls +(0,-.40) and +(0,-.40) .. (T5);
\draw[blue] (B2) .. controls +(0,+.30) and +(0,+.30) .. (B3);
\draw[blue] (B4) .. controls +(0,+.30) and +(0,+.30) .. (B5);
\draw[blue] (T1) -- (B2);\draw[blue] (T5) -- (B6);
\foreach \i in {1,...,6} 
{\filldraw[fill=white,draw=black,line width = 1pt] (T\i) circle (3pt); \filldraw[fill=white,draw=black,line width = 1pt](B\i) circle (3pt); } 
\end{tikzpicture}\end{array} \\
\begin{array}{c}\begin{tikzpicture}[scale=.6,line width=1.25pt] 
\foreach \i in {1,...,6} 
{ \path (\i,1) coordinate (T\i); \path (\i,0) coordinate (B\i); } 
\filldraw[fill= gray!50,draw=gray!50,line width=4pt]  (T1) -- (T6) -- (B6) -- (B1) -- (T1);
\draw[blue] (T2) .. controls +(0,-.30) and +(0,-.30) .. (T3);
\draw[blue] (B1) .. controls +(0,+.40) and +(0,+.40) .. (B3);
\draw[blue] (B4) .. controls +(0,+.30) and +(0,+.30) .. (B5) .. controls +(0,+.30) and +(0,+.30) .. (B6);
\draw[blue] (T1) -- (B2);\draw[blue] (T6) -- (B6);
\foreach \i in {1,...,6} 
{\filldraw[fill=white,draw=black,line width = 1pt](T\i) circle (3pt);\filldraw[fill=white,draw=black,line width = 1pt](B\i) circle (3pt); } 
\end{tikzpicture}
\end{array}\end{array}= & 0,
\end{array}
\qquad
\begin{array}{cl}
\begin{array}{c}
\begin{array}{c}\begin{tikzpicture}[scale=.6,line width=1.25pt] 
\foreach \i in {1,...,6} 
{ \path (\i,1) coordinate (T\i); \path (\i,0) coordinate (B\i); } 
\filldraw[fill= gray!50,draw=gray!50,line width=4pt]  (T1) -- (T6) -- (B6) -- (B1) -- (T1);
\draw[blue] (T2) -- (B3);
\draw[blue] (T4) -- (B5);
\draw[blue] (T6) -- (B4);
\draw[blue] (T3) .. controls +(0,-.50) and +(0,-.50) .. (T5);
\draw[blue] (T1) -- (B2);
\foreach \i in {1,...,6} 
{\filldraw[fill=white,draw=black,line width = 1pt] (T\i) circle (3pt); \filldraw[fill=white,draw=black,line width = 1pt](B\i) circle (3pt); } 
\end{tikzpicture}\end{array} \\
\begin{array}{c}\begin{tikzpicture}[scale=.6,line width=1.25pt] 
\foreach \i in {1,...,6} 
{ \path (\i,1) coordinate (T\i); \path (\i,0) coordinate (B\i); } 
\filldraw[fill= gray!50,draw=gray!50,line width=4pt]  (T1) -- (T6) -- (B6) -- (B1) -- (T1);
\draw[blue] (T1) -- (B2);
\draw[blue] (T2) -- (B1);
\draw[blue] (T3) -- (B4);;
\draw[blue] (T5) -- (B3);
\draw[blue] (T6) -- (B5);
\draw[blue] (T4) -- (B6);
\foreach \i in {1,...,6} 
{\filldraw[fill=white,draw=black,line width = 1pt](T\i) circle (3pt);\filldraw[fill=white,draw=black,line width = 1pt](B\i) circle (3pt); } 
\end{tikzpicture}
\end{array}
\end{array}= & \begin{array}{c}\begin{tikzpicture}[scale=.6,line width=1.25pt] 
\foreach \i in {1,...,6} 
{ \path (\i,1) coordinate (T\i); \path (\i,0) coordinate (B\i); } 
\filldraw[fill= gray!50,draw=gray!50,line width=4pt]  (T1) -- (T6) -- (B6) -- (B1) -- (T1);
\draw[blue] (T1) -- (B1);
\draw[blue] (T2) -- (B4);
\draw[blue] (T4) -- (B3);
\draw[blue] (T3) .. controls +(0,-.50) and +(0,-.50) .. (T5);
\draw[blue] (T6) -- (B6);
\foreach \i in {1,...,6} 
{\filldraw[fill=white,draw=black,line width = 1pt](T\i) circle (3pt);\filldraw[fill=white,draw=black,line width = 1pt](B\i) circle (3pt); } 
\end{tikzpicture}
\end{array}.
\end{array}
$$
\end{examples}

\noindent   \emph{Proof of Lemma \ref{T:mult}.} \   
Assume $\pi_1, \pi_2 \in \Pi_{2k}(n)$.  Then    
$$X_{\pi_1} = \sum_{(\rs|\rs') \in [1,n]^{2k}} \left(X_{\pi_1}\right)_\rs^{\rs'} \EE_{\rs}^{\rs'} \quad \text{and} \quad
X_{\pi_2} = \sum_{(\rs|\rs')\in [1,n]^{2k}} \left(X_{\pi_2}\right)_\rs^{\rs'}   \EE_{\rs}^{\rs'},$$
where the coefficients are 0 or 1  as in \eqref{eq:Phi-coeffs-orbit}.  The coefficient of
$\EE_{\rs}^{\rs'}$ in the product  $X_{\pi_1}X_{\pi_2}$  
is given by
\begin{equation}\label{eq:p1p2} \left(X_{\pi_1}X_{\pi_2}\right)_{\rs}^{\rs'}  = \sum_{\qs \in [1,n]^k}  
 \left(X_{\pi_1}\right)_\qs^{\rs'} \left(X_{\pi_2}\right)_\rs^{\qs}.
 \end{equation}
In this expression, $\qs$ simultaneously labels the bottom row of $\pi_1$ and the top row of $\pi_2$, so in order for 
$X_{\pi_1}X_{\pi_2}$ to be nonzero,  it must be that $\pi_1 \ast \pi_2$ exactly matches in the middle.

We consider for which $(\rs|\rs')$ the expression for
$ \left(X_{\pi_1}X_{\pi_2}\right)_{\rs}^{\rs'}$ in \eqref{eq:p1p2} is nonzero. 
The tuple $\rs' \in [1,n]^k$ is some permutation of the standard labeling $\bs_{\pi_1}'$ of the 
top row of $\pi_1$,  and $\rs$ is some permutation of the standard labeling  
$\bs_{\pi_2}$ of the bottom row of $\pi_2$, where the standard labelings are as in Definition \ref{D:std}.  If a block in the top row of $\pi_1$ is connected to a block in the bottom row of $\pi_2$ in
$\pi_1 \ast \pi_2$,  then those blocks must carry the same label for all $(\rs| \rs')$  with 
$ \left(X_{\pi_1}X_{\pi_2}\right)_{\rs}^{\rs'}$ nonzero. 
However, blocks that lie entirely in the top row of $\pi_1$ and blocks that lie entirely in the bottom row of $\pi_2$ may or may not have the same label in some $(\rs| \rs')$.   
We account for those possibilities by the coarsenings $\vr$ of $\pi_1 \ast \pi_2$  that are
obtained by connecting blocks entirely 
in the top row of $\pi_1$ with blocks living entirely in the bottom row of $\pi_2$ in the concatenation $\pi_1 \ast \pi_2$.  (Example \ref{Example:LabeledDiagramProof} illustrates labeled diagrams for the product $ \left(X_{\pi_1}\right)_\qs^{\rs'}\left(X_{\pi_2}\right)_\rs^{\qs}$ in \eqref{eq:p1p2}.)

Suppose $(\rs|\rs')$ is a labeling of such a coarsening $\vr$.  In summing over the $\qs \in [1,n]^k$, 
corresponding to the middle row blocks of $\pi_1 \ast \pi_2$, we have $n - |\vr|$ choices for the entry of $\qs$
that is assigned to the first block;  $n-|\vr|-1$ for the next one,  as it must be different from the first;  and so forth, since
the entries of $\qs$ must be distinct from those  of $(\rs|\rs')$.   
Hence,
$$
(X_{\pi_1} X_{\pi_2})_{\rs}^{\rs'}  
= \sum_{\vr} (n-|\vr|)_{[\pi_1 \ast \pi_2]} \, (X_\varrho)_{\rs}^{\rs'} 
$$
where the sum ranges over all coarsenings of $\pi_1 \ast \pi_2$ obtained by connecting a block lying entirely in the top
row of $\pi_1$ to a block lying entirely in the bottom row of $\pi_2$.    Since this equality is true entrywise for all $(\rs|\rs') \in [1,n]^{2k}$,  the desired result follows.  \qed

\begin{example}\label{Example:LabeledDiagramProof}  We illustrate the  labeled diagrams of $\left(X_{\pi_1}X_{\pi_2}\right)_{\rs}^{\rs'}$ in the proof of Lemma \ref{T:mult} for the specific case of Example \ref{exs:mult} (2). 
$$
\begin{array}{l}
\begin{array}{c}
\begin{tikzpicture}[scale=.6,line width=1.25pt] 
\foreach \i in {1,...,4} 
{\path (\i,1) coordinate (T\i); \path (\i,0) coordinate (B\i); } 
\filldraw[fill= gray!50,draw=gray!50,line width=4pt]  (T1) -- (T4) -- (B4) -- (B1) -- (T1);
\draw[blue] (T2) .. controls +(0,-.30) and +(0,-.30) .. (T3);
\draw[blue] (B3) .. controls +(0,+.30) and +(0,+.30) .. (B4);
\draw[blue] (B2) -- (T4);; 
\draw  (T1)  node[black,above=0.01cm]{${r_1'}$};
\draw  (T2)  node[black,above=0.01cm]{${r_2'}$};
\draw  (T3)  node[black,above=0.01cm]{${r_2'}$};
\draw  (T4)  node[black,above=0.01cm]{${r_2}$};
\draw  (B1)  node[black,below=0.01cm]{${q_1}$};
\draw  (B2)  node[black,below=0.01cm]{${r_2}$};
\draw  (B3)  node[black,below=0.01cm]{${q_2}$};
\draw  (B4)  node[black,below=0.01cm]{${q_2}$};
\foreach \i in {1,...,4} 
{\filldraw[fill=white,draw=black,line width = 1pt] (T\i) circle (3pt); \filldraw[fill=white,draw=black,line width = 1pt] (B\i) circle (3pt); } 
\foreach \i in {1,...,4} {\path (\i,-1) coordinate (TT\i); \path (\i,-2) coordinate (BB\i); } 
\filldraw[fill= gray!50,draw=gray!50,line width=4pt]  (TT1) -- (TT4) -- (BB4) -- (BB1) -- (TT1);
\draw[blue] (TT3) .. controls +(0,-.30) and +(0,-.30) .. (TT4);
\draw[blue] (BB1) .. controls +(0,+.40) and +(0,+.40) .. (BB3);
\draw[blue] (TT2) .. controls +(0,+.40) and +(0,+.40) .. (BB2);
\draw  (BB1)  node[black,below=0.01cm]{${r_1}$};
\draw  (BB2)  node[black,below=0.01cm]{${r_2}$};
\draw  (BB3)  node[black,below=0.01cm]{${r_1}$};
\draw  (BB4)  node[black,below=0.01cm]{${r_3}$};
\foreach \i in {1,...,4} 
{\filldraw[fill=white,draw=black,line width = 1pt] (TT\i) circle (3pt); \filldraw[fill=white,draw=black,line width = 1pt] (BB\i) circle (3pt); } 
\end{tikzpicture} 
\end{array}
=  (n-5)(n-6)
\begin{array}{c}
\begin{tikzpicture}[scale=.6,line width=1.25pt] 
\foreach \i in {1,...,4} 
{ \path (\i,1) coordinate (T\i); \path (\i,0) coordinate (B\i); } 
\filldraw[fill= gray!50,draw=gray!50,line width=4pt]  (T1) -- (T4) -- (B4) -- (B1) -- (T1);
\draw[blue] (T2) .. controls +(0,-.30) and +(0,-.30) .. (T3);
\draw[blue] (B1) .. controls +(0,+.40) and +(0,+.40) .. (B3);
\draw[blue] (T4) -- (B2);
\draw  (T1)  node[black,above=0.01cm]{${r_1'}$};
\draw  (T2)  node[black,above=0.01cm]{${r_2'}$};
\draw  (T3)  node[black,above=0.01cm]{${r_2'}$};
\draw  (T4)  node[black,above=0.01cm]{${r_2}$};
\draw  (B1)  node[black,below=0.01cm]{${r_1}$};
\draw  (B2)  node[black,below=0.01cm]{${r_2}$};
\draw  (B3)  node[black,below=0.01cm]{${r_1}$};
\draw  (B4)  node[black,below=0.01cm]{${r_3}$};
\foreach \i in {1,...,4} 
{\filldraw[fill=white,draw=black,line width = 1pt] (T\i) circle (3pt); \filldraw[fill=white,draw=black,line width = 1pt] (B\i) circle (3pt); } 
\end{tikzpicture}
\end{array}  \\
\hskip0.0in
+ (n-4)(n-5) 
\Big(
\begin{array}{c}
\begin{tikzpicture}[scale=.6,line width=1.25pt] 
\foreach \i in {1,...,4} 
{ \path (\i,1) coordinate (T\i); \path (\i,0) coordinate (B\i); } 
\filldraw[fill= gray!50,draw=gray!50,line width=4pt]  (T1) -- (T4) -- (B4) -- (B1) -- (T1);
\draw[blue] (T2) .. controls +(0,-.30) and +(0,-.30) .. (T3);
\draw[blue] (B1) .. controls +(0,+.40) and +(0,+.40) .. (B3);
\draw[purple] (T1)--(B1);
\draw[blue] (T4) -- (B2);
\draw  (T1)  node[black,above=0.01cm]{${r_1}$};
\draw  (T2)  node[black,above=0.01cm]{${r_2'}$};
\draw  (T3)  node[black,above=0.01cm]{${r_2'}$};
\draw  (T4)  node[black,above=0.01cm]{${r_2}$};
\draw  (B1)  node[black,below=0.01cm]{${r_1}$};
\draw  (B2)  node[black,below=0.01cm]{${r_2}$};
\draw  (B3)  node[black,below=0.01cm]{${r_1}$};
\draw  (B4)  node[black,below=0.01cm]{${r_3}$};
\foreach \i in {1,...,4} 
{\filldraw[fill=white,draw=black,line width = 1pt] (T\i) circle (3pt); \filldraw[fill=white,draw=black,line width = 1pt] (B\i) circle (3pt); } 
\end{tikzpicture}
\end{array}
\!\!+\!\!
\begin{array}{c}
\begin{tikzpicture}[scale=.6,line width=1.25pt] 
\foreach \i in {1,...,4} 
{ \path (\i,1) coordinate (T\i); \path (\i,0) coordinate (B\i); } 
\filldraw[fill= gray!50,draw=gray!50,line width=4pt]  (T1) -- (T4) -- (B4) -- (B1) -- (T1);
\draw[blue] (T2) .. controls +(0,-.30) and +(0,-.30) .. (T3);
\draw[blue] (B1) .. controls +(0,+.40) and +(0,+.40) .. (B3);
\draw[purple] (T1)--(B4);
\draw[blue] (T4) -- (B2);
\draw  (T1)  node[black,above=0.01cm]{${r_3}$};
\draw  (T2)  node[black,above=0.01cm]{${r_2'}$};
\draw  (T3)  node[black,above=0.01cm]{${r_2'}$};
\draw  (T4)  node[black,above=0.01cm]{${r_2}$};
\draw  (B1)  node[black,below=0.01cm]{${r_1}$};
\draw  (B2)  node[black,below=0.01cm]{${r_2}$};
\draw  (B3)  node[black,below=0.01cm]{${r_1}$};
\draw  (B4)  node[black,below=0.01cm]{${r_3}$};
\foreach \i in {1,...,4} 
{\filldraw[fill=white,draw=black,line width = 1pt](T\i) circle (3pt); \filldraw[fill=white,draw=black,line width = 1pt](B\i) circle (3pt); } 
\end{tikzpicture}
\end{array}
\!\!+\!\!
\begin{array}{c}
\begin{tikzpicture}[scale=.6,line width=1.25pt] 
\foreach \i in {1,...,4} 
{ \path (\i,1) coordinate (T\i); \path (\i,0) coordinate (B\i); } 
\filldraw[fill= gray!50,draw=gray!50,line width=4pt]  (T1) -- (T4) -- (B4) -- (B1) -- (T1);
\draw[blue] (T2) .. controls +(0,-.30) and +(0,-.30) .. (T3);
\draw[blue] (B1) .. controls +(0,+.40) and +(0,+.40) .. (B3);
\draw[purple] (T2).. controls +(0,-.30) and +(0,+.30) ..(B1);
\draw[blue] (T4) -- (B2);
\draw  (T1)  node[black,above=0.01cm]{${r_1'}$};
\draw  (T2)  node[black,above=0.01cm]{${r_1}$};
\draw  (T3)  node[black,above=0.01cm]{${r_1}$};
\draw  (T4)  node[black,above=0.01cm]{${r_2}$};
\draw  (B1)  node[black,below=0.01cm]{${r_1}$};
\draw  (B2)  node[black,below=0.01cm]{${r_2}$};
\draw  (B3)  node[black,below=0.01cm]{${r_1}$};
\draw  (B4)  node[black,below=0.01cm]{${r_3}$};
\foreach \i in {1,...,4} 
{ \filldraw[fill=white,draw=black,line width = 1pt] (T\i) circle (3pt); \filldraw[fill=white,draw=black,line width = 1pt](B\i) circle (3pt); } 
\end{tikzpicture}
\end{array}\!\!+\!\!
\begin{array}{c}
\begin{tikzpicture}[scale=.6,line width=1.25pt] 
\foreach \i in {1,...,4} 
{ \path (\i,1) coordinate (T\i); \path (\i,0) coordinate (B\i); } 
\filldraw[fill= gray!50,draw=gray!50,line width=4pt]  (T1) -- (T4) -- (B4) -- (B1) -- (T1);
\draw[blue] (T2) .. controls +(0,-.30) and +(0,-.30) .. (T3);
\draw[blue] (B1) .. controls +(0,+.40) and +(0,+.40) .. (B3);
\draw[blue] (T4) -- (B2);
\draw[purple] (T3).. controls +(0,-.30) and +(0,+.30) ..(B4);
\draw  (T1)  node[black,above=0.01cm]{${r_1'}$};
\draw  (T2)  node[black,above=0.01cm]{${r_3}$};
\draw  (T3)  node[black,above=0.01cm]{${r_3}$};
\draw  (T4)  node[black,above=0.01cm]{${r_2}$};
\draw  (B1)  node[black,below=0.01cm]{${r_1}$};
\draw  (B2)  node[black,below=0.01cm]{${r_2}$};
\draw  (B3)  node[black,below=0.01cm]{${r_1}$};
\draw  (B4)  node[black,below=0.01cm]{${r_3}$};
\foreach \i in {1,...,4} 
{\filldraw[fill=white,draw=black,line width = 1pt](T\i) circle (3pt); \filldraw[fill=white,draw=black,line width = 1pt](B\i) circle (3pt); } 
\end{tikzpicture}
\end{array}\!\!\Big) \\
\hskip0.0in+(n-3)(n-4) \Big(
\begin{array}{c}
\begin{tikzpicture}[scale=.6,line width=1.25pt] 
\foreach \i in {1,...,4} 
{ \path (\i,1) coordinate (T\i); \path (\i,0) coordinate (B\i); } 
\filldraw[fill= gray!50,draw=gray!50,line width=4pt]  (T1) -- (T4) -- (B4) -- (B1) -- (T1);
\draw[blue] (T2) .. controls +(0,-.30) and +(0,-.30) .. (T3);
\draw[blue] (B1) .. controls +(0,+.40) and +(0,+.40) .. (B3);
\draw[purple] (T1)--(B1);
\draw[purple] (T3).. controls +(0,-.30) and +(0,+.30) ..(B4);;
\draw[blue] (T4) -- (B2);
\draw  (T1)  node[black,above=0.01cm]{${r_1}$};
\draw  (T2)  node[black,above=0.01cm]{${r_3}$};
\draw  (T3)  node[black,above=0.01cm]{${r_3}$};
\draw  (T4)  node[black,above=0.01cm]{${r_2}$};
\draw  (B1)  node[black,below=0.01cm]{${r_1}$};
\draw  (B2)  node[black,below=0.01cm]{${r_2}$};
\draw  (B3)  node[black,below=0.01cm]{${r_1}$};
\draw  (B4)  node[black,below=0.01cm]{${r_3}$};
\foreach \i in {1,...,4} 
{\filldraw[fill=white,draw=black,line width = 1pt](T\i) circle (3pt); \filldraw[fill=white,draw=black,line width = 1pt] (B\i) circle (3pt); } 
\end{tikzpicture}
\end{array}+
\begin{array}{c}
\begin{tikzpicture}[scale=.6,line width=1.25pt] 
\foreach \i in {1,...,4} 
{ \path (\i,1) coordinate (T\i); \path (\i,0) coordinate (B\i); } 
\filldraw[fill= gray!50,draw=gray!50,line width=4pt]  (T1) -- (T4) -- (B4) -- (B1) -- (T1);
\draw[blue] (T2) .. controls +(0,-.30) and +(0,-.30) .. (T3);
\draw[blue] (B1) .. controls +(0,+.40) and +(0,+.40) .. (B3);
\draw[purple] (T1)--(B4);
\draw[purple] (T2).. controls +(0,-.30) and +(0,+.30) ..(B1);
\draw[blue] (T4) -- (B2);
\draw  (T1)  node[black,above=0.01cm]{${r_3}$};
\draw  (T2)  node[black,above=0.01cm]{${r_1}$};
\draw  (T3)  node[black,above=0.01cm]{${r_1}$};
\draw  (T4)  node[black,above=0.01cm]{${r_2}$};
\draw  (B1)  node[black,below=0.01cm]{${r_1}$};
\draw  (B2)  node[black,below=0.01cm]{${r_2}$};
\draw  (B3)  node[black,below=0.01cm]{${r_1}$};
\draw  (B4)  node[black,below=0.01cm]{${r_3}$};
\foreach \i in {1,...,4} 
{\filldraw[fill=white,draw=black,line width = 1pt] (T\i) circle (3pt); \filldraw[fill=white,draw=black,line width = 1pt] (B\i) circle (3pt); } 
\end{tikzpicture}
\end{array}\Big).
\end{array}
$$
\end{example} 
 
\begin{thm}\label{C:mult} 
Assume  $\pi_1, \pi_2 \in \Pi_{2k}$, and let $x_{\pi_1}, x_{\pi_2}$ be the
corresponding orbit basis elements in $\P_k(\para)$ for $\para \in \CC\setminus \{0\}$.   Then
\begin{equation}\label{eq:orbitmult}  x_{\pi_1} x_{\pi_2} = \begin{cases}
\displaystyle{\sum_{\vr}   (\para-|\varrho|)_{[\pi_1\ast \pi_2]} \  x_\varrho} &  \quad \hbox{if $\pi_1 \ast \pi_2$ exactly matches in the middle,}\\  
0 &\quad \hbox{otherwise,} 
\end{cases}\end{equation}
where  the sum is over all coarsenings $\vr$ of $\pi_1 \ast \pi_2$ obtained by connecting blocks that lie entirely in the top row of $\pi_1$
to blocks that lie entirely in the bottom row of $\pi_2$.  
\end{thm} 
 
\begin{proof}  Consider first $\P_k(n)$ with $n \geq 2k$.  Then   
$x_{\pi_1}x_{\pi_2}- \sum_{\vr} (n-|\vr|)_{[\pi_1 \ast \pi_2]} \, x_\vr$ (where the sum is over the coarsenings $\vr$ of $\pi_1 \ast \pi_2$ 
as in the statement of the corollary)  lies in the
kernel of the representation $\Phi_{k,n}$, which equals $(0)$ by Theorem \ref{T:Phi}.  Thus,
\eqref{eq:orbitmult} holds for the partition algebras $\P_k(n)$ for all $n \geq 2k$.  
More generally, when $\xi \in \CC\setminus \{0\}$ and we multiply two orbit
basis elements in $\P_k(\xi)$,  we get a linear combination of orbit basis diagrams  $x_\varrho$
whose coefficients are integer combinations of powers of  $\xi$.  
Assuming for the moment that $\xi$ is an indeterminate, we have
that the coefficient of $x_\varrho$ is a polynomial in $\ZZ[\xi]$.   
Moreover,  when  $\xi = n \ge 2k$,  that polynomial is $(n-|\varrho|)_{[\pi_1*\pi_2]}$ if
$\vr$ is an appropriate kind of coarsening of $\pi_1 \ast \pi_2$  or it is 0.
These two polynomials agree on infinitely many values, and so they
must be equal since the field has characteristic 0. Therefore the result must 
hold for all nonzero values $\xi$. \end{proof}  

\subsection{Rook diagrams and permutation diagrams}\label{subsec:rook} 
  
A set partition $\pi \in \Pi_{2k}$ is a \emph{rook partition} if $\pi$ consists of blocks of size one and two such that the blocks of size two in $\pi$ contain one element from the bottom row $\{1, 2, \ldots, k\}$ and one element from the top row $\{k+1,k+2, \ldots, 2k\}$.  Let $\mathcal{R}_{2k} \subseteq \Pi_{2k}$ denote the subset of rook partitions.  It is easy to check that rook partitions are characterized by the property  $\pn(\pi) + |\pi| = 2k.$

A set partition $\pi\in\Pi_{2k}$ is a \emph{permutation} if it consists of $k$ blocks, each of size two, with exactly one element from the bottom row $\{1,2,\ldots, k\}$ and one element from the top row $\{k+1,k+2, \ldots, 2k\}$ in each block.  Let $\mathcal{S}_{2k} \subseteq \Pi_{2k}$ denote the subset of permutations. The corresponding set of permutation diagrams 
$\{d_\pi \mid \pi \in \mathcal{S}_{2k}\}$ is isomorphic to the symmetric group $\S_k$  under diagram multiplication. The permutation $\sigma \in \S_k$ corresponds to the set partition $\pi_{\sigma} = \{1, \sigma(1) \,|\,  2, \sigma(2) \,|\  \cdots \ |\,  k, \sigma(k)\}$, and we identify $\sigma$ with its diagram $\sigma = d_{\pi_\sigma}$.   By
 \eqref{eq:Phi-coeffs-diagram},  $\sigma$ acts on  $\us_{\mathsf{r}} = \us_{r_1} \ot \us_{r_2} \ot \cdots \ot \us_{r_k} \in \modu^{\ot k}$  as follows:\,
$\sigma \us_\mathsf{r} = \sum_{\mathsf{s} \in [1,n]^k} \mathsf{E}_{\mathsf{s}}^{\mathsf{\sigma(s)}}\us_{\mathsf{r}} = 
 \us_{\sigma(r_1)} \ot \us_{\sigma(r_2)} \ot \cdots \ot \us_{\sigma(r_k)}.$  
Permutations $\pi$ are characterized by the property that $\pn(\pi) = |\pi| = k$.

If $\pi \in \Pi_{2k}$ and $\sigma, \sigma' \in \S_k$,  then with the above identifications,  $\sigma' d_\pi \sigma = d_{\sigma'\ast\pi\ast\sigma}$
under diagram multiplication,  where $\sigma'\ast\pi\ast\sigma \in \Pi_{2k}$ is the set partition obtained from $\pi$ by applying $\sigma$ to $\{1,2, \ldots, k\}$ in the bottom row of $\pi$ and by applying $\sigma'$ to $\{k+1, k+2,\ldots, 2k\}$ in the top row of $\pi$ by sending $k + i$ to $k + \sigma'(i)$. In other words, $\sigma'$ permutes the vertices in the top row of $d_\pi$ and $\sigma$ permutes the vertices in the bottom row while maintaining the edge connections.  
This multiplication extends to orbit diagrams,  since 
$$
\sigma' x_\pi \sigma =  \sum_{\pi \preceq \varrho} \mu_{2k}(\pi,\vr) \sigma' d_\varrho \sigma = \sum_{\sigma'\ast \pi \ast\sigma \preceq \sigma'\ast \vr\ast \sigma } \mu_{2k}(\sigma'\ast \pi \ast\sigma, \sigma'\ast \vr\ast \sigma) d_{\sigma' \ast \vr \ast \sigma} = x_{\sigma' \ast \pi \ast \sigma}.$$   For example,
\begin{equation}\label{DiagramPermutation}
\begin{array}{r}
\sigma' = \begin{array}{c}\begin{tikzpicture}[scale=.6,line width=1.25pt] 
\foreach \i in {1,...,6} 
{ \path (\i,1) coordinate (T\i); \path (\i,0) coordinate (B\i); } 
\filldraw[fill= gray!50,draw=gray!50,line width=4pt]  (T1) -- (T6) -- (B6) -- (B1) -- (T1);
\draw[blue] (T1) -- (B2);
\draw[blue] (T2) -- (B4);
\draw[blue] (T3) -- (B1);
\draw[blue] (T4) -- (B6);
\draw[blue] (T5) -- (B5);
\draw[blue] (T6) -- (B3);
\foreach \i in {1,...,6} 
{\filldraw[fill=black,draw=black,line width = 1pt] (T\i) circle (3pt); \filldraw[fill=black,draw=black,line width = 1pt](B\i) circle (3pt); } 
\end{tikzpicture}\end{array}  \\
x_{\pi}=\begin{array}{c}\begin{tikzpicture}[scale=.6,line width=1.25pt] 
\foreach \i in {1,...,6} 
{ \path (\i,1) coordinate (T\i); \path (\i,0) coordinate (B\i); } 
\filldraw[fill= gray!50,draw=gray!50,line width=4pt]  (T1) -- (T6) -- (B6) -- (B1) -- (T1);
\draw[blue] (T2) .. controls +(0,-.50) and +(0,-.50) .. (T4);
\draw[blue] (T3) .. controls +(0,-.40) and +(0,-.40) .. (T5);
\draw[blue] (B2) .. controls +(0,+.30) and +(0,+.30) .. (B3);
\draw[blue] (B4) .. controls +(0,+.30) and +(0,+.30) .. (B5);
\draw[blue] (T1) -- (B2);\draw[blue] (T5) -- (B6);
\foreach \i in {1,...,6} 
{\filldraw[fill=white,draw=black,line width = 1pt] (T\i) circle (3pt); \filldraw[fill=white,draw=black,line width = 1pt](B\i) circle (3pt); } 
\end{tikzpicture}\end{array}  \\
\sigma=\begin{array}{c}\begin{tikzpicture}[scale=.6,line width=1.25pt] 
\foreach \i in {1,...,6} 
{ \path (\i,1) coordinate (T\i); \path (\i,0) coordinate (B\i); } 
\filldraw[fill= gray!50,draw=gray!50,line width=4pt]  (T1) -- (T6) -- (B6) -- (B1) -- (T1);
\draw[blue] (T1) -- (B2);
\draw[blue] (T2) -- (B3);
\draw[blue] (T3) -- (B1);
\draw[blue] (T4) -- (B5);
\draw[blue] (T5) -- (B4);
\draw[blue] (T6) -- (B6);
\foreach \i in {1,...,6} 
{\filldraw[fill=black,draw=black,line width = 1pt] (T\i) circle (3pt); \filldraw[fill=black,draw=black,line width = 1pt](B\i) circle (3pt); } 
\end{tikzpicture}\end{array}  \\
\end{array} 
= 
\begin{array}{c}\begin{tikzpicture}[scale=.6,line width=1.25pt] 
\foreach \i in {1,...,6} 
{ \path (\i,1) coordinate (T\i); \path (\i,0) coordinate (B\i); } 
\filldraw[fill= gray!50,draw=gray!50,line width=4pt]  (T1) -- (T6) -- (B6) -- (B1) -- (T1);
\draw[blue] (T1) .. controls +(0,-.40) and +(0,-.40) .. (T2);
\draw[blue] (T5) .. controls +(0,-.40) and +(0,-.40) .. (T6) -- (B6);
\draw[blue]  (B1) .. controls +(0,+.40) and +(0,+.40) .. (B3) -- (T3);
\draw[blue] (B4) .. controls +(0,+.40) and +(0,+.40) .. (B5);
\foreach \i in {1,...,6} 
{\filldraw[fill=white,draw=black,line width = 1pt] (T\i) circle (3pt); \filldraw[fill=white,draw=black,line width = 1pt](B\i) circle (3pt); } 
\end{tikzpicture}\end{array} 
= x_{\sigma'\ast\pi\ast\sigma}.
\end{equation}
Observe that we are multiplying $\sigma,\sigma'$ in the diagram basis with $x_\pi$ in the orbit basis.
 
\section{The Kernel of the Representation  $\Phi_{k,n}: \P_k(n) \to \End_{\S_n}(\modu^{\ot k})$}

The surjection $\Phi_{k,n}: \P_k(n) \rightarrow \End_{\S_n}(\modu^{\otimes k}) $ from Theorem \ref{T:Phi}\,(a)   
 is an isomorphism when $n \ge 2k$. This section is  dedicated to showing that  $\ker \Phi_{k,n}$
 is generated  as a two-sided ideal by a single  (essential) idempotent when $2k > n$.    
\subsection{The kernel is principally generated,  \, $\ker \Phi_{k,n} = \langle \ef_{k,n} \rangle$ when $2k > n$}

For  $n,k \in \ZZ_{\ge 1}$ with $2k > n$, define the following orbit basis elements,
\vspace{-.1cm}
\begin{align}\label{eq:ef}
\ef_{k,n}  &= \begin{cases}
\!\!\begin{array}{c} 
\begin{tikzpicture}[xscale=.5,yscale=.5,line width=1.25pt] 
\foreach \i in {1,2,3,4,5,6,7}  { \path (\i,1.25) coordinate (T\i); \path (\i,.25) coordinate (B\i); } 
\filldraw[fill= black!12,draw=black!12,line width=4pt]  (T1) -- (T7) -- (B7) -- (B1) -- (T1);
\draw[blue] (T4) -- (B4);
\draw[blue] (T5) -- (B5);
\draw[blue] (T7) -- (B7);
\draw (T2) node {$\cdots$};\draw (B2) node {$\cdots$};
\draw (T6) node {$\cdots$};\draw (B6) node {$\cdots$};
\draw (2,-.5) node {$\underbrace{\phantom{iiiiiiiii}}_{n+1-k}$};
\draw (5.5,-.5) node {$\underbrace{\phantom{iiiiiiiiiiiii}}_{2k-n-1}$};
\foreach \i in {1,3,4,5,7}  { \filldraw[fill=white,draw=black,line width = 1pt] (T\i) circle (4pt); \filldraw[fill=white,draw=black,line width = 1pt]  (B\i) circle (4pt); } 
\end{tikzpicture}
\end{array} & \quad \text{if \ $n \ge k > n/2$,}\\ 
\begin{array}{c} 
\begin{tikzpicture}[xscale=.5,yscale=.5,line width=1.25pt] 
\foreach \i in {1,2,3,4,5,6,7}  { \path (\i,1.25) coordinate (T\i); \path (\i,.25) coordinate (B\i); } 
\filldraw[fill= black!12,draw=black!12,line width=4pt]  (T1) -- (T7) -- (B7) -- (B1) -- (T1);
\draw[blue] (T1) -- (B1);
\draw[blue] (T2) -- (B2);
\draw[blue] (T3) -- (B3);
\draw[blue] (T4) -- (B4);
\draw (T5) node {$\cdots$};\draw (B5) node {$\cdots$};
\draw (T6) node {$\cdots$};\draw (B6) node {$\cdots$};
\draw[blue] (T7) -- (B7);
\draw (4,-.5) node {$\underbrace{\phantom{iiiiiiiiiiiiiiiiiiiiiiii}}_{k}$};
\foreach \i in {1,2,3,4,7}  { \filldraw[fill=white,draw=black,line width = 1pt] (T\i) circle (4pt); \filldraw[fill=white,draw=black,line width = 1pt]  (B\i) circle (4pt); } 
\end{tikzpicture} 
\end{array}&\quad  \text{if \ $k >n$.}
\end{cases}\\
\ef_{k-\half,n}  &= \ef_{k,n},  \hskip.2in  \text{if \  $ 2k - 1 >  n$}. \label{eq:efhalf} 
\end{align}
Observe that if $k>n$, then $\pn(\ef_{k,n}) = k$,  and  $\ef_{k,n}$  has $k$ blocks, which
we signify by writing $|\ef_{k,n}| = k$.      If $n \ge k > n/2$, then  
$\pn(\ef_{k,n}) = 2k-n-1$, \ $\ef_{k,n}$ has two rows each with $n+1-k$ isolated vertices, and the number of
blocks in $\ef_{k,n}$ is  $|\ef_{k,n}| =  2(n+1-k) + 2k-n-1 = n+1$.
For example, 
\vspace{-.17cm} 
\begin{center}{$
\ef_{4\half,6} = \ef_{5,6} =   \begin{array}{c} 
\begin{tikzpicture}[xscale=.5,yscale=.5,line width=1.25pt] 
\foreach \i in {1,2,3,4,5}  { \path (\i,1.25) coordinate (T\i); \path (\i,.25) coordinate (B\i); } 
\filldraw[fill= black!12,draw=black!12,line width=4pt]  (T1) -- (T5) -- (B5) -- (B1) -- (T1);
\draw[blue] (T3) -- (B3);
\draw[blue] (T4) -- (B4);
\draw[blue] (T5) -- (B5);
\foreach \i in {1,2,3,4,5}  { \filldraw[fill=white,draw=black,line width = 1pt] (T\i) circle (4pt); \filldraw[fill=white,draw=black,line width = 1pt]  (B\i) circle (4pt); } 
\end{tikzpicture} 
\end{array}
$} \end{center}
 has  $\pn(\ef_{5,6}) = 2\cdot 5 - 6 - 1 = 3$ and $|\ef_{5,6}| = 6 + 1 = 7$ blocks.
The elements $\ef_{k,n}$  for $k \le 5$ and $n \le 9$ are displayed
in Figure \ref{figure:generatortable}. 
\begin{figure}
\noindent $$\begin{array}{c|ccccccccc}
& k=1&k = 1 \frac{1}{2}&k=2&k = 2 \frac{1}{2}&k = 3&k = 3\frac{1}{2}&k = 4& k = 4\frac{1}{2}& k = 5\\
\hline
n = 1 & 
\vphantom{\bigg\vert}
\begin{array}{c} 
\begin{tikzpicture}[xscale=.28,yscale=.28,line width=1.25pt] 
\foreach \i in {1}  { \path (\i,1.25) coordinate (T\i); \path (\i,.25) coordinate (B\i); } 
\foreach \i in {1}  { \filldraw[fill=white,draw=black,line width = 1pt] (T\i) circle (4pt); \filldraw[fill=white,draw=black,line width = 1pt]  (B\i) circle (4pt); } 
\end{tikzpicture} 
\end{array}
&
\begin{array}{c} 
\begin{tikzpicture}[xscale=.28,yscale=.28,line width=1.25pt] 
\foreach \i in {1,2}  { \path (\i,1.25) coordinate (T\i); \path (\i,.25) coordinate (B\i); } 
\filldraw[fill= black!12,draw=black!12,line width=4pt]  (T1) -- (T2) -- (B2) -- (B1) -- (T1);
\draw[blue] (T1) -- (B1);
\draw[blue] (T2) -- (B2);
\foreach \i in {1,2}  { \filldraw[fill=white,draw=black,line width = 1pt] (T\i) circle (4pt); \filldraw[fill=white,draw=black,line width = 1pt]  (B\i) circle (4pt); } 
\end{tikzpicture} 
\end{array}
&
\begin{array}{c} 
\begin{tikzpicture}[xscale=.28,yscale=.28,line width=1.25pt] 
\foreach \i in {1,2}  { \path (\i,1.25) coordinate (T\i); \path (\i,.25) coordinate (B\i); } 
\filldraw[fill= black!12,draw=black!12,line width=4pt]  (T1) -- (T2) -- (B2) -- (B1) -- (T1);
\draw[blue] (T1) -- (B1);
\draw[blue] (T2) -- (B2);
\foreach \i in {1,2}  { \filldraw[fill=white,draw=black,line width = 1pt] (T\i) circle (4pt); \filldraw[fill=white,draw=black,line width = 1pt]  (B\i) circle (4pt); } 
\end{tikzpicture} 
\end{array}
&
\begin{array}{c} 
\begin{tikzpicture}[xscale=.28,yscale=.28,line width=1.25pt] 
\foreach \i in {1,2,3}  { \path (\i,1.25) coordinate (T\i); \path (\i,.25) coordinate (B\i); } 
\filldraw[fill= black!12,draw=black!12,line width=4pt]  (T1) -- (T3) -- (B3) -- (B1) -- (T1);
\draw[blue] (T1) -- (B1);
\draw[blue] (T2) -- (B2);
\draw[blue] (T3) -- (B3);
\foreach \i in {1,2,3}  { \filldraw[fill=white,draw=black,line width = 1pt] (T\i) circle (4pt); \filldraw[fill=white,draw=black,line width = 1pt]  (B\i) circle (4pt); } 
\end{tikzpicture} 
\end{array}
&
\begin{array}{c} 
\begin{tikzpicture}[xscale=.28,yscale=.28,line width=1.25pt] 
\foreach \i in {1,2,3}  { \path (\i,1.25) coordinate (T\i); \path (\i,.25) coordinate (B\i); } 
\filldraw[fill= black!12,draw=black!12,line width=4pt]  (T1) -- (T3) -- (B3) -- (B1) -- (T1);
\draw[blue] (T1) -- (B1);
\draw[blue] (T2) -- (B2);
\draw[blue] (T3) -- (B3);
\foreach \i in {1,2,3}  { \filldraw[fill=white,draw=black,line width = 1pt] (T\i) circle (4pt); \filldraw[fill=white,draw=black,line width = 1pt]  (B\i) circle (4pt); } 
\end{tikzpicture} 
\end{array}
&
\begin{array}{c} 
\begin{tikzpicture}[xscale=.28,yscale=.28,line width=1.25pt] 
\foreach \i in {1,2,3,4}  { \path (\i,1.25) coordinate (T\i); \path (\i,.25) coordinate (B\i); } 
\filldraw[fill= black!12,draw=black!12,line width=4pt]  (T1) -- (T4) -- (B4) -- (B1) -- (T1);
\draw[blue] (T1) -- (B1);
\draw[blue] (T2) -- (B2);
\draw[blue] (T3) -- (B3);
\draw[blue] (T4) -- (B4);
\foreach \i in {1,2,3,4}  { \filldraw[fill=white,draw=black,line width = 1pt] (T\i) circle (4pt); \filldraw[fill=white,draw=black,line width = 1pt]  (B\i) circle (4pt); } 
\end{tikzpicture} 
\end{array}
&
\begin{array}{c} 
\begin{tikzpicture}[xscale=.28,yscale=.28,line width=1.25pt] 
\foreach \i in {1,2,3,4}  { \path (\i,1.25) coordinate (T\i); \path (\i,.25) coordinate (B\i); } 
\filldraw[fill= black!12,draw=black!12,line width=4pt]  (T1) -- (T4) -- (B4) -- (B1) -- (T1);
\draw[blue] (T1) -- (B1);
\draw[blue] (T2) -- (B2);
\draw[blue] (T3) -- (B3);
\draw[blue] (T4) -- (B4);
\foreach \i in {1,2,3,4}  { \filldraw[fill=white,draw=black,line width = 1pt] (T\i) circle (4pt); \filldraw[fill=white,draw=black,line width = 1pt]  (B\i) circle (4pt); } 
\end{tikzpicture} 
\end{array}
&
\begin{array}{c} 
\begin{tikzpicture}[xscale=.28,yscale=.28,line width=1.25pt] 
\foreach \i in {1,2,3,4,5}  { \path (\i,1.25) coordinate (T\i); \path (\i,.25) coordinate (B\i); } 
\filldraw[fill= black!12,draw=black!12,line width=4pt]  (T1) -- (T5) -- (B5) -- (B1) -- (T1);
\draw[blue] (T1) -- (B1);
\draw[blue] (T2) -- (B2);
\draw[blue] (T3) -- (B3);
\draw[blue] (T4) -- (B4);
\draw[blue] (T5) -- (B5);
\foreach \i in {1,2,3,4,5}  { \filldraw[fill=white,draw=black,line width = 1pt] (T\i) circle (4pt); \filldraw[fill=white,draw=black,line width = 1pt]  (B\i) circle (4pt); } 
\end{tikzpicture} 
\end{array}
&
\begin{array}{c} 
\begin{tikzpicture}[xscale=.28,yscale=.28,line width=1.25pt] 
\foreach \i in {1,2,3,4,5}  { \path (\i,1.25) coordinate (T\i); \path (\i,.25) coordinate (B\i); } 
\filldraw[fill= black!12,draw=black!12,line width=4pt]  (T1) -- (T5) -- (B5) -- (B1) -- (T1);
\draw[blue] (T1) -- (B1);
\draw[blue] (T2) -- (B2);
\draw[blue] (T3) -- (B3);
\draw[blue] (T4) -- (B4);
\draw[blue] (T5) -- (B5);
\foreach \i in {1,2,3,4,5}  { \filldraw[fill=white,draw=black,line width = 1pt] (T\i) circle (4pt); \filldraw[fill=white,draw=black,line width = 1pt]  (B\i) circle (4pt); } 
\end{tikzpicture} 
\end{array}
\\
n = 2 
& 
\vphantom{\bigg\vert}
&
\begin{array}{c} 
\begin{tikzpicture}[xscale=.28,yscale=.28,line width=1.25pt] 
\foreach \i in {1,2}  { \path (\i,1.25) coordinate (T\i); \path (\i,.25) coordinate (B\i); } 
\filldraw[fill= black!12,draw=black!12,line width=4pt]  (T1) -- (T2) -- (B2) -- (B1) -- (T1);
\draw[blue] (T2) -- (B2);
\foreach \i in {1,2}  { \filldraw[fill=white,draw=black,line width = 1pt] (T\i) circle (4pt); \filldraw[fill=white,draw=black,line width = 1pt]  (B\i) circle (4pt); } 
\end{tikzpicture} 
\end{array}
&
\begin{array}{c} 
\begin{tikzpicture}[xscale=.28,yscale=.28,line width=1.25pt] 
\foreach \i in {1,2}  { \path (\i,1.25) coordinate (T\i); \path (\i,.25) coordinate (B\i); } 
\filldraw[fill= black!12,draw=black!12,line width=4pt]  (T1) -- (T2) -- (B2) -- (B1) -- (T1);
\draw[blue] (T2) -- (B2);
\foreach \i in {1,2}  { \filldraw[fill=white,draw=black,line width = 1pt] (T\i) circle (4pt); \filldraw[fill=white,draw=black,line width = 1pt]  (B\i) circle (4pt); } 
\end{tikzpicture} 
\end{array}
&
\begin{array}{c} 
\begin{tikzpicture}[xscale=.28,yscale=.28,line width=1.25pt] 
\foreach \i in {1,2,3}  { \path (\i,1.25) coordinate (T\i); \path (\i,.25) coordinate (B\i); } 
\filldraw[fill= black!12,draw=black!12,line width=4pt]  (T1) -- (T3) -- (B3) -- (B1) -- (T1);
\draw[blue] (T1) -- (B1);
\draw[blue] (T2) -- (B2);
\draw[blue] (T3) -- (B3);
\foreach \i in {1,2,3}  { \filldraw[fill=white,draw=black,line width = 1pt] (T\i) circle (4pt); \filldraw[fill=white,draw=black,line width = 1pt]  (B\i) circle (4pt); } 
\end{tikzpicture} 
\end{array}
&
\begin{array}{c} 
\begin{tikzpicture}[xscale=.28,yscale=.28,line width=1.25pt] 
\foreach \i in {1,2,3}  { \path (\i,1.25) coordinate (T\i); \path (\i,.25) coordinate (B\i); } 
\filldraw[fill= black!12,draw=black!12,line width=4pt]  (T1) -- (T3) -- (B3) -- (B1) -- (T1);
\draw[blue] (T1) -- (B1);
\draw[blue] (T2) -- (B2);
\draw[blue] (T3) -- (B3);
\foreach \i in {1,2,3}  { \filldraw[fill=white,draw=black,line width = 1pt] (T\i) circle (4pt); \filldraw[fill=white,draw=black,line width = 1pt]  (B\i) circle (4pt); } 
\end{tikzpicture} 
\end{array}
&
\begin{array}{c} 
\begin{tikzpicture}[xscale=.28,yscale=.28,line width=1.25pt] 
\foreach \i in {1,2,3,4}  { \path (\i,1.25) coordinate (T\i); \path (\i,.25) coordinate (B\i); } 
\filldraw[fill= black!12,draw=black!12,line width=4pt]  (T1) -- (T4) -- (B4) -- (B1) -- (T1);
\draw[blue] (T1) -- (B1);
\draw[blue] (T2) -- (B2);
\draw[blue] (T3) -- (B3);
\draw[blue] (T4) -- (B4);
\foreach \i in {1,2,3,4}  { \filldraw[fill=white,draw=black,line width = 1pt] (T\i) circle (4pt); \filldraw[fill=white,draw=black,line width = 1pt]  (B\i) circle (4pt); } 
\end{tikzpicture} 
\end{array}
&
\begin{array}{c} 
\begin{tikzpicture}[xscale=.28,yscale=.28,line width=1.25pt] 
\foreach \i in {1,2,3,4}  { \path (\i,1.25) coordinate (T\i); \path (\i,.25) coordinate (B\i); } 
\filldraw[fill= black!12,draw=black!12,line width=4pt]  (T1) -- (T4) -- (B4) -- (B1) -- (T1);
\draw[blue] (T1) -- (B1);
\draw[blue] (T2) -- (B2);
\draw[blue] (T3) -- (B3);
\draw[blue] (T4) -- (B4);
\foreach \i in {1,2,3,4}  { \filldraw[fill=white,draw=black,line width = 1pt] (T\i) circle (4pt); \filldraw[fill=white,draw=black,line width = 1pt]  (B\i) circle (4pt); } 
\end{tikzpicture} 
\end{array}
&
\begin{array}{c} 
\begin{tikzpicture}[xscale=.28,yscale=.28,line width=1.25pt] 
\foreach \i in {1,2,3,4,5}  { \path (\i,1.25) coordinate (T\i); \path (\i,.25) coordinate (B\i); } 
\filldraw[fill= black!12,draw=black!12,line width=4pt]  (T1) -- (T5) -- (B5) -- (B1) -- (T1);
\draw[blue] (T1) -- (B1);
\draw[blue] (T2) -- (B2);
\draw[blue] (T3) -- (B3);
\draw[blue] (T4) -- (B4);
\draw[blue] (T5) -- (B5);
\foreach \i in {1,2,3,4,5}  { \filldraw[fill=white,draw=black,line width = 1pt] (T\i) circle (4pt); \filldraw[fill=white,draw=black,line width = 1pt]  (B\i) circle (4pt); } 
\end{tikzpicture} 
\end{array}
&
\begin{array}{c} 
\begin{tikzpicture}[xscale=.28,yscale=.28,line width=1.25pt] 
\foreach \i in {1,2,3,4,5}  { \path (\i,1.25) coordinate (T\i); \path (\i,.25) coordinate (B\i); } 
\filldraw[fill= black!12,draw=black!12,line width=4pt]  (T1) -- (T5) -- (B5) -- (B1) -- (T1);
\draw[blue] (T1) -- (B1);
\draw[blue] (T2) -- (B2);
\draw[blue] (T3) -- (B3);
\draw[blue] (T4) -- (B4);
\draw[blue] (T5) -- (B5);
\foreach \i in {1,2,3,4,5}  { \filldraw[fill=white,draw=black,line width = 1pt] (T\i) circle (4pt); \filldraw[fill=white,draw=black,line width = 1pt]  (B\i) circle (4pt); } 
\end{tikzpicture} 
\end{array}
\\
n = 3
& 
\vphantom{\bigg\vert}
&
&
\begin{array}{c} 
\begin{tikzpicture}[xscale=.28,yscale=.28,line width=1.25pt] 
\foreach \i in {1,2}  { \path (\i,1.25) coordinate (T\i); \path (\i,.25) coordinate (B\i); } 
\filldraw[fill= black!12,draw=black!12,line width=4pt]  (T1) -- (T2) -- (B2) -- (B1) -- (T1);
\foreach \i in {1,2}  { \filldraw[fill=white,draw=black,line width = 1pt] (T\i) circle (4pt); \filldraw[fill=white,draw=black,line width = 1pt]  (B\i) circle (4pt); } 
\end{tikzpicture} 
\end{array}
&
\begin{array}{c} 
\begin{tikzpicture}[xscale=.28,yscale=.28,line width=1.25pt] 
\foreach \i in {1,2,3}  { \path (\i,1.25) coordinate (T\i); \path (\i,.25) coordinate (B\i); } 
\filldraw[fill= black!12,draw=black!12,line width=4pt]  (T1) -- (T3) -- (B3) -- (B1) -- (T1);
\draw[blue] (T2) -- (B2);
\draw[blue] (T3) -- (B3);
\foreach \i in {1,2,3}  { \filldraw[fill=white,draw=black,line width = 1pt] (T\i) circle (4pt); \filldraw[fill=white,draw=black,line width = 1pt]  (B\i) circle (4pt); } 
\end{tikzpicture} 
\end{array}
&
\begin{array}{c} 
\begin{tikzpicture}[xscale=.28,yscale=.28,line width=1.25pt] 
\foreach \i in {1,2,3}  { \path (\i,1.25) coordinate (T\i); \path (\i,.25) coordinate (B\i); } 
\filldraw[fill= black!12,draw=black!12,line width=4pt]  (T1) -- (T3) -- (B3) -- (B1) -- (T1);
\draw[blue] (T2) -- (B2);
\draw[blue] (T3) -- (B3);
\foreach \i in {1,2,3}  { \filldraw[fill=white,draw=black,line width = 1pt] (T\i) circle (4pt); \filldraw[fill=white,draw=black,line width = 1pt]  (B\i) circle (4pt); } 
\end{tikzpicture} 
\end{array}
&
\begin{array}{c} 
\begin{tikzpicture}[xscale=.28,yscale=.28,line width=1.25pt] 
\foreach \i in {1,2,3,4}  { \path (\i,1.25) coordinate (T\i); \path (\i,.25) coordinate (B\i); } 
\filldraw[fill= black!12,draw=black!12,line width=4pt]  (T1) -- (T4) -- (B4) -- (B1) -- (T1);
\draw[blue] (T1) -- (B1);
\draw[blue] (T2) -- (B2);
\draw[blue] (T3) -- (B3);
\draw[blue] (T4) -- (B4);
\foreach \i in {1,2,3,4}  { \filldraw[fill=white,draw=black,line width = 1pt] (T\i) circle (4pt); \filldraw[fill=white,draw=black,line width = 1pt]  (B\i) circle (4pt); } 
\end{tikzpicture} 
\end{array}
&
\begin{array}{c} 
\begin{tikzpicture}[xscale=.28,yscale=.28,line width=1.25pt] 
\foreach \i in {1,2,3,4}  { \path (\i,1.25) coordinate (T\i); \path (\i,.25) coordinate (B\i); } 
\filldraw[fill= black!12,draw=black!12,line width=4pt]  (T1) -- (T4) -- (B4) -- (B1) -- (T1);
\draw[blue] (T1) -- (B1);
\draw[blue] (T2) -- (B2);
\draw[blue] (T3) -- (B3);
\draw[blue] (T4) -- (B4);
\foreach \i in {1,2,3,4}  { \filldraw[fill=white,draw=black,line width = 1pt] (T\i) circle (4pt); \filldraw[fill=white,draw=black,line width = 1pt]  (B\i) circle (4pt); } 
\end{tikzpicture} 
\end{array}
&
\begin{array}{c} 
\begin{tikzpicture}[xscale=.28,yscale=.28,line width=1.25pt] 
\foreach \i in {1,2,3,4,5}  { \path (\i,1.25) coordinate (T\i); \path (\i,.25) coordinate (B\i); } 
\filldraw[fill= black!12,draw=black!12,line width=4pt]  (T1) -- (T5) -- (B5) -- (B1) -- (T1);
\draw[blue] (T1) -- (B1);
\draw[blue] (T2) -- (B2);
\draw[blue] (T3) -- (B3);
\draw[blue] (T4) -- (B4);
\draw[blue] (T5) -- (B5);
\foreach \i in {1,2,3,4,5}  { \filldraw[fill=white,draw=black,line width = 1pt] (T\i) circle (4pt); \filldraw[fill=white,draw=black,line width = 1pt]  (B\i) circle (4pt); } 
\end{tikzpicture} 
\end{array}
&
\begin{array}{c} 
\begin{tikzpicture}[xscale=.28,yscale=.28,line width=1.25pt] 
\foreach \i in {1,2,3,4,5}  { \path (\i,1.25) coordinate (T\i); \path (\i,.25) coordinate (B\i); } 
\filldraw[fill= black!12,draw=black!12,line width=4pt]  (T1) -- (T5) -- (B5) -- (B1) -- (T1);
\draw[blue] (T1) -- (B1);
\draw[blue] (T2) -- (B2);
\draw[blue] (T3) -- (B3);
\draw[blue] (T4) -- (B4);
\draw[blue] (T5) -- (B5);
\foreach \i in {1,2,3,4,5}  { \filldraw[fill=white,draw=black,line width = 1pt] (T\i) circle (4pt); \filldraw[fill=white,draw=black,line width = 1pt]  (B\i) circle (4pt); } 
\end{tikzpicture} 
\end{array}
\\
n = 4
& 
\vphantom{\bigg\vert}
&
&
&
\begin{array}{c} 
\begin{tikzpicture}[xscale=.28,yscale=.28,line width=1.25pt] 
\foreach \i in {1,2,3}  { \path (\i,1.25) coordinate (T\i); \path (\i,.25) coordinate (B\i); } 
\filldraw[fill= black!12,draw=black!12,line width=4pt]  (T1) -- (T3) -- (B3) -- (B1) -- (T1);
\draw[blue] (T3) -- (B3);
\foreach \i in {1,2,3}  { \filldraw[fill=white,draw=black,line width = 1pt] (T\i) circle (4pt); \filldraw[fill=white,draw=black,line width = 1pt]  (B\i) circle (4pt); } 
\end{tikzpicture} 
\end{array}
&
\begin{array}{c} 
\begin{tikzpicture}[xscale=.28,yscale=.28,line width=1.25pt] 
\foreach \i in {1,2,3}  { \path (\i,1.25) coordinate (T\i); \path (\i,.25) coordinate (B\i); } 
\filldraw[fill= black!12,draw=black!12,line width=4pt]  (T1) -- (T3) -- (B3) -- (B1) -- (T1);
\draw[blue] (T3) -- (B3);
\foreach \i in {1,2,3}  { \filldraw[fill=white,draw=black,line width = 1pt] (T\i) circle (4pt); \filldraw[fill=white,draw=black,line width = 1pt]  (B\i) circle (4pt); } 
\end{tikzpicture} 
\end{array}
&
\begin{array}{c} 
\begin{tikzpicture}[xscale=.28,yscale=.28,line width=1.25pt] 
\foreach \i in {1,2,3,4}  { \path (\i,1.25) coordinate (T\i); \path (\i,.25) coordinate (B\i); } 
\filldraw[fill= black!12,draw=black!12,line width=4pt]  (T1) -- (T4) -- (B4) -- (B1) -- (T1);
\draw[blue] (T2) -- (B2);
\draw[blue] (T3) -- (B3);
\draw[blue] (T4) -- (B4);
\foreach \i in {1,2,3,4}  { \filldraw[fill=white,draw=black,line width = 1pt] (T\i) circle (4pt); \filldraw[fill=white,draw=black,line width = 1pt]  (B\i) circle (4pt); } 
\end{tikzpicture} 
\end{array}
&
\begin{array}{c} 
\begin{tikzpicture}[xscale=.28,yscale=.28,line width=1.25pt] 
\foreach \i in {1,2,3,4}  { \path (\i,1.25) coordinate (T\i); \path (\i,.25) coordinate (B\i); } 
\filldraw[fill= black!12,draw=black!12,line width=4pt]  (T1) -- (T4) -- (B4) -- (B1) -- (T1);
\draw[blue] (T2) -- (B2);
\draw[blue] (T3) -- (B3);
\draw[blue] (T4) -- (B4);
\foreach \i in {1,2,3,4}  { \filldraw[fill=white,draw=black,line width = 1pt] (T\i) circle (4pt); \filldraw[fill=white,draw=black,line width = 1pt]  (B\i) circle (4pt); } 
\end{tikzpicture} 
\end{array}
&
\begin{array}{c} 
\begin{tikzpicture}[xscale=.28,yscale=.28,line width=1.25pt] 
\foreach \i in {1,2,3,4,5}  { \path (\i,1.25) coordinate (T\i); \path (\i,.25) coordinate (B\i); } 
\filldraw[fill= black!12,draw=black!12,line width=4pt]  (T1) -- (T5) -- (B5) -- (B1) -- (T1);
\draw[blue] (T1) -- (B1);
\draw[blue] (T2) -- (B2);
\draw[blue] (T3) -- (B3);
\draw[blue] (T4) -- (B4);
\draw[blue] (T5) -- (B5);
\foreach \i in {1,2,3,4,5}  { \filldraw[fill=white,draw=black,line width = 1pt] (T\i) circle (4pt); \filldraw[fill=white,draw=black,line width = 1pt]  (B\i) circle (4pt); } 
\end{tikzpicture} 
\end{array}
&
\begin{array}{c} 
\begin{tikzpicture}[xscale=.28,yscale=.28,line width=1.25pt] 
\foreach \i in {1,2,3,4,5}  { \path (\i,1.25) coordinate (T\i); \path (\i,.25) coordinate (B\i); } 
\filldraw[fill= black!12,draw=black!12,line width=4pt]  (T1) -- (T5) -- (B5) -- (B1) -- (T1);
\draw[blue] (T1) -- (B1);
\draw[blue] (T2) -- (B2);
\draw[blue] (T3) -- (B3);
\draw[blue] (T4) -- (B4);
\draw[blue] (T5) -- (B5);
\foreach \i in {1,2,3,4,5}  { \filldraw[fill=white,draw=black,line width = 1pt] (T\i) circle (4pt); \filldraw[fill=white,draw=black,line width = 1pt]  (B\i) circle (4pt); } 
\end{tikzpicture} 
\end{array}
\\
n = 5
& 
\vphantom{\bigg\vert}
&
&
&
&
\begin{array}{c} 
\begin{tikzpicture}[xscale=.28,yscale=.28,line width=1.25pt] 
\foreach \i in {1,2,3}  { \path (\i,1.25) coordinate (T\i); \path (\i,.25) coordinate (B\i); } 
\filldraw[fill= black!12,draw=black!12,line width=4pt]  (T1) -- (T3) -- (B3) -- (B1) -- (T1);
\foreach \i in {1,2,3}  { \filldraw[fill=white,draw=black,line width = 1pt] (T\i) circle (4pt); \filldraw[fill=white,draw=black,line width = 1pt]  (B\i) circle (4pt); } 
\end{tikzpicture} 
\end{array}
&
\begin{array}{c} 
\begin{tikzpicture}[xscale=.28,yscale=.28,line width=1.25pt] 
\foreach \i in {1,2,3,4}  { \path (\i,1.25) coordinate (T\i); \path (\i,.25) coordinate (B\i); } 
\filldraw[fill= black!12,draw=black!12,line width=4pt]  (T1) -- (T4) -- (B4) -- (B1) -- (T1);
\draw[blue] (T3) -- (B3);
\draw[blue] (T4) -- (B4);
\foreach \i in {1,2,3,4}  { \filldraw[fill=white,draw=black,line width = 1pt] (T\i) circle (4pt); \filldraw[fill=white,draw=black,line width = 1pt]  (B\i) circle (4pt); } 
\end{tikzpicture} 
\end{array}
&
\begin{array}{c} 
\begin{tikzpicture}[xscale=.28,yscale=.28,line width=1.25pt] 
\foreach \i in {1,2,3,4}  { \path (\i,1.25) coordinate (T\i); \path (\i,.25) coordinate (B\i); } 
\filldraw[fill= black!12,draw=black!12,line width=4pt]  (T1) -- (T4) -- (B4) -- (B1) -- (T1);
\draw[blue] (T3) -- (B3);
\draw[blue] (T4) -- (B4);
\foreach \i in {1,2,3,4}  { \filldraw[fill=white,draw=black,line width = 1pt] (T\i) circle (4pt); \filldraw[fill=white,draw=black,line width = 1pt]  (B\i) circle (4pt); } 
\end{tikzpicture} 
\end{array}
&
\begin{array}{c} 
\begin{tikzpicture}[xscale=.28,yscale=.28,line width=1.25pt] 
\foreach \i in {1,2,3,4,5}  { \path (\i,1.25) coordinate (T\i); \path (\i,.25) coordinate (B\i); } 
\filldraw[fill= black!12,draw=black!12,line width=4pt]  (T1) -- (T5) -- (B5) -- (B1) -- (T1);
\draw[blue] (T2) -- (B2);
\draw[blue] (T3) -- (B3);
\draw[blue] (T4) -- (B4);
\draw[blue] (T5) -- (B5);
\foreach \i in {1,2,3,4,5}  { \filldraw[fill=white,draw=black,line width = 1pt] (T\i) circle (4pt); \filldraw[fill=white,draw=black,line width = 1pt]  (B\i) circle (4pt); } 
\end{tikzpicture} 
\end{array}
&
\begin{array}{c} 
\begin{tikzpicture}[xscale=.28,yscale=.28,line width=1.25pt] 
\foreach \i in {1,2,3,4,5}  { \path (\i,1.25) coordinate (T\i); \path (\i,.25) coordinate (B\i); } 
\filldraw[fill= black!12,draw=black!12,line width=4pt]  (T1) -- (T5) -- (B5) -- (B1) -- (T1);
\draw[blue] (T2) -- (B2);
\draw[blue] (T3) -- (B3);
\draw[blue] (T4) -- (B4);
\draw[blue] (T5) -- (B5);
\foreach \i in {1,2,3,4,5}  { \filldraw[fill=white,draw=black,line width = 1pt] (T\i) circle (4pt); \filldraw[fill=white,draw=black,line width = 1pt]  (B\i) circle (4pt); } 
\end{tikzpicture} 
\end{array}
\\
n = 6
& 
\vphantom{\bigg\vert}
&
&
&
&
&
\begin{array}{c} 
\begin{tikzpicture}[xscale=.28,yscale=.28,line width=1.25pt] 
\foreach \i in {1,2,3,4}  { \path (\i,1.25) coordinate (T\i); \path (\i,.25) coordinate (B\i); } 
\filldraw[fill= black!12,draw=black!12,line width=4pt]  (T1) -- (T4) -- (B4) -- (B1) -- (T1);
\draw[blue] (T4) -- (B4);
\foreach \i in {1,2,3,4}  { \filldraw[fill=white,draw=black,line width = 1pt] (T\i) circle (4pt); \filldraw[fill=white,draw=black,line width = 1pt]  (B\i) circle (4pt); } 
\end{tikzpicture} 
\end{array}
&
\begin{array}{c} 
\begin{tikzpicture}[xscale=.28,yscale=.28,line width=1.25pt] 
\foreach \i in {1,2,3,4}  { \path (\i,1.25) coordinate (T\i); \path (\i,.25) coordinate (B\i); } 
\filldraw[fill= black!12,draw=black!12,line width=4pt]  (T1) -- (T4) -- (B4) -- (B1) -- (T1);
\draw[blue] (T4) -- (B4);
\foreach \i in {1,2,3,4}  { \filldraw[fill=white,draw=black,line width = 1pt] (T\i) circle (4pt); \filldraw[fill=white,draw=black,line width = 1pt]  (B\i) circle (4pt); } 
\end{tikzpicture} 
\end{array}
&
\begin{array}{c} 
\begin{tikzpicture}[xscale=.28,yscale=.28,line width=1.25pt] 
\foreach \i in {1,2,3,4,5}  { \path (\i,1.25) coordinate (T\i); \path (\i,.25) coordinate (B\i); } 
\filldraw[fill= black!12,draw=black!12,line width=4pt]  (T1) -- (T5) -- (B5) -- (B1) -- (T1);
\draw[blue] (T3) -- (B3);
\draw[blue] (T4) -- (B4);
\draw[blue] (T5) -- (B5);
\foreach \i in {1,2,3,4,5}  { \filldraw[fill=white,draw=black,line width = 1pt] (T\i) circle (4pt); \filldraw[fill=white,draw=black,line width = 1pt]  (B\i) circle (4pt); } 
\end{tikzpicture} 
\end{array}
&
\begin{array}{c} 
\begin{tikzpicture}[xscale=.28,yscale=.28,line width=1.25pt] 
\foreach \i in {1,2,3,4,5}  { \path (\i,1.25) coordinate (T\i); \path (\i,.25) coordinate (B\i); } 
\filldraw[fill= black!12,draw=black!12,line width=4pt]  (T1) -- (T5) -- (B5) -- (B1) -- (T1);
\draw[blue] (T3) -- (B3);
\draw[blue] (T4) -- (B4);
\draw[blue] (T5) -- (B5);
\foreach \i in {1,2,3,4,5}  { \filldraw[fill=white,draw=black,line width = 1pt] (T\i) circle (4pt); \filldraw[fill=white,draw=black,line width = 1pt]  (B\i) circle (4pt); } 
\end{tikzpicture} 
\end{array}
\\
n = 7
& 
\vphantom{\bigg\vert}
&
&
&
&
&
&
\begin{array}{c} 
\begin{tikzpicture}[xscale=.28,yscale=.28,line width=1.25pt] 
\foreach \i in {1,2,3,4}  { \path (\i,1.25) coordinate (T\i); \path (\i,.25) coordinate (B\i); } 
\filldraw[fill= black!12,draw=black!12,line width=4pt]  (T1) -- (T4) -- (B4) -- (B1) -- (T1);
\foreach \i in {1,2,3,4}  { \filldraw[fill=white,draw=black,line width = 1pt] (T\i) circle (4pt); \filldraw[fill=white,draw=black,line width = 1pt]  (B\i) circle (4pt); } 
\end{tikzpicture} 
\end{array}
&
\begin{array}{c} 
\begin{tikzpicture}[xscale=.28,yscale=.28,line width=1.25pt] 
\foreach \i in {1,2,3,4,5}  { \path (\i,1.25) coordinate (T\i); \path (\i,.25) coordinate (B\i); } 
\filldraw[fill= black!12,draw=black!12,line width=4pt]  (T1) -- (T5) -- (B5) -- (B1) -- (T1);
\draw[blue] (T4) -- (B4);
\draw[blue] (T5) -- (B5);
\foreach \i in {1,2,3,4,5}  { \filldraw[fill=white,draw=black,line width = 1pt] (T\i) circle (4pt); \filldraw[fill=white,draw=black,line width = 1pt]  (B\i) circle (4pt); } 
\end{tikzpicture} 
\end{array}
&
\begin{array}{c} 
\begin{tikzpicture}[xscale=.28,yscale=.28,line width=1.25pt] 
\foreach \i in {1,2,3,4,5}  { \path (\i,1.25) coordinate (T\i); \path (\i,.25) coordinate (B\i); } 
\filldraw[fill= black!12,draw=black!12,line width=4pt]  (T1) -- (T5) -- (B5) -- (B1) -- (T1);
\draw[blue] (T4) -- (B4);
\draw[blue] (T5) -- (B5);
\foreach \i in {1,2,3,4,5}  { \filldraw[fill=white,draw=black,line width = 1pt] (T\i) circle (4pt); \filldraw[fill=white,draw=black,line width = 1pt]  (B\i) circle (4pt); } 
\end{tikzpicture} 
\end{array}
\\
n = 8
& 
\vphantom{\bigg\vert}
&
&
&
&
&
&
&
\begin{array}{c} 
\begin{tikzpicture}[xscale=.28,yscale=.28,line width=1.25pt] 
\foreach \i in {1,2,3,4,5}  { \path (\i,1.25) coordinate (T\i); \path (\i,.25) coordinate (B\i); } 
\filldraw[fill= black!12,draw=black!12,line width=4pt]  (T1) -- (T5) -- (B5) -- (B1) -- (T1);
\draw[blue] (T5) -- (B5);
\foreach \i in {1,2,3,4,5}  { \filldraw[fill=white,draw=black,line width = 1pt] (T\i) circle (4pt); \filldraw[fill=white,draw=black,line width = 1pt]  (B\i) circle (4pt); } 
\end{tikzpicture} 
\end{array}
&
\begin{array}{c} 
\begin{tikzpicture}[xscale=.28,yscale=.28,line width=1.25pt] 
\foreach \i in {1,2,3,4,5}  { \path (\i,1.25) coordinate (T\i); \path (\i,.25) coordinate (B\i); } 
\filldraw[fill= black!12,draw=black!12,line width=4pt]  (T1) -- (T5) -- (B5) -- (B1) -- (T1);
\draw[blue] (T5) -- (B5);
\foreach \i in {1,2,3,4,5}  { \filldraw[fill=white,draw=black,line width = 1pt] (T\i) circle (4pt); \filldraw[fill=white,draw=black,line width = 1pt]  (B\i) circle (4pt); } 
\end{tikzpicture} 
\end{array}
\\
n = 9
& 
\vphantom{\bigg\vert}
&
&
&
&
&
&
&
&
\begin{array}{c} 
\begin{tikzpicture}[xscale=.28,yscale=.28,line width=1.25pt] 
\foreach \i in {1,2,3,4,5}  { \path (\i,1.25) coordinate (T\i); \path (\i,.25) coordinate (B\i); } 
\filldraw[fill= black!12,draw=black!12,line width=4pt]  (T1) -- (T5) -- (B5) -- (B1) -- (T1);
\foreach \i in {1,2,3,4,5}  { \filldraw[fill=white,draw=black,line width = 1pt] (T\i) circle (4pt); \filldraw[fill=white,draw=black,line width = 1pt]  (B\i) circle (4pt); } 
\end{tikzpicture} 
\end{array}
\end{array}
$$ 
 \caption {The elements  $\ef_{k,n}$  for $k \le 5$ and $n \le 9$.
 \label{figure:generatortable}}  
\end{figure}

\begin{remark} Theorem  \ref{thm:generator} below shows for  
$k \in \half \ZZ_{\ge 1}$ and $n \in \ZZ_{\ge 1}$ such that $2k > n$  that $\ker \Phi_{k,n} = \langle \ef_{k,n} \rangle$ as a two-sided ideal, and Theorem \ref{T:secfund} shows that $\ef_{k,n}$ is an essential idempotent.
For a fixed value of $n$,  the first time the kernel is nonzero is when $k = \half(n+1)$ (i.e., when $n=2k-1$).  This is the first entry in each row of the table  in Figure  \ref{figure:generatortable}.
\end{remark} 

The expression for $\ef_{k,n}$ in the diagram basis is given by 
\begin{equation}\label{eq:ekn}  \ef_{k,n}  =  \sum_{\pi_{k,n} \preceq \vr}  \mu_{2k}(\pi_{k,n},\vr)  d_{\vr}\end{equation}
where $\pi_{k,n}$ is the set partition of $[1,2k]$ corresponding to $\ef_{k,n}$.   In particular, when
$k = \half(n+1)$,  all $\vr$ in  $\Pi_{2k-1}$ occur in the expression for $\ef_{k,n}$ and have integer coefficients
that can be computed using \eqref{eq:mobiusb}.

\begin{example} 
Here is the diagram basis expansion for  $\ef_{3,3}$ using the M\"obius formulas in \eqref{eq:mobiusa} and \eqref{eq:mobiusb}.  The diagram basis expansion for  
$\ef_{2,3} =  \begin{array}{c}\begin{tikzpicture}[xscale=.30,yscale=.30,line width=1.0pt] 
\foreach \i in {1,2}  { \path (\i,1.25) coordinate (T\i); \path (\i,.25) coordinate (B\i); } 
\filldraw[fill= black!12,draw=black!12,line width=4pt]  (T1) -- (T2) -- (B2) -- (B1) -- (T1);
\foreach \i in {1,2}  { \filldraw[fill=white,draw=black,line width = 1pt] (T\i) circle (5pt); \filldraw[fill=white,draw=black,line width = 1pt]  (B\i) circle (5pt); } 
\end{tikzpicture}\end{array}$ and 
$\ef_{3,2}= \mathsf{I}_k^{\boldsymbol \circ} = \begin{array}{c}\begin{tikzpicture}[xscale=.30,yscale=.30,line width=1.0pt] 
\foreach \i in {1,2,3}  { \path (\i,1.25) coordinate (T\i); \path (\i,.25) coordinate (B\i); } 
\filldraw[fill= black!12,draw=black!12,line width=4pt]  (T1) -- (T3) -- (B3) -- (B1) -- (T1);
\draw[blue] (T1) -- (B1);
\draw[blue] (T2) -- (B2);
\draw[blue] (T3) -- (B3);
\foreach \i in {1,2,3}  { \filldraw[fill=white,draw=black,line width = 1pt] (T\i) circle (5pt); \filldraw[fill=white,draw=black,line width = 1pt]  (B\i) circle (5pt); } 
\end{tikzpicture}\end{array}$  
can be found in \eqref{eq:orbit2diag} and Remark \ref{R:orbid}, respectively. 
\begin{align*}
\ef_{3,3} & =  \begin{array}{c}\begin{tikzpicture}[xscale=.36,yscale=.36,line width=1.0pt] 
\foreach \i in {1,2,3}  { \path (\i,1.25) coordinate (T\i); \path (\i,.25) coordinate (B\i); } 
\filldraw[fill= black!12,draw=black!12,line width=4pt]  (T1) -- (T3) -- (B3) -- (B1) -- (T1);
\draw[blue] (T2) -- (B2);
\draw[blue] (T3) -- (B3);
\foreach \i in {1,2,3}  { \filldraw[fill=white,draw=black,line width = 1pt] (T\i) circle (4pt); \filldraw[fill=white,draw=black,line width = 1pt]  (B\i) circle (4pt); } 
\end{tikzpicture}\end{array}
=
\begin{array}{c}\begin{tikzpicture}[xscale=.36,yscale=.36,line width=1.0pt] 
\foreach \i in {1,2,3}  { \path (\i,1.25) coordinate (T\i); \path (\i,.25) coordinate (B\i); } 
\filldraw[fill= black!12,draw=black!12,line width=4pt]  (T1) -- (T3) -- (B3) -- (B1) -- (T1);
\draw[blue] (T2) -- (B2);
\draw[blue] (T3) -- (B3);
\foreach \i in {1,2,3}  { \filldraw[fill=black,draw=black,line width = 1pt] (T\i) circle (4pt); \filldraw[fill=black,draw=black,line width = 1pt]  (B\i) circle (4pt); } 
\end{tikzpicture}\end{array}
-
\begin{array}{c}\begin{tikzpicture}[xscale=.36,yscale=.36,line width=1.0pt] 
\foreach \i in {1,2,3}  { \path (\i,1.25) coordinate (T\i); \path (\i,.25) coordinate (B\i); } 
\filldraw[fill= black!12,draw=black!12,line width=4pt]  (T1) -- (T3) -- (B3) -- (B1) -- (T1);
\draw[blue] (T1) -- (B1);
\draw[blue] (T2) -- (B2);
\draw[blue] (T3) -- (B3);
\foreach \i in {1,2,3}  { \filldraw[fill=black,draw=black,line width = 1pt] (T\i) circle (4pt); \filldraw[fill=black,draw=black,line width = 1pt]  (B\i) circle (4pt); } 
\end{tikzpicture}\end{array}
-
\begin{array}{c}\begin{tikzpicture}[xscale=.36,yscale=.36,line width=1.0pt] 
\foreach \i in {1,2,3}  { \path (\i,1.25) coordinate (T\i); \path (\i,.25) coordinate (B\i); } 
\filldraw[fill= black!12,draw=black!12,line width=4pt]  (T1) -- (T3) -- (B3) -- (B1) -- (T1);
\draw[blue] (T1)--(T2);
\draw[blue] (T2)--(B2);
\draw[blue] (T3) -- (B3);
\foreach \i in {1,2,3}  { \filldraw[fill=black,draw=black,line width = 1pt] (T\i) circle (4pt); \filldraw[fill=black,draw=black,line width = 1pt]  (B\i) circle (4pt); } 
\end{tikzpicture}\end{array}
-
\begin{array}{c}\begin{tikzpicture}[xscale=.36,yscale=.36,line width=1.0pt] 
\foreach \i in {1,2,3}  { \path (\i,1.25) coordinate (T\i); \path (\i,.25) coordinate (B\i); } 
\filldraw[fill= black!12,draw=black!12,line width=4pt]  (T1) -- (T3) -- (B3) -- (B1) -- (T1);
\draw[blue] (B1)--(B2);
\draw[blue] (T2)--(B2);
\draw[blue] (T3) -- (B3);
\foreach \i in {1,2,3}  { \filldraw[fill=black,draw=black,line width = 1pt] (T\i) circle (4pt); \filldraw[fill=black,draw=black,line width = 1pt]  (B\i) circle (4pt); } 
\end{tikzpicture}\end{array}
-
\begin{array}{c}\begin{tikzpicture}[xscale=.36,yscale=.36,line width=1.0pt] 
\foreach \i in {1,2,3}  { \path (\i,1.25) coordinate (T\i); \path (\i,.25) coordinate (B\i); } 
\filldraw[fill= black!12,draw=black!12,line width=4pt]  (T1) -- (T3) -- (B3) -- (B1) -- (T1);
\draw[blue] (T2) -- (B2);
\draw[blue] (T3) -- (B3);
\draw[blue] (T1) .. controls +(0,-.50) and +(0,-.50) .. (T3);
\foreach \i in {1,2,3}  { \filldraw[fill=black,draw=black,line width = 1pt] (T\i) circle (4pt); \filldraw[fill=black,draw=black,line width = 1pt]  (B\i) circle (4pt); } 
\end{tikzpicture}\end{array}
-
\begin{array}{c}\begin{tikzpicture}[xscale=.36,yscale=.36,line width=1.0pt] 
\foreach \i in {1,2,3}  { \path (\i,1.25) coordinate (T\i); \path (\i,.25) coordinate (B\i); } 
\filldraw[fill= black!12,draw=black!12,line width=4pt]  (T1) -- (T3) -- (B3) -- (B1) -- (T1);
\draw[blue] (T2) -- (B2);
\draw[blue] (T3) -- (B3);
\draw[blue] (B1) .. controls +(0,.50) and +(0,.50) .. (B3);
\foreach \i in {1,2,3}  { \filldraw[fill=black,draw=black,line width = 1pt] (T\i) circle (4pt); \filldraw[fill=black,draw=black,line width = 1pt]  (B\i) circle (4pt); } 
\end{tikzpicture}\end{array}
-
\begin{array}{c}\begin{tikzpicture}[xscale=.36,yscale=.36,line width=1.0pt] 
\foreach \i in {1,2,3}  { \path (\i,1.25) coordinate (T\i); \path (\i,.25) coordinate (B\i); } 
\filldraw[fill= black!12,draw=black!12,line width=4pt]  (T1) -- (T3) -- (B3) -- (B1) -- (T1);
\draw[blue] (T2) -- (B2)--(B3);
\draw[blue] (T2)--(T3) -- (B3);
\foreach \i in {1,2,3}  { \filldraw[fill=black,draw=black,line width = 1pt] (T\i) circle (4pt); \filldraw[fill=black,draw=black,line width = 1pt]  (B\i) circle (4pt); } 
\end{tikzpicture}\end{array} 
\\
& 
+ 
\begin{array}{c}\begin{tikzpicture}[xscale=.36,yscale=.36,line width=1.0pt] 
\foreach \i in {1,2,3}  { \path (\i,1.25) coordinate (T\i); \path (\i,.25) coordinate (B\i); } 
\filldraw[fill= black!12,draw=black!12,line width=4pt]  (T1) -- (T3) -- (B3) -- (B1) -- (T1);
\draw[blue] (T1) -- (B1);
\draw[blue] (T3)--(T2) -- (B2) -- (B3);
\draw[blue] (T3) -- (B3);
\foreach \i in {1,2,3}  { \filldraw[fill=black,draw=black,line width = 1pt] (T\i) circle (4pt); \filldraw[fill=black,draw=black,line width = 1pt]  (B\i) circle (4pt); } 
\end{tikzpicture}\end{array}
+
\begin{array}{c}\begin{tikzpicture}[xscale=.36,yscale=.36,line width=1.0pt] 
\foreach \i in {1,2,3}  { \path (\i,1.25) coordinate (T\i); \path (\i,.25) coordinate (B\i); } 
\filldraw[fill= black!12,draw=black!12,line width=4pt]  (T1) -- (T3) -- (B3) -- (B1) -- (T1);
\draw[blue] (T1)--(T2);
\draw[blue] (T2)--(B2);
\draw[blue] (T3) -- (B3);
\draw[blue] (B1) .. controls +(0,.50) and +(0,.50) .. (B3);
\foreach \i in {1,2,3}  { \filldraw[fill=black,draw=black,line width = 1pt] (T\i) circle (4pt); \filldraw[fill=black,draw=black,line width = 1pt]  (B\i) circle (4pt); } 
\end{tikzpicture}\end{array}
+
\begin{array}{c}\begin{tikzpicture}[xscale=.36,yscale=.36,line width=1.0pt] 
\foreach \i in {1,2,3}  { \path (\i,1.25) coordinate (T\i); \path (\i,.25) coordinate (B\i); } 
\filldraw[fill= black!12,draw=black!12,line width=4pt]  (T1) -- (T3) -- (B3) -- (B1) -- (T1);
\draw[blue] (B1)--(B2);
\draw[blue] (T2)--(B2);
\draw[blue] (T3) -- (B3);
\draw[blue] (T1) .. controls +(0,-.50) and +(0,-.50) .. (T3);
\foreach \i in {1,2,3}  { \filldraw[fill=black,draw=black,line width = 1pt] (T\i) circle (4pt); \filldraw[fill=black,draw=black,line width = 1pt]  (B\i) circle (4pt); } 
\end{tikzpicture}\end{array}
+2
\begin{array}{c}\begin{tikzpicture}[xscale=.36,yscale=.36,line width=1.0pt] 
\foreach \i in {1,2,3}  { \path (\i,1.25) coordinate (T\i); \path (\i,.25) coordinate (B\i); } 
\filldraw[fill= black!12,draw=black!12,line width=4pt]  (T1) -- (T3) -- (B3) -- (B1) -- (T1);
\draw[blue] (T1)--(T2) -- (B2) -- (B3);
\draw[blue] (T2)--(T3) -- (B3);
\foreach \i in {1,2,3}  { \filldraw[fill=black,draw=black,line width = 1pt] (T\i) circle (4pt); \filldraw[fill=black,draw=black,line width = 1pt]  (B\i) circle (4pt); } 
\end{tikzpicture}\end{array}
+2
\begin{array}{c}\begin{tikzpicture}[xscale=.36,yscale=.36,line width=1.0pt] 
\foreach \i in {1,2,3}  { \path (\i,1.25) coordinate (T\i); \path (\i,.25) coordinate (B\i); } 
\filldraw[fill= black!12,draw=black!12,line width=4pt]  (T1) -- (T3) -- (B3) -- (B1) -- (T1);
\draw[blue] (B1)--(B2) -- (T2);
\draw[blue] (T2)--(T3) -- (B3)--(B2);
\foreach \i in {1,2,3}  { \filldraw[fill=black,draw=black,line width = 1pt] (T\i) circle (4pt); \filldraw[fill=black,draw=black,line width = 1pt]  (B\i) circle (4pt); } 
\end{tikzpicture}\end{array} +2
\begin{array}{c}\begin{tikzpicture}[xscale=.36,yscale=.36,line width=1.0pt] 
\foreach \i in {1,2,3}  { \path (\i,1.25) coordinate (T\i); \path (\i,.25) coordinate (B\i); } 
\filldraw[fill= black!12,draw=black!12,line width=4pt]  (T1) -- (T3) -- (B3) -- (B1) -- (T1);
\draw[blue] (T1) -- (T2) -- (B2) -- (B1) -- (T1);
\draw[blue] (T3)--(B3);
\foreach \i in {1,2,3}  { \filldraw[fill=black,draw=black,line width = 1pt] (T\i) circle (4pt); \filldraw[fill=black,draw=black,line width = 1pt]  (B\i) circle (4pt); } 
\end{tikzpicture}\end{array}
+ 2 
\begin{array}{c}\begin{tikzpicture}[xscale=.36,yscale=.36,line width=1.0pt] 
\foreach \i in {1,2,3}  { \path (\i,1.25) coordinate (T\i); \path (\i,.25) coordinate (B\i); } 
\filldraw[fill= black!12,draw=black!12,line width=4pt]  (T1) -- (T3) -- (B3) -- (B1) -- (T1);
\draw[blue] (T1) -- (B1);
\draw[blue] (T2) -- (B2);
\draw[blue] (T3) -- (B3);
\draw[blue] (T1) .. controls +(0,-.50) and +(0,-.50) .. (T3);
\draw[blue] (B1) .. controls +(0,.50) and +(0,.50) .. (B3);
\foreach \i in {1,2,3}  { \filldraw[fill=black,draw=black,line width = 1pt] (T\i) circle (4pt); \filldraw[fill=black,draw=black,line width = 1pt]  (B\i) circle (4pt); } 
\end{tikzpicture}\end{array}
- 6
\begin{array}{c}\begin{tikzpicture}[xscale=.36,yscale=.36,line width=1.0pt] 
\foreach \i in {1,2,3}  { \path (\i,1.25) coordinate (T\i); \path (\i,.25) coordinate (B\i); } 
\filldraw[fill= black!12,draw=black!12,line width=4pt]  (T1) -- (T3) -- (B3) -- (B1) -- (T1);
\draw[blue] (T1) -- (B1) -- (B2);
\draw[blue] (T1)--(T2) -- (B2) -- (B3);
\draw[blue] (T2)--(T3) -- (B3);
\foreach \i in {1,2,3}  { \filldraw[fill=black,draw=black,line width = 1pt] (T\i) circle (4pt); \filldraw[fill=black,draw=black,line width = 1pt]  (B\i) circle (4pt); } 
\end{tikzpicture}\end{array}.
\end{align*}
\end{example}

\begin{thm} \label{thm:generator} For all $n\in \ZZ_{\ge 1}$ and $k \in \half\ZZ_{\ge 1}$ with $2k > n$,   $\ker \Phi_{k,n} = \langle \ef_{k,n} \rangle$.
\end{thm}

\begin{proof} First, assume that $k$ is an integer such that $2k > n$.  Then  $\ef_{k,n}$ has $n+1$ blocks when $n \ge k > n/2$ and $k$ blocks when $k > n$.  Thus,  by Theorem \ref{T:Phi}\,(a),  $\e_{k,n} \in \ker \Phi_{k,n}$, and so $\langle \ef_{k,n} \rangle \subseteq \ker \Phi_{k,n}$. Theorem \ref{T:Phi}\,(a) also tells us that $\ker \Phi_{k,n}$ is spanned by  $\{x_\pi\mid |\pi| > n\}$, so to establish the reverse inclusion $\langle \ef_{k,n} \rangle \supseteq \ker \Phi_{k,n}$,  it suffices to show that if $|\pi| > n$, then $x_\pi \in \langle \ef_{k,n} \rangle$. 

By \eqref{DiagramPermutation},   we have the following symmetry property,  which we  use throughout this proof.
\begin{equation}\label{eq:symmetry}
\text{ If } \  x_\pi \in \langle \ef_{k,n} \rangle \  \text{and}\ k \in \ZZ_{\ge 1},  \text{ then }\ \sigma' x_\pi \sigma = x_{\sigma'\ast\pi\ast\sigma} \in \langle \ef_{k,n} \rangle \quad \text{ for all }\ \  \sigma', \sigma \in \S_k.
\end{equation}
We say that $ \sigma' x_\pi \sigma = x_{\sigma'\ast\pi\ast\sigma}$ is a permutation of $x_\pi$.

When $k > n$, then $\pn(\ef_{k,n}) = k$, which is the maximum propagating number of any diagram of $\P_k(n)$. 
When $n \ge k > n/2$, then a diagram with $n + 1$ blocks has a maximum propagating number of $2k-n-1$. Such a diagram is a rook orbit diagram with $2k-n-1$  edges,  and thus is a permutation of $\e_{k,n}$. This  follows from  the observation that adding any more propagating blocks requires merging two blocks.  Thus, when $k > n/2$,   $\ef_{k,n}$ has the maximum possible propagating number of any orbit diagram in the kernel.
 
Let $m$ be that maximum propagating number, so that  $m = k$ when $k > n$, and  $m = 2k-n-1$ when $n \ge k > n/2$.
For $0 \le t \le k$, let  $\mathsf{O}_t = \{ x_\pi \mid \pi \in \mathcal{R}_{2k}, \ \pn(x_\pi) = t \}$ be the rook orbit diagrams with $t$ edges. We show that $\mathsf{O}_t \subseteq \langle \ef_{k,n} \rangle$ for $t \le m$ by reverse induction on $t$.  When $t = m$,  every element of $\mathsf{O}_t$ is a permutation of $\e_{k,n}$ and therefore is in $\langle \ef_{k,n} \rangle$ by \eqref{eq:symmetry}, so assume $t < m$, and the result holds for values greater than $t$.  Define $x_t \in \mathsf{O}_t$ to be the orbit diagram with $t$ orbit identity edges in columns $k, k-1, \ldots, k-t+1$. By \eqref{eq:symmetry},  it is sufficient to show that $x_t \in \langle
 \ef_{k,n} \rangle$. 
 
Let  $y_t  \in \mathsf{O}_{k-1}$ be the diagram with edges $\vertedge$  in every column except column $k-t$. By the inductive hypothesis $x_{t+1} \in \mathsf{O}_{t+1} \subseteq \langle \ef_{k,n} \rangle $ and so $y_t  x_{t+1} \in \langle \ef_{k,n} \rangle $. Expanding the product $y_t x_{t+1}$ in the orbit basis gives $y_t x_{t+1} = x_t + z$, where $z$ is a sum of rook orbit diagrams with $t+1$ edges.  Thus by the inductive hypothesis,  $x_t  = y_t x_{t+1} - z \in \langle \ef_{k,n} \rangle$, as desired (this is illustrated below for the product  $y_3 x_4$ when $k = 6$). 
$$
\begin{array}{c} 
y_3=\begin{array}{c}\begin{tikzpicture}[xscale=.5,yscale=.5,line width=1.25pt] 
\foreach \i in {1,2,3,4,5,6}  { \path (\i,1.25) coordinate (T\i); \path (\i,.25) coordinate (B\i); } 
\filldraw[fill= black!12,draw=black!12,line width=4pt]  (T1) -- (T6) -- (B6) -- (B1) -- (T1);
\draw[blue] (T1) -- (B1);\draw[blue] (T2) -- (B2);
\draw[blue] (T4) -- (B4);
\draw[blue] (T5) -- (B5);
\draw[blue] (T6) -- (B6);
\foreach \i in {1,2,3,4,5,6}  { \filldraw[fill=white,draw=black,line width = 1pt] (T\i) circle (4pt); \filldraw[fill=white,draw=black,line width = 1pt]  (B\i) circle (4pt); } 
\end{tikzpicture} \end{array} \\
x_4=\begin{array}{c}\begin{tikzpicture}[xscale=.5,yscale=.5,line width=1.25pt] 
\foreach \i in {1,2,3,4,5,6}  { \path (\i,1.25) coordinate (T\i); \path (\i,.25) coordinate (B\i); } 
\filldraw[fill= black!12,draw=black!12,line width=4pt]  (T1) -- (T6) -- (B6) -- (B1) -- (T1);
\draw[blue] (T3) -- (B3);
\draw[blue] (T4) -- (B4);
\draw[blue] (T5) -- (B5);
\draw[blue] (T6) -- (B6);
\foreach \i in {1,2,3,4,5,6}  { \filldraw[fill=white,draw=black,line width = 1pt] (T\i) circle (4pt); \filldraw[fill=white,draw=black,line width = 1pt]  (B\i) circle (4pt); } 
\end{tikzpicture}
\end{array}
\end{array}
=
\begin{array}{c} 
\underbrace{\begin{tikzpicture}[xscale=.5,yscale=.5,line width=1.25pt] 
\foreach \i in {1,2,3,4,5.6}  { \path (\i,1.25) coordinate (T\i); \path (\i,.25) coordinate (B\i); } 
\filldraw[fill= black!12,draw=black!12,line width=4pt]  (T1) -- (T6) -- (B6) -- (B1) -- (T1);
\draw[blue]  (T4) -- (B4);
\draw[blue]  (T5) -- (B5);
\draw[blue]  (T6) -- (B6);
\foreach \i in {1,2,3,4,5,6}  { \filldraw[fill=white,draw=black,line width = 1pt] (T\i) circle (4pt); \filldraw[fill=white,draw=black,line width = 1pt]  (B\i) circle (4pt); } 
\end{tikzpicture}}_{\text{$x_3$}}  
\end{array}
+
\begin{array}{c} 
\underbrace{\begin{tikzpicture}[xscale=.5,yscale=.5,line width=1.25pt] 
\foreach \i in {1,2,3,4,5,6}  { \path (\i,1.25) coordinate (T\i); \path (\i,.25) coordinate (B\i); } 
\filldraw[fill= black!12,draw=black!12,line width=4pt]  (T1) -- (T6) -- (B6) -- (B1) -- (T1);
\draw[purple] (T3) -- (B1);
\draw[blue]  (T4) -- (B4);
\draw[blue]  (T5) -- (B5);
\draw[blue]  (T6) -- (B6);
\foreach \i in {1,2,3,4,5,6}  { \filldraw[fill=white,draw=black,line width = 1pt] (T\i) circle (4pt); \filldraw[fill=white,draw=black,line width = 1pt]  (B\i) circle (4pt); } 
\end{tikzpicture}}_{\text{in $\mathsf{O}_4$}}
\end{array}
+
\begin{array}{c} 
\underbrace{\begin{tikzpicture}[xscale=.5,yscale=.5,line width=1.25pt] 
\foreach \i in {1,2,3,4,5,6}  { \path (\i,1.25) coordinate (T\i); \path (\i,.25) coordinate (B\i); } 
\filldraw[fill= black!12,draw=black!12,line width=4pt]  (T1) -- (T6) -- (B6) -- (B1) -- (T1);
\draw[purple]    (T3) -- (B2);
\draw[blue]  (T4) -- (B4);
\draw[blue]  (T5) -- (B5);
\draw[blue]  (T6) -- (B6);
\foreach \i in {1,2,3,4,5,6}  { \filldraw[fill=white,draw=black,line width = 1pt] (T\i) circle (4pt); \filldraw[fill=white,draw=black,line width = 1pt]  (B\i) circle (4pt); } 
\end{tikzpicture}}_{\text{in $\mathsf{O}_4$}}
\end{array}.
$$

Since the kernel of $\Phi_{k,n}$ consists of the orbit diagrams $x_\pi$ such that $\pi$ has more than $n$ blocks,
we complete the proof by arguing that the orbit diagram $x_\pi \in \langle \ef_{k,n} \rangle $ whenever  $|\pi| > n$.  Using \eqref{eq:symmetry},  we may assume that any propagating block of $x_\pi$ has a unique
edge intersecting both the top row and the bottom row, and that edge is an identity edge. We factor the diagram as $x_\pi = x_\tau x_\mu x_\beta$,   where
the three factors are obtained as follows: \   $x_\tau$ is the orbit diagram gotten from $x_\pi$ by deleting any horizontal edges in the bottom row of $x_\pi$ and adding all identity edges $\{i,k+i\}_{1 \le i \le k}$; \,  $x_\mu$ is the orbit diagram obtained from $x_\pi$ by deleting all horizontal edges; and $x_\beta$ is the orbit diagram constructed from $x_\pi$ by deleting any horizontal edges in the top row of $x_\pi$ and adding all identity edges $\{i,k+i\}_{1 \le i \le k}$, as demonstrated below for 
the orbit diagram $x_\pi$,  which has $\pn(x_\pi) = \pn(\pi) = 3$, $k = 8$, and $n < 7$,
$$
x_\pi = 
\begin{array}{c} 
{\begin{tikzpicture}[scale=.55,line width=1.25pt] 
\foreach \i in {1,...,8} 
{ \path (\i,1) coordinate (T\i); \path (\i,0) coordinate (B\i); } 
\filldraw[fill= gray!40,draw=gray!40,line width=3.2pt]  (T1) -- (T8) -- (B8) -- (B1) -- (T1);
\draw[blue] (T1) .. controls +(.1,-.30) and +(-.1,-.30) .. (T2) .. controls +(.1,-.30) and +(-.1,-.30) .. (T3);
\draw[blue] (T1) -- (B1);
\draw[blue] (T4) .. controls +(.1,-.30) and +(-.1,-.30) .. (T5) .. controls +(.1,-.30) and +(-.1,-.30) .. (T6) ;
\draw[blue] (T4) -- (B4) ;
\draw[blue] (B2) .. controls +(.1,.30) and +(-.1,.30) .. (B3) ;
\draw[blue] (B4) .. controls +(.1,.30) and +(-.1,.30) .. (B5) ;
\draw[blue] (T8) -- (B8) ;
\foreach \i in {1,...,8} { \filldraw[fill=white,draw=black,line width = 1pt] (T\i) circle (4pt); \filldraw[fill=white,draw=black,line width = 1pt]  (B\i) circle (4pt); }
\end{tikzpicture}}
\end{array} 
=
\begin{array}{l}
\begin{array}{c}
{\begin{tikzpicture}[scale=.55,line width=1.25pt] 
\foreach \i in {1,...,8} 
{ \path (\i,1) coordinate (T\i); \path (\i,0) coordinate (B\i); } 
\filldraw[fill= gray!40,draw=gray!40,line width=3.2pt]  (T1) -- (T8) -- (B8) -- (B1) -- (T1);
\draw[blue] (T1) .. controls +(.1,-.30) and +(-.1,-.30) .. (T2) .. controls +(.1,-.30) and +(-.1,-.30) .. (T3);
\draw[blue] (T4) .. controls +(.1,-.30) and +(-.1,-.30) .. (T5) .. controls +(.1,-.30) and +(-.1,-.30) .. (T6) ;
\draw[blue] (T1) -- (B1);
\draw[blue] (T2) -- (B2);
\draw[blue] (T3) -- (B3);
\draw[blue] (T4) -- (B4) ;
\draw[blue] (T5) -- (B5);
\draw[blue] (T6) -- (B6);
\draw[blue] (T7) -- (B7);
\draw[blue] (T8) -- (B8) ;
\foreach \i in {1,...,8} { \filldraw[fill=white,draw=black,line width = 1pt] (T\i) circle (4pt); \filldraw[fill=white,draw=black,line width = 1pt]  (B\i) circle (4pt); }
\end{tikzpicture}} \end{array} = x_\tau
\\
\begin{array}{c}
{\begin{tikzpicture}[scale=.55,line width=1.25pt] 
\foreach \i in {1,...,8} 
{ \path (\i,1) coordinate (T\i); \path (\i,0) coordinate (B\i); } 
\filldraw[fill= gray!40,draw=gray!40,line width=3.2pt]  (T1) -- (T8) -- (B8) -- (B1) -- (T1);
\draw[blue] (T1) -- (B1);
\draw[blue] (T4) -- (B4) ;
\draw[blue] (T8) -- (B8) ;
\foreach \i in {1,...,8} { \filldraw[fill=white,draw=black,line width = 1pt] (T\i) circle (4pt); \filldraw[fill=white,draw=black,line width = 1pt]  (B\i) circle (4pt); }
\end{tikzpicture}} \end{array} = x_\mu \in \mathsf{O}_3
\\
\begin{array}{c}
{\begin{tikzpicture}[scale=.55,line width=1.25pt] 
\foreach \i in {1,...,8} 
{ \path (\i,1) coordinate (T\i); \path (\i,0) coordinate (B\i); } 
\filldraw[fill= gray!40,draw=gray!40,line width=3.2pt]  (T1) -- (T8) -- (B8) -- (B1) -- (T1);
\draw[blue] (T1) -- (B1);
\draw[blue] (T2) -- (B2);
\draw[blue] (T3) -- (B3);
\draw[blue] (T4) -- (B4) ;
\draw[blue] (T5) -- (B5);
\draw[blue] (T6) -- (B6);
\draw[blue] (T7) -- (B7);
\draw[blue] (T8) -- (B8) ;
\draw[blue] (B2) .. controls +(.1,.30) and +(-.1,.30) .. (B3) ;
\draw[blue] (B4) .. controls +(.1,.30) and +(-.1,.30) .. (B5) ;
\foreach \i in {1,...,8} { \filldraw[fill=white,draw=black,line width = 1pt] (T\i) circle (4pt); \filldraw[fill=white,draw=black,line width = 1pt]  (B\i) circle (4pt); }
\end{tikzpicture}} \end{array} = x_\beta
\end{array}
$$
By construction $x_\mu \in \mathsf{O}_t$ with $t = \pn(\pi) \le m$, so by the previous paragraph $x_\mu \in \langle \ef_{k,n} \rangle$,  and since  $\langle \ef_{k,n} \rangle$ is a two-sided ideal, $x_\pi = x_\tau x_\mu x_\beta \in \langle \ef_{k,n} \rangle$. 

Now we consider the half-integer levels and show that   $\ker \Phi_{k - \half,n} = \langle \ef_{k-\half,n} \rangle$ for $n,k \in \ZZ_{\ge 1}$.
We assume that $2k - 1 >n$, as $\ker \Phi_{k - \half,n} =(0)$ when $n \ge 2k-1$.   When $2k-1 > n$, then $\ef_{k - \half, n} = \ef_{k,n}$.
Furthermore by Theorem \ref{T:Phi}\,(b),  $\ker \Phi_{k - \half,n}$ equals the span of $\{ x_\pi \mid \pi \in \Pi_{2k-1} \subseteq \Pi_{2k}, \ \pi \text{ has more than $n$ blocks}\}$. The argument above shows that $\langle \ef_{k - \half, n} \rangle = \langle \ef_{k,n} \rangle$ is spanned by those orbit basis diagrams in $\Pi_{2k-1}$ having more than $n$ blocks.  When the  symmetry argument is applied in the proof, only permutations in $\S_{n-1} \subset \S_n$ are used, as it is
not necessary to permute the rightmost edge of a diagram, since it must remain connecting columns $k$ and $2k$.  \end{proof}

\begin{thm}\label{T:secfund} For all  $k, n \in \ZZ_{\ge 1}$, with $2k >n$, $\ef_{k,n}$ is an essential idempotent such that  $(\ef_{k,n})^2 = c_{k,n} \ef_{k,n}$, where 
\begin{equation}\label{eq:ckn}  c_{k,n} = \begin{cases} 
(-1)^{n+1-k}\,  (n+1-k)! & \quad \text{if} \ \  n \ge k> n/2,  \\
1 &  \quad \text{if} \ \ k > n.\end{cases} \end{equation}   Therefore, when $2k>n$ the kernel of 
$\Phi_{k,n}$ is generated by the idempotent $c_{k,n}^{-1} \ef_{k,n}$.   The kernel of the representation
$\Phi_{k-\half,n}$ is the ideal of $\P_{k-\half}(n)$  generated by the idempotent $c_{k,n}^{-1} \ef_{k-\half,n}$
= $c_{k,n}^{-1} \ef_{k,n}$, when $2k-1 > n$. 
\end{thm} 

\begin{proof}   By Theorem \ref{thm:generator},  $\ker\Phi_{k,n}$ is the ideal of $\P_k(n)$  generated by the element $\ef_{k,n}$ in \eqref{eq:ef}.   If $k > n$, then $(\ef_{k,n})^2 = \ef_{k,n}$ since, when squaring $\e_{k,n}$, there are no blocks in the middle to remove and no empty blocks entirely in the top or bottom rows to connect.

Now assume $n\ge k > n/2$,  and set  $s = n+1 -k$. For $0 \le t \le s$, let $v_t$ be the sum of the $\binom{s}{t}^2 t!$ rook orbit diagrams obtained by connecting $t$ empty vertices in the top of $\ef_{k,n}$ to $t$ empty vertices in the bottom. Then since $\e_{k,n}$ has $n+1$ blocks,
$$
(\e_{k,n})^2 = \sum_{t = 0}^s (n - (n+1) + t)_s \  v_t = \sum_{t = 0}^s  \left(\prod_{i =1}^s (t-i) \right) v_t.
$$
When $t=0$, the coefficient is $(n - (n+1))_s = (-1)_s = (-1) (-2) \cdots (-s) = (-1)^s s!$. All the other terms are zero, since for $0 < t \le s$, the coefficient is $\prod_{i =1}^s (t-i) = 0.$    Thus,  $ 
(\e_{k,n})^2= (-1)^{n+1-k} \, (n+1-k)! \, \e_{k,n}$, as claimed. 
(Example \ref{ex:essentialidempotent} below  demonstrates the vanishing term phenomenon.)  The assertion
about $\ker \Phi_{k-\half,n}$ follows directly from Theorem \ref{thm:generator} and the definition of
$\ef_{k-\half,n}$ in \eqref{eq:efhalf}.   
\end{proof}

\begin{example}\label{ex:essentialidempotent}
When $k = 5$ and $n = 6$,  we see that $c_{5,6} = 2$ and $(\e_{5,6})^2 = 2 \e_{5,6}$ by setting $n=6$
in the following expression:
\begin{align*}
(&\ef_{5,6})^2  = 
\begin{array}{c} 
\begin{tikzpicture}[xscale=.5,yscale=.5,line width=1.25pt] 
\foreach \i in {1,2,3,4,5}  { \path (\i,1.25) coordinate (T\i); \path (\i,.25) coordinate (B\i); } 
\filldraw[fill= black!12,draw=black!12,line width=4pt]  (T1) -- (T5) -- (B5) -- (B1) -- (T1);
\draw[blue] (T3) -- (B3);
\draw[blue] (T4) -- (B4);
\draw[blue] (T5) -- (B5);
\foreach \i in {1,2,3,4,5}  { \filldraw[fill=white,draw=black,line width = 1pt] (T\i) circle (4pt); \filldraw[fill=white,draw=black,line width = 1pt]  (B\i) circle (4pt); } 
\end{tikzpicture} \\
\begin{tikzpicture}[xscale=.5,yscale=.5,line width=1.25pt] 
\foreach \i in {1,2,3,4,5}  { \path (\i,1.25) coordinate (T\i); \path (\i,.25) coordinate (B\i); } 
\filldraw[fill= black!12,draw=black!12,line width=4pt]  (T1) -- (T5) -- (B5) -- (B1) -- (T1);
\draw[blue] (T3) -- (B3);
\draw[blue] (T4) -- (B4);
\draw[blue] (T5) -- (B5);
\foreach \i in {1,2,3,4,5}  { \filldraw[fill=white,draw=black,line width = 1pt] (T\i) circle (4pt); \filldraw[fill=white,draw=black,line width = 1pt]  (B\i) circle (4pt); } 
\end{tikzpicture}
\end{array}
 = (n - 7)(n-8)
\begin{array}{c}
\begin{tikzpicture}[xscale=.5,yscale=.5,line width=1.25pt] 
\foreach \i in {1,2,3,4,5}  { \path (\i,1.25) coordinate (T\i); \path (\i,.25) coordinate (B\i); } 
\filldraw[fill= black!12,draw=black!12,line width=4pt]  (T1) -- (T5) -- (B5) -- (B1) -- (T1);
\draw[blue] (T3) -- (B3);
\draw[blue] (T4) -- (B4);
\draw[blue] (T5) -- (B5);
\foreach \i in {1,2,3,4,5}  { \filldraw[fill=white,draw=black,line width = 1pt] (T\i) circle (4pt); \filldraw[fill=white,draw=black,line width = 1pt]  (B\i) circle (4pt); } 
\end{tikzpicture}
\end{array} \\
& + (n - 6)(n-7) \left(
\begin{array}{c}
\begin{tikzpicture}[xscale=.5,yscale=.5,line width=1.25pt] 
\foreach \i in {1,2,3,4,5}  { \path (\i,1.25) coordinate (T\i); \path (\i,.25) coordinate (B\i); } 
\filldraw[fill= black!12,draw=black!12,line width=4pt]  (T1) -- (T5) -- (B5) -- (B1) -- (T1);
\draw[blue] (T3) -- (B3);
\draw[blue] (T4) -- (B4);
\draw[blue] (T5) -- (B5);
\draw[purple] (T1) -- (B1);
\foreach \i in {1,2,3,4,5}  { \filldraw[fill=white,draw=black,line width = 1pt] (T\i) circle (4pt); \filldraw[fill=white,draw=black,line width = 1pt]  (B\i) circle (4pt); } 
\end{tikzpicture}
\end{array} 
+\begin{array}{c}
\begin{tikzpicture}[xscale=.5,yscale=.5,line width=1.25pt] 
\foreach \i in {1,2,3,4,5}  { \path (\i,1.25) coordinate (T\i); \path (\i,.25) coordinate (B\i); } 
\filldraw[fill= black!12,draw=black!12,line width=4pt]  (T1) -- (T5) -- (B5) -- (B1) -- (T1);
\draw[blue] (T3) -- (B3);
\draw[blue] (T4) -- (B4);
\draw[blue] (T5) -- (B5);
\draw[purple] (T1) -- (B2);
\foreach \i in {1,2,3,4,5}  { \filldraw[fill=white,draw=black,line width = 1pt] (T\i) circle (4pt); \filldraw[fill=white,draw=black,line width = 1pt]  (B\i) circle (4pt); } 
\end{tikzpicture}
\end{array} 
+\begin{array}{c}
\begin{tikzpicture}[xscale=.5,yscale=.5,line width=1.25pt] 
\foreach \i in {1,2,3,4,5}  { \path (\i,1.25) coordinate (T\i); \path (\i,.25) coordinate (B\i); } 
\filldraw[fill= black!12,draw=black!12,line width=4pt]  (T1) -- (T5) -- (B5) -- (B1) -- (T1);
\draw[blue] (T3) -- (B3);
\draw[blue] (T4) -- (B4);
\draw[blue] (T5) -- (B5);
\draw[purple] (T2) -- (B1);
\foreach \i in {1,2,3,4,5}  { \filldraw[fill=white,draw=black,line width = 1pt] (T\i) circle (4pt); \filldraw[fill=white,draw=black,line width = 1pt]  (B\i) circle (4pt); } 
\end{tikzpicture}
\end{array} 
+\begin{array}{c}
\begin{tikzpicture}[xscale=.5,yscale=.5,line width=1.25pt] 
\foreach \i in {1,2,3,4,5}  { \path (\i,1.25) coordinate (T\i); \path (\i,.25) coordinate (B\i); } 
\filldraw[fill= black!12,draw=black!12,line width=4pt]  (T1) -- (T5) -- (B5) -- (B1) -- (T1);
\draw[blue] (T3) -- (B3);
\draw[blue] (T4) -- (B4);
\draw[blue] (T5) -- (B5);
\draw[purple] (T2) -- (B2);
\foreach \i in {1,2,3,4,5}  { \filldraw[fill=white,draw=black,line width = 1pt] (T\i) circle (4pt); \filldraw[fill=white,draw=black,line width = 1pt]  (B\i) circle (4pt); } 
\end{tikzpicture}
\end{array} 
\right) 
\\
& + (n - 5)(n-6) \left(
\begin{array}{c}
\begin{tikzpicture}[xscale=.5,yscale=.5,line width=1.25pt] 
\foreach \i in {1,2,3,4,5}  { \path (\i,1.25) coordinate (T\i); \path (\i,.25) coordinate (B\i); } 
\filldraw[fill= black!12,draw=black!12,line width=4pt]  (T1) -- (T5) -- (B5) -- (B1) -- (T1);
\draw[blue] (T3) -- (B3);
\draw[blue] (T4) -- (B4);
\draw[blue] (T5) -- (B5);
\draw[purple] (T1) -- (B1);
\draw[purple] (T2) -- (B2);
\foreach \i in {1,2,3,4,5}  { \filldraw[fill=white,draw=black,line width = 1pt] (T\i) circle (4pt); \filldraw[fill=white,draw=black,line width = 1pt]  (B\i) circle (4pt); } 
\end{tikzpicture}
\end{array} 
+\begin{array}{c}
\begin{tikzpicture}[xscale=.5,yscale=.5,line width=1.25pt] 
\foreach \i in {1,2,3,4,5}  { \path (\i,1.25) coordinate (T\i); \path (\i,.25) coordinate (B\i); } 
\filldraw[fill= black!12,draw=black!12,line width=4pt]  (T1) -- (T5) -- (B5) -- (B1) -- (T1);
\draw[blue] (T3) -- (B3);
\draw[blue] (T4) -- (B4);
\draw[blue] (T5) -- (B5);
\draw[purple] (T1) -- (B2);
\draw[purple] (T2) -- (B1);
\foreach \i in {1,2,3,4,5}  { \filldraw[fill=white,draw=black,line width = 1pt] (T\i) circle (4pt); \filldraw[fill=white,draw=black,line width = 1pt]  (B\i) circle (4pt); } 
\end{tikzpicture}
\end{array} 
\right).
\end{align*}
\end{example}

\begin{remark} By  \eqref{eq:symmetry},  $\ker \Phi_{k,n} = \langle \sigma' \ef_{k,n} \sigma \rangle$ for $\sigma', \sigma  \in \S_k$. Thus,  the kernel is generated by any orbit rook diagram with propagating number equal to $\pn(\ef_{k,n}) = 2k-n-1$ when $n \ge  k > n/2$, and it is generated by any  permutation diagram (i.e., a diagram with propagating number equal to $k$) when $k > n$.
\end{remark} 

\begin{remark}\label{R:central}  The essential idempotent  $\ef_{k,n}$ is not central for all $n$  such that  $2k-1 > n$. This can be seen by multiplying  $\e_{k,n}$ on either side by the orbit basis element $x_\pi$ for $\pi = \{1, k+1 \mid 2, k+2 \mid \cdots \mid k-2, 2k -2 \mid k-1 \mid k, 2k -1, 2k \}$.  
However, as we discuss in Remark \ref{R:ek2k-1}, $\ef_{k,2k-1}$ is central in $\P_k(2k-1)$.  
The calculations below for $k = 5$ and $n = 7, 8, 9$  illustrate the exceptional behavior of $\ef_{k,2k-1}$.  
$$
\begin{array}{c}
\begin{array}{c} 
\begin{tikzpicture}[xscale=.4,yscale=.4,line width=1.25pt] 
\foreach \i in {1,2,3,4,5}  { \path (\i,1.25) coordinate (T\i); \path (\i,.25) coordinate (B\i); } 
\filldraw[fill= black!12,draw=black!12,line width=4pt]  (T1) -- (T5) -- (B5) -- (B1) -- (T1);
\draw[blue] (T4) -- (B4);
\draw[blue] (T5) -- (B5);
\foreach \i in {1,2,3,4,5}  { \filldraw[fill=white,draw=black,line width = 1pt] (T\i) circle (4pt); \filldraw[fill=white,draw=black,line width = 1pt]  (B\i) circle (4pt); } 
\end{tikzpicture} 
\end{array}
\\
\begin{array}{c} 
\begin{tikzpicture}[xscale=.4,yscale=.4,line width=1.25pt] 
\foreach \i in {1,2,3,4,5}  { \path (\i,1.25) coordinate (T\i); \path (\i,.25) coordinate (B\i); } 
\filldraw[fill= black!12,draw=black!12,line width=4pt]  (T1) -- (T5) -- (B5) -- (B1) -- (T1);
\draw[blue] (T1) -- (B1);
\draw[blue] (T2) -- (B2);
\draw[blue] (T3) -- (B3);
\draw[blue] (T4) -- (T5);
\draw[blue] (T5) -- (B5);
\foreach \i in {1,2,3,4,5}  { \filldraw[fill=white,draw=black,line width = 1pt] (T\i) circle (4pt); \filldraw[fill=white,draw=black,line width = 1pt]  (B\i) circle (4pt); } 
\end{tikzpicture} 
\end{array}
\end{array}
=
0
\qquad\hbox{and}\qquad
\begin{array}{c}
\begin{array}{c} 
\begin{tikzpicture}[xscale=.4,yscale=.4,line width=1.25pt] 
\foreach \i in {1,2,3,4,5}  { \path (\i,1.25) coordinate (T\i); \path (\i,.25) coordinate (B\i); } 
\filldraw[fill= black!12,draw=black!12,line width=4pt]  (T1) -- (T5) -- (B5) -- (B1) -- (T1);
\draw[blue] (T1) -- (B1);
\draw[blue] (T2) -- (B2);
\draw[blue] (T3) -- (B3);
\draw[blue] (T4) -- (T5);
\draw[blue] (T5) -- (B5);
\foreach \i in {1,2,3,4,5}  { \filldraw[fill=white,draw=black,line width = 1pt] (T\i) circle (4pt); \filldraw[fill=white,draw=black,line width = 1pt]  (B\i) circle (4pt); } 
\end{tikzpicture} 
\end{array}
\\
\begin{array}{c} 
\begin{tikzpicture}[xscale=.4,yscale=.4,line width=1.25pt] 
\foreach \i in {1,2,3,4,5}  { \path (\i,1.25) coordinate (T\i); \path (\i,.25) coordinate (B\i); } 
\filldraw[fill= black!12,draw=black!12,line width=4pt]  (T1) -- (T5) -- (B5) -- (B1) -- (T1);
\draw[blue] (T4) -- (B4);
\draw[blue] (T5) -- (B5);
\foreach \i in {1,2,3,4,5}  { \filldraw[fill=white,draw=black,line width = 1pt] (T\i) circle (4pt); \filldraw[fill=white,draw=black,line width = 1pt]  (B\i) circle (4pt); } 
\end{tikzpicture} 
\end{array}
\end{array}
=
\begin{array}{c} 
\begin{tikzpicture}[xscale=.4,yscale=.4,line width=1.25pt] 
\foreach \i in {1,2,3,4,5}  { \path (\i,1.25) coordinate (T\i); \path (\i,.25) coordinate (B\i); } 
\filldraw[fill= black!12,draw=black!12,line width=4pt]  (T1) -- (T5) -- (B5) -- (B1) -- (T1);
\draw[blue] (T4) -- (T5);
\draw[blue] (T5) -- (B5);
\foreach \i in {1,2,3,4,5}  { \filldraw[fill=white,draw=black,line width = 1pt] (T\i) circle (4pt); \filldraw[fill=white,draw=black,line width = 1pt]  (B\i) circle (4pt); } 
\end{tikzpicture} 
\end{array} \, , \hskip.9in
$$
$$
\begin{array}{c}
\begin{array}{c} 
\begin{tikzpicture}[xscale=.4,yscale=.4,line width=1.25pt] 
\foreach \i in {1,2,3,4,5}  { \path (\i,1.25) coordinate (T\i); \path (\i,.25) coordinate (B\i); } 
\filldraw[fill= black!12,draw=black!12,line width=4pt]  (T1) -- (T5) -- (B5) -- (B1) -- (T1);
\draw[blue] (T5) -- (B5);
\foreach \i in {1,2,3,4,5}  { \filldraw[fill=white,draw=black,line width = 1pt] (T\i) circle (4pt); \filldraw[fill=white,draw=black,line width = 1pt]  (B\i) circle (4pt); } 
\end{tikzpicture} 
\end{array}
\\
\begin{array}{c} 
\begin{tikzpicture}[xscale=.4,yscale=.4,line width=1.25pt] 
\foreach \i in {1,2,3,4,5}  { \path (\i,1.25) coordinate (T\i); \path (\i,.25) coordinate (B\i); } 
\filldraw[fill= black!12,draw=black!12,line width=4pt]  (T1) -- (T5) -- (B5) -- (B1) -- (T1);
\draw[blue] (T1) -- (B1);
\draw[blue] (T2) -- (B2);
\draw[blue] (T3) -- (B3);
\draw[blue] (T4) -- (T5);
\draw[blue] (T5) -- (B5);
\foreach \i in {1,2,3,4,5}  { \filldraw[fill=white,draw=black,line width = 1pt] (T\i) circle (4pt); \filldraw[fill=white,draw=black,line width = 1pt]  (B\i) circle (4pt); } 
\end{tikzpicture} 
\end{array}
\end{array}
=
0
\qquad\hbox{and}\qquad
\begin{array}{c}
\begin{array}{c} 
\begin{tikzpicture}[xscale=.4,yscale=.4,line width=1.25pt] 
\foreach \i in {1,2,3,4,5}  { \path (\i,1.25) coordinate (T\i); \path (\i,.25) coordinate (B\i); } 
\filldraw[fill= black!12,draw=black!12,line width=4pt]  (T1) -- (T5) -- (B5) -- (B1) -- (T1);
\draw[blue] (T1) -- (B1);
\draw[blue] (T2) -- (B2);
\draw[blue] (T3) -- (B3);
\draw[blue] (T4) -- (T5);
\draw[blue] (T5) -- (B5);
\foreach \i in {1,2,3,4,5}  { \filldraw[fill=white,draw=black,line width = 1pt] (T\i) circle (4pt); \filldraw[fill=white,draw=black,line width = 1pt]  (B\i) circle (4pt); } 
\end{tikzpicture} 
\end{array}
\\
\begin{array}{c} 
\begin{tikzpicture}[xscale=.4,yscale=.4,line width=1.25pt] 
\foreach \i in {1,2,3,4,5}  { \path (\i,1.25) coordinate (T\i); \path (\i,.25) coordinate (B\i); } 
\filldraw[fill= black!12,draw=black!12,line width=4pt]  (T1) -- (T5) -- (B5) -- (B1) -- (T1);
\draw[blue] (T5) -- (B5);
\foreach \i in {1,2,3,4,5}  { \filldraw[fill=white,draw=black,line width = 1pt] (T\i) circle (4pt); \filldraw[fill=white,draw=black,line width = 1pt]  (B\i) circle (4pt); } 
\end{tikzpicture} 
\end{array}
\end{array}
= (n - (2k-2))
\begin{array}{c} 
\begin{tikzpicture}[xscale=.4,yscale=.4,line width=1.25pt] 
\foreach \i in {1,2,3,4,5}  { \path (\i,1.25) coordinate (T\i); \path (\i,.25) coordinate (B\i); } 
\filldraw[fill= black!12,draw=black!12,line width=4pt]  (T1) -- (T5) -- (B5) -- (B1) -- (T1);
\draw[blue] (T4) -- (T5);
\draw[blue] (T5) -- (B5);
\foreach \i in {1,2,3,4,5}  { \filldraw[fill=white,draw=black,line width = 1pt] (T\i) circle (4pt); \filldraw[fill=white,draw=black,line width = 1pt]  (B\i) circle (4pt); } 
\end{tikzpicture} 
\end{array} \, ,
$$ 
$$
\begin{array}{c}
\begin{array}{c} 
\begin{tikzpicture}[xscale=.4,yscale=.4,line width=1.25pt] 
\foreach \i in {1,2,3,4,5}  { \path (\i,1.25) coordinate (T\i); \path (\i,.25) coordinate (B\i); } 
\filldraw[fill= black!12,draw=black!12,line width=4pt]  (T1) -- (T5) -- (B5) -- (B1) -- (T1);
\foreach \i in {1,2,3,4,5}  { \filldraw[fill=white,draw=black,line width = 1pt] (T\i) circle (4pt); \filldraw[fill=white,draw=black,line width = 1pt]  (B\i) circle (4pt); } 
\end{tikzpicture} 
\end{array}
\\
\begin{array}{c} 
\begin{tikzpicture}[xscale=.4,yscale=.4,line width=1.25pt] 
\foreach \i in {1,2,3,4,5}  { \path (\i,1.25) coordinate (T\i); \path (\i,.25) coordinate (B\i); } 
\filldraw[fill= black!12,draw=black!12,line width=4pt]  (T1) -- (T5) -- (B5) -- (B1) -- (T1);
\draw[blue] (T1) -- (B1);
\draw[blue] (T2) -- (B2);
\draw[blue] (T3) -- (B3);
\draw[blue] (T4) -- (T5);
\draw[blue] (T5) -- (B5);
\foreach \i in {1,2,3,4,5}  { \filldraw[fill=white,draw=black,line width = 1pt] (T\i) circle (4pt); \filldraw[fill=white,draw=black,line width = 1pt]  (B\i) circle (4pt); } 
\end{tikzpicture} 
\end{array}
\end{array}
=
0
\qquad\hbox{and}\qquad
\begin{array}{c}
\begin{array}{c} 
\begin{tikzpicture}[xscale=.4,yscale=.4,line width=1.25pt] 
\foreach \i in {1,2,3,4,5}  { \path (\i,1.25) coordinate (T\i); \path (\i,.25) coordinate (B\i); } 
\filldraw[fill= black!12,draw=black!12,line width=4pt]  (T1) -- (T5) -- (B5) -- (B1) -- (T1);
\draw[blue] (T1) -- (B1);
\draw[blue] (T2) -- (B2);
\draw[blue] (T3) -- (B3);
\draw[blue] (T4) -- (T5);
\draw[blue] (T5) -- (B5);
\foreach \i in {1,2,3,4,5}  { \filldraw[fill=white,draw=black,line width = 1pt] (T\i) circle (4pt); \filldraw[fill=white,draw=black,line width = 1pt]  (B\i) circle (4pt); } 
\end{tikzpicture} 
\end{array}
\\
\begin{array}{c} 
\begin{tikzpicture}[xscale=.4,yscale=.4,line width=1.25pt] 
\foreach \i in {1,2,3,4,5}  { \path (\i,1.25) coordinate (T\i); \path (\i,.25) coordinate (B\i); } 
\filldraw[fill= black!12,draw=black!12,line width=4pt]  (T1) -- (T5) -- (B5) -- (B1) -- (T1);
\foreach \i in {1,2,3,4,5}  { \filldraw[fill=white,draw=black,line width = 1pt] (T\i) circle (4pt); \filldraw[fill=white,draw=black,line width = 1pt]  (B\i) circle (4pt); } 
\end{tikzpicture} 
\end{array}
\end{array}
= (n - (2k-1))
\begin{array}{c} 
\begin{tikzpicture}[xscale=.4,yscale=.4,line width=1.25pt] 
\foreach \i in {1,2,3,4,5}  { \path (\i,1.25) coordinate (T\i); \path (\i,.25) coordinate (B\i); } 
\filldraw[fill= black!12,draw=black!12,line width=4pt]  (T1) -- (T5) -- (B5) -- (B1) -- (T1);
\draw[blue] (T4) -- (T5);
\foreach \i in {1,2,3,4,5}  { \filldraw[fill=white,draw=black,line width = 1pt] (T\i) circle (4pt); \filldraw[fill=white,draw=black,line width = 1pt]  (B\i) circle (4pt); } 
\end{tikzpicture} 
\end{array}\, .
$$ 
In the last case, the coefficient $n-(2k-1)$ equals 0 when $n = 2k-1$, and the two diagrams commute. \end{remark}

\begin{example} The surjection $\Phi_{2,3}: \P_2(3) \to \End_{\S_3}(\M_3^{\otimes 2})$ has a one-dimensional kernel.  The image is spanned by the the images of the 15 diagrams shown in the Hasse diagram of Figure \ref{fig:Hasse} subject to the dependence relation 
\begin{align*}
 0 & = 
 \begin{array}{c}
\begin{tikzpicture}[xscale=.38,yscale=.38,line width=1.25pt] 
\foreach \i in {1,2} 
{ \path (\i,.5) coordinate (T\i); \path (\i,-.5) coordinate (B\i); } 
\filldraw[fill= black!12,draw=black!12,line width=4pt]  (T1) -- (T2) -- (B2) -- (B1) -- (T1);
\foreach \i in {1,2} 
{ \filldraw[fill=white,draw=black,line width = 1pt] (T\i) circle (4pt); \filldraw[fill=white,draw=black,line width = 1pt]  (B\i) circle (4pt); } 
\end{tikzpicture}
\end{array} 
 =
\begin{array}{c}
\begin{tikzpicture}[xscale=.38,yscale=.38,line width=1.25pt] 
\foreach \i in {1,2} 
{ \path (\i,.5) coordinate (T\i); \path (\i,-.5) coordinate (B\i); } 
\filldraw[fill= black!12,draw=black!12,line width=4pt]  (T1) -- (T2) -- (B2) -- (B1) -- (T1);
\foreach \i in {1,2} 
{ \fill (T\i) circle (4pt); \fill (B\i) circle (4pt); } 
\end{tikzpicture}
\end{array}
-
\begin{array}{c}
\begin{tikzpicture}[xscale=.38,yscale=.38,line width=1.25pt] 
\foreach \i in {1,2} 
{ \path (\i,1.25) coordinate (T\i); \path (\i,.25) coordinate (B\i); } 
\filldraw[fill= black!12,draw=black!12,line width=4pt]  (T1) -- (T2) -- (B2) -- (B1) -- (T1);
\draw[blue] (B1) -- (B2);
\foreach \i in {1,2} 
{ \fill (T\i) circle (4pt); \fill (B\i) circle (4pt); } 
\end{tikzpicture}
\end{array}
-
\begin{array}{c}
\begin{tikzpicture}[xscale=.38,yscale=.38,line width=1.25pt] 
\foreach \i in {1,2} 
{ \path (\i,1.25) coordinate (T\i); \path (\i,.25) coordinate (B\i); } 
\filldraw[fill= black!12,draw=black!12,line width=4pt]  (T1) -- (T2) -- (B2) -- (B1) -- (T1);
\draw[blue] (B1) -- (T1);
\foreach \i in {1,2} 
{ \fill (T\i) circle (4pt); \fill (B\i) circle (4pt); } 
\end{tikzpicture}
\end{array}
-
\begin{array}{c}
\begin{tikzpicture}[xscale=.38,yscale=.38,line width=1.25pt] 
\foreach \i in {1,2} 
{ \path (\i,1.25) coordinate (T\i); \path (\i,.25) coordinate (B\i); } 
\filldraw[fill= black!12,draw=black!12,line width=4pt]  (T1) -- (T2) -- (B2) -- (B1) -- (T1);
\draw[blue] (B1) -- (T2);
\foreach \i in {1,2} 
{ \fill (T\i) circle (4pt); \fill (B\i) circle (4pt); } 
\end{tikzpicture}
\end{array}
-
\begin{array}{c}
\begin{tikzpicture}[xscale=.38,yscale=.38,line width=1.25pt] 
\foreach \i in {1,2} 
{ \path (\i,1.25) coordinate (T\i); \path (\i,.25) coordinate (B\i); } 
\filldraw[fill= black!12,draw=black!12,line width=4pt]  (T1) -- (T2) -- (B2) -- (B1) -- (T1);
\draw[blue] (B2) -- (T1);
\foreach \i in {1,2} 
{ \fill (T\i) circle (4pt); \fill (B\i) circle (4pt); } 
\end{tikzpicture}
\end{array}
-
\begin{array}{c}
\begin{tikzpicture}[xscale=.38,yscale=.38,line width=1.25pt] 
\foreach \i in {1,2} 
{ \path (\i,1.25) coordinate (T\i); \path (\i,.25) coordinate (B\i); } 
\filldraw[fill= black!12,draw=black!12,line width=4pt]  (T1) -- (T2) -- (B2) -- (B1) -- (T1);
\draw[blue] (B2) -- (T2);
\foreach \i in {1,2} 
{ \fill (T\i) circle (4pt); \fill (B\i) circle (4pt); } 
\end{tikzpicture}
\end{array}
-
\begin{array}{c}
\begin{tikzpicture}[xscale=.38,yscale=.38,line width=1.25pt] 
\foreach \i in {1,2} 
{ \path (\i,1.25) coordinate (T\i); \path (\i,.25) coordinate (B\i); } 
\filldraw[fill= black!12,draw=black!12,line width=4pt]  (T1) -- (T2) -- (B2) -- (B1) -- (T1);
\draw[blue] (T2) -- (T1);
\foreach \i in {1,2} 
{ \fill (T\i) circle (4pt); \fill (B\i) circle (4pt); } 
\end{tikzpicture}
\end{array}
+ 2 
\begin{array}{c}
\begin{tikzpicture}[xscale=.38,yscale=.38,line width=1.25pt] 
\foreach \i in {1,2} 
{ \path (\i,1.25) coordinate (T\i); \path (\i,.25) coordinate (B\i); } 
\filldraw[fill= black!12,draw=black!12,line width=4pt]  (T1) -- (T2) -- (B2) -- (B1) -- (T1);
\draw[blue] (T1) -- (B2);\draw[blue] (T1) -- (B1);
\draw[blue] (B1) -- (B2);
\foreach \i in {1,2} 
{ \fill (T\i) circle (4pt); \fill (B\i) circle (4pt); } 
\end{tikzpicture}
\end{array}
+2
\begin{array}{c}
\begin{tikzpicture}[xscale=.38,yscale=.38,line width=1.25pt] 
\foreach \i in {1,2} 
{ \path (\i,1.25) coordinate (T\i); \path (\i,.25) coordinate (B\i); } 
\filldraw[fill= black!12,draw=black!12,line width=4pt]  (T1) -- (T2) -- (B2) -- (B1) -- (T1);
\draw[blue] (T2) -- (B2)--(T1)--(T2);
\foreach \i in {1,2} 
{ \fill (T\i) circle (4pt); \fill (B\i) circle (4pt); } 
\end{tikzpicture}
\end{array} \\ 
& \hskip1.0in
+2
\begin{array}{c}
\begin{tikzpicture}[xscale=.38,yscale=.38,line width=1.25pt] 
\foreach \i in {1,2} 
{ \path (\i,1.25) coordinate (T\i); \path (\i,.25) coordinate (B\i); } 
\filldraw[fill= black!12,draw=black!12,line width=4pt]  (T1) -- (T2) -- (B2) -- (B1) -- (T1);
\draw[blue] (T2) -- (B1) -- (T1) -- (T2);
\foreach \i in {1,2} 
{ \fill (T\i) circle (4pt); \fill (B\i) circle (4pt); } 
\end{tikzpicture}
\end{array}
+2
\begin{array}{c}
\begin{tikzpicture}[xscale=.38,yscale=.38,line width=1.25pt] 
\foreach \i in {1,2} 
{ \path (\i,1.25) coordinate (T\i); \path (\i,.25) coordinate (B\i); } 
\filldraw[fill= black!12,draw=black!12,line width=4pt]  (T1) -- (T2) -- (B2) -- (B1) -- (T1);
\draw[blue] (B1) -- (B2);\draw[blue] (T2) -- (B2);
\draw[blue] (B1) -- (T2);
\foreach \i in {1,2} 
{ \fill (T\i) circle (4pt); \fill (B\i) circle (4pt); } 
\end{tikzpicture}
\end{array}
+\begin{array}{c}
\begin{tikzpicture}[xscale=.38,yscale=.38,line width=1.25pt] 
\foreach \i in {1,2} 
{ \path (\i,1.25) coordinate (T\i); \path (\i,.25) coordinate (B\i); } 
\filldraw[fill= black!12,draw=black!12,line width=4pt]  (T1) -- (T2) -- (B2) -- (B1) -- (T1);
\draw[blue] (B2) -- (B1);\draw[blue] (T2) -- (T1);
\foreach \i in {1,2} 
{ \fill (T\i) circle (4pt); \fill (B\i) circle (4pt); } 
\end{tikzpicture}
\end{array}
+
\begin{array}{c}
\begin{tikzpicture}[xscale=.38,yscale=.38,line width=1.25pt] 
\foreach \i in {1,2} 
{ \path (\i,1.25) coordinate (T\i); \path (\i,.25) coordinate (B\i); } 
\filldraw[fill= black!12,draw=black!12,line width=4pt]  (T1) -- (T2) -- (B2) -- (B1) -- (T1);
\draw[blue] (T2) -- (B1);
\draw[blue] (T1) -- (B2);
\foreach \i in {1,2} 
{ \fill (T\i) circle (4pt); \fill (B\i) circle (4pt); } 
\end{tikzpicture}
\end{array}
+
\begin{array}{c}
\begin{tikzpicture}[xscale=.38,yscale=.38,line width=1.25pt] 
\foreach \i in {1,2} 
{ \path (\i,1.25) coordinate (T\i); \path (\i,.25) coordinate (B\i); } 
\filldraw[fill= black!12,draw=black!12,line width=4pt]  (T1) -- (T2) -- (B2) -- (B1) -- (T1);
\draw[blue] (T1) -- (B1);
\draw[blue] (T2) -- (B2);
\foreach \i in {1,2} 
{ \fill (T\i) circle (4pt); \fill (B\i) circle (4pt); } 
\end{tikzpicture}
\end{array}
- 6
\begin{array}{c}
\begin{tikzpicture}[xscale=.38,yscale=.38,line width=1.25pt] 
\foreach \i in {1,2} 
{ \path (\i,1.25) coordinate (T\i); \path (\i,.25) coordinate (B\i); } 
\filldraw[fill= black!12,draw=black!12,line width=4pt]  (T1) -- (T2) -- (B2) -- (B1) -- (T1);
\draw[blue] (T2) -- (B1);\draw[blue] (T1) -- (B2);\draw[blue] (T1) -- (B1);\draw[blue] (T2) -- (B2);
\draw[blue] (T1) -- (T2);\draw[blue] (B1) -- (B2);
\foreach \i in {1,2} 
{ \fill (T\i) circle (4pt); \fill (B\i) circle (4pt); } 
\end{tikzpicture}
\end{array}.
\end{align*}
\end{example}

\subsection{$\ker \Phi_{k,n} = \langle \ef_{n,n} \rangle$ when $k \ge n$}

In this section,  we show that under the natural embedding of $\P_n(n)$ into $\P_k(n)$ for $k  \ge n$,  the kernel of $\Phi_{k,n}$ is in fact generated by $\ef_{n,n}$.   Towards
this purpose, we adopt the following conventions.  When $a$ is  a basis element that is either a diagram or an orbit diagram  in $\P_k(n)$, and $b$ is similarly one in $\P_\ell(n)$, let  $a \ot b$ be the diagram in $\P_{k+\ell}(n)$ obtained by juxtaposing
$b$ to the right of $a$.   Write  $\langle a \rangle_k$ for the two-sided ideal of $\P_k(n)$ generated
by $a$.  

For  $k,n \in \ZZ_{\ge 1}$ with $2k > n$, suppose $\pi = \pi_{k,n}$ is  the set partition associated with the essential idempotent $\ef_{k,n}$.     Then $\ef_{k,n} = x_\pi$,   and for 
$\ell \in \ZZ_{> 0}$, the expression for
$\ef_{k,n} \ot  (\blvertedge)^{\ot \ell}$ in the diagram basis is

$$\ef_{k,n} \ot  (\blvertedge)^{\ot \ell}  =  x_{\pi} \ot  (\blvertedge)^{\ot \ell} = \sum_{\pi \preceq \varrho}
\mu_{2k}(\pi,\varrho) \, d_{\varrho} \ot  (\blvertedge)^{\ot \ell}.$$

We convert this expression to one in the orbit basis, observing that  
for each $j=1,\dots, \ell$,  the vertices $k+j$ and $2(k+j)$ must be in the same block in any orbit basis element that
 appears with nonzero coefficient:
 
\begin{equation}\label{eq:efexpress} \ef_{k,n} \ot  (\blvertedge)^{\ot \ell}  =  \sum_{\pi \preceq \varrho} \mu_{2k}(\pi,\varrho)
\left(\sum_{\varrho \preceq \nu} \zeta_{2k}(\varrho,\nu) x_{\nu} \ot  (\vertedge)^{\ot \ell}\right) + \sum_{\tau} \xi_\tau x_\tau,\end{equation}
where $\tau$ runs over the set partitions of $[1,2(k+\ell)]$ which have at least one block that contains
$k+j$, $2(k+j)$ and some other element of $[1,2(k+\ell)]$ for some $j$ such that $1 \leq j \leq \ell$, and $\xi_\tau \in  \ZZ$.  

Now it follows from \eqref{eq:efexpress}  and the fact that the matrices $\mu_{2k}$ and $\zeta_{2k}$ are inverses
that 
\begin{align*}\label{eq:efexpress2} \ef_{k,n} \ot  (\blvertedge)^{\ot \ell}  &= \sum_{\pi \preceq \nu} \Bigg( \sum_{\varrho, \pi \preceq \varrho \preceq \nu}  \mu_{2k}(\pi,\varrho)\zeta_{2k}(\varrho,\nu)
\Bigg) x_{\nu} \ot  (\vertedge)^{\ot \ell}   + \sum_{\tau} \xi_\tau x_\tau \\
& = x_\pi \ot  (\vertedge)^{\ot \ell} + \sum_{\tau} \xi_\tau x_\tau =  \e_{k,n}\ot  (\vertedge)^{\ot \ell} + \sum_{\tau} \xi_\tau x_\tau.
\end{align*}
By the rule for multiplication in the orbit basis, $(\vertedge)^{\ot (k+\ell)}\cdot x_\tau \cdot (\vertedge)^{\ot (k+\ell)} = 0$ when $\tau$ has any horizontal edges.  Therefore,
\begin{equation}\label{eq:contain}(\vertedge)^{\ot (k+\ell)} \cdot \left( \ef_{k,n} \ot   (\blvertedge)^{\ot \ell}\right)\cdot (\vertedge)^{\ot (k+\ell)} = \ef_{k,n} \ot    
 (\vertedge)^{\ot \ell}\end{equation} 
 holds for $2k > n$ and all $\ell \ge 1$.  Consequently, we have the following:
 
 \begin{prop}\label{P:one} For $k,\ell \in \ZZ_{\ge 1}$, 
  \begin{itemize}  \item[{\rm (i)}]  if $2k > n$,  then $\ef_{k,n} \ot    
 (\vertedge)^{\ot \ell}$ is in the two-sided ideal of $\P_{k+\ell}(n)$
  generated by 
 $\ef_{k,n} \ot   (\blvertedge)^{\ot \ell}$;
 \item[{\rm(ii)}]  if $2k-1 > n$, then $\ef_{k-\half,n} \ot    
 (\vertedge)^{\ot \ell}$    is in the two-sided ideal of $\P_{k+\ell-\half}(n)$
  generated by 
 $\ef_{k,n} \ot   (\blvertedge)^{\ot \ell}$. 
 \end{itemize}
 \end{prop}

\begin{proof}  Part (i)  follows immediately from \eqref{eq:contain}.   
 For  part (ii),  note first that $\ef_{k-\half,n} = \ef_{k,n}$ when $2k-1 > n$.    
 Since $(\vertedge)^{\ot k+\ell} \in \P_{k+\ell-\half}(n)$, equation \eqref{eq:contain}  
 allows us to conclude 
that 
$$\ef_{k-\half,n} \ot (\vertedge)^{\ot \ell}  = \ef_{k,n} \ot (\vertedge)^{\ot \ell} \in  \langle \ef_{k,n} \ot (\blvertedge)^{\ot \ell}\rangle_{k+\ell-\half}.$$\end{proof}

\begin{thm}\label{T:gensk>n} \begin{itemize} \item[{\rm (a)}]   The kernel of the representation $\Phi_{k,n}: \P_k(n) \to  \End_{\S_n}(\M_n^{\ot k})$ is
generated as a two-sided ideal by the element $\ef_{n,n} \ot  (\blvertedge)^{\ot k-n}$ for   $k \in \ZZ_{\ge n}$. 
\item[{\rm (b)}]  The kernel of the representation $\Phi_{k-\half,n}: \P_{k-\half}(n)  \to \End_{\S_{n-1}}(\M_n^{\ot k-1})$ is generated
as a two-sided ideal by the element $\ef_{n,n} \ot  (\blvertedge)^{\ot k-n}$ for $n>1$ and $k \in \ZZ_{\ge n}$.  
\end{itemize}
\end{thm}

\begin{proof}  (a)  We know from Theorem 5.6 that for $k \in \half \ZZ$ and $2k > n$, the essential idempotent  $\ef_{k,n}$ generates the kernel of $\Phi_{k,n}$ as a two-sided ideal.   Moreover, 
$\ef_{n,n} = (\novertedge\!) \ot (\vertedge)^{n-1}$.  Then 
 for $\ell \ge 0$,
\begin{align} \begin{split}\label{eq:square} \big(\ef_{n,n} \ot (\vertedge)^{\ot \ell}\big)^2  & = (n-(n+\ell+1)) \ef_{n,n}  \ot (\vertedge)^{\ot \ell}
+ (n-(n+\ell)) \ef_{n+\ell,n} \\
&= -(\ell+1) \ef_{n,n} \ot (\vertedge)^{\ot \ell} - \ell \ef_{n+\ell,n}.\end{split}\end{align}
Therefore,  $\ef_{n+\ell,n} \in \langle \ef_{n,n} \ot (\vertedge)^{\ot \ell}\rangle_{n+\ell}$.  By  Proposition \ref{P:one}\,(i), $\langle \ef_{n,n} \ot (\vertedge)^{\ot \ell}\rangle_{n+\ell} 
\subseteq  \langle \ef_{n,n} \ot (\blvertedge)^{\ot \ell}\rangle_{n+\ell}$ for $\ell \ge  0$, which implies that
$\ker \Phi_{n+\ell,n} = \langle \ef_{n+\ell,n} \rangle_{n+\ell} \subseteq \langle \ef_{n,n} \ot (\blvertedge)^{\ot \ell}\rangle_{n+\ell}$.      Since 
$\ef_{n,n} \ot  (\blvertedge)^{\ot \ell}$  acts as $\ef_{n,n} \ot \mathrm{Id}_{\M_n}^{\ot \ell}$  on  $\M_n^{\ot n+\ell}$ for all $\ell \ge 0$,
 it lies in the kernel of $\Phi_{n+\ell}$, and the reverse containment holds, 
$\langle \ef_{n,n} \ot  (\blvertedge)^{\ot \ell}\rangle_{n+\ell} \subseteq  
\langle \ef_{n+\ell, n} \rangle_{n+\ell}$.     Hence,  $\ker \Phi_{n+\ell,n} = \langle \ef_{n+\ell,n}\rangle_{n+\ell}
= \langle \ef_{n,n} \ot (\blvertedge)^{\ot \ell}\rangle_{n+\ell}$ for all $\ell \ge 0$.      Letting $k = n+\ell$, we
have part (a):   $\ker \Phi_{k,n} = \langle \ef_{n,n} \ot (\blvertedge)^{\ot k-n}\rangle_{k}$ for all $k \ge n$.

For part (b),  using the fact that  $\ef_{k-\half,n} = \ef_{k,n}$ for $k \ge n >1$, and applying 
\eqref{eq:square}  with $\ell = k-n$,  we obtain
\begin{align*} \big(\ef_{n,n} \ot (\vertedge)^{\ot k-n}\big)^2   &=
-(k-n+1) \ef_{n,n} \ot (\vertedge)^{\ot k-n} - (k-n) \ef_{k,n}\\
&=-(k-n+1) \ef_{n,n} \ot (\vertedge)^{\ot k-n} - (k-n) \ef_{k-\half,n}.\end{align*}
Consequently,  $\ef_{k-\half,n} \in \langle \ef_{n,n} \ot (\vertedge)^{\ot k-n}\rangle_{k-\half}
\subseteq \langle \ef_{n,n} \ot (\blvertedge)^{\ot k-n}\rangle_{k-\half}$, where
the last containment comes from Proposition \ref{P:one}\,(ii).    Now
$\ef_{n,n} \ot (\blvertedge)^{\ot k-n}$  acts as $\ef_{n,n} \ot \mathrm{Id}_{\M_n}^{\ot k-n}$
on $\M_n^{\ot k-1} \ot \mathsf{u}_n$, which is isomorphic to $\M_n^{\ot k-1}$ as an $\S_{n-1}$-module.
Therefore,  $\ef_{n,n} \ot (\blvertedge)^{\ot k-n} \in \ker \Phi_{k-\half,n} = \langle \ef_{k-\half,n} \rangle_{k-\half}$.
From these results, we see that
$\ker \Phi_{k-\half,n} = \langle \ef_{n,n} \ot (\blvertedge)^{\ot k-n}\rangle_{k-\half}$, as desired.   \end{proof}
 
\subsection{Fundamental theorems of invariant theory for $\S_n$}\label{S:funthms}

The fundamental theorems of invariant theory for the symmetric group $\S_n$ describe generators and relations 
for the invariants of  the natural realization of $\S_n$ as permutations on the $n$-dimensional permutation module $\M_n$
and  the $\S_n$-action on tensor powers of $\M_n$.    The surjection $\Phi_{k,n}: \P_k(n) \rightarrow \End_{\S_n}(\M_n^{\ot k})$, which
is an isomorphism when $n \ge 2k$,  gives a representation-theoretic approach to the fundamental theorems using the fact that the
partition algebra $\P_k(n)$ has a presentation as a unital algebra with generators $\mathfrak{s}_i  \  (1 \le i \le k-1)$, 
$\mathfrak{p}_i \ (1 \le i \le k)$, and $\mathfrak{p}_{i+\half} =\mathfrak{b}_i \  (1 \le i \le k-1)$
and with the relations given in (a)-(c) of Theorem \ref{T:present}.    
As a consequence of that result and Theorems
\ref{thm:generator} and \ref{T:gensk>n}, we have the following:
\smallskip

\begin{thm}\label{T:2ndfund}(Second Fundamental Theorem of Invariant Theory for $\S_n$) \, For all $k,n\in \ZZ_{\ge 1}$,   $\im \Phi_{k,n}
= \End_{\S_n}(\modu^{\ot k})$ is generated by the partition algebra generators and relations in Theorem \ref{T:present} (a)-(c) together with the one additional relation $\ef_{k,n} = 0$ in the case that $2k > n$.   When $k \ge n$,  the relation $\ef_{k,n} = 0$ can be replaced with  $\ef_{n,n} = 0$.
\end{thm} 
 
Theorem \ref{thm:generator} shows that $\ef_{k,n}$ generates the kernel of $\Phi_{k,n}$, Theorem \ref{T:secfund} shows  $\ef_{k,n}$ is an essential idempotent, and equation \eqref{eq:ekn} gives the expression for $\ef_{k,n}$ in the
diagram basis.   The last statement in Theorem \ref{T:2ndfund} comes from identifying $\ef_{n,n}$ with its embedded image in $\P_k(n)$
in Theorem \ref{T:gensk>n}. 

\begin{remark}\label{R:nosmaller}  We know from Theorem \ref{T:2ndfund} that   $\mathsf{g}_n:=\ef_{n,n} \ot (\blvertedge)^{\ot k-n}$ (which we have identified with $\ef_{n,n}$)  generates $\ker \Phi_{k,n}$ as a two-sided ideal for $k \ge n$.  It is 
natural to ask if $\ker \Phi_{k,n}$ can be principally generated by any of the elements 
$$\mathsf{g}_\ell:=\ef_{\ell,n} \ot (\blvertedge)^{\ot k-\ell} = (\novertedge)^{\ot n+1-\ell} \ot  (\vertedge)^{\ot 2\ell-n-1}\ot (\blvertedge)^{\ot k-\ell}$$ for  $\ell$ such that $k \ge n > \ell \ge \half(n+1)$.  
Let  $J = J_{k-1}$, the ideal of $\P_k(n)$ spanned by all the diagram basis elements with propagating number less than or equal to $k-1$,  and set $I_\ell = \langle \mathsf{g}_\ell \rangle$.   Since $\ef_{n,n} \ot  (\blvertedge)^{\ot k-n} = (\novertedge) \ot (\vertedge)^{\ot n-1} \ot (\blvertedge)^{\ot k-n} \equiv  \mathsf{I}_k$ modulo $J$, we have
$(I_n + J)/J \cong \P_k(n)/J$.   
We claim that $(I_{\ell}+J)/J\neq (I_{m}+J)/J$ for $\ell < m\le n$.    Indeed,  $\mathsf{g}_\ell \equiv \sum_{\sigma \in \S_{n+1-\ell}}\,\sigma\ot (\blvertedge)^{\ot k+\ell-n-1}$ 
and $\mathsf{g}_m \equiv \sum_{\tau \in \S_{n+1-m}} \tau \ot (\blvertedge)^{\ot k+m-n-1}$ modulo $J$,  
and because $ \sum_{\sigma \in \S_{n+1-\ell}}\,\sigma$  and $ \sum_{\tau \in \S_{n+1-m}}\,\tau \ot \blvertedge^{\ot m-\ell}$ generate different ideals in $\CC \S_{n+1-\ell}$,  the ideals $I_\ell$ and $I_m$
cannot be equal.   Therefore,  $n$ is the smallest value such that $\ef_{n,n}$ generates
$\ker \Phi_{k,n}$ for all $k \geq n$. 
\end{remark} 

\section{Connections with Primitive Central Idempotents} 
\subsection{The irreducible module labeled by the partition $[n-k,k]$}
For $\lambda$ a  partition of $n$ (which we write $\lambda \vdash n$),  the primitive central idempotent in $\CC \S_n$ corresponding to the irreducible $\S_n$-module $\S_n^\lambda$ indexed by $\lambda$ is 
\begin{equation}\label{def:projector}
\proj_\lambda = \frac{f^\lambda}{n!} \sum_{\sigma \in \S_n}\chi^\lambda_{\S_n}\!(\sigma^{-1}) \, \sigma, 
\end{equation}
where $\chi^\lambda_{\S_n}$ is the character of $\S_n^\lambda$,  and 
 $f^\lambda$ is the dimension of $\S_n^\lambda$, which can  be computed by the well-known hook formula
 $f^\lambda = n!/\prod_{b} h(b)$. The denominator is the product of the hook lengths $h(b)$ of the boxes $b$ in the Young diagram of $\lambda$  (see \cite[(2.32)]{FH}).
When $n \geq 2k$,   the partition algebra $\P_k(n)$ can be identified with the centralizer algebra $\End_{\S_n}(\modu^{\ot k})$,
and its irreducible modules $\P_{k,n}^\lambda$ can also be indexed by the partitions of $n$. 
Since $\S_n$ and $\P_k(n)$ have commuting actions on  $\modu^{\ot k}$,  by Schur-Weyl duality  $\modu^{\ot k}$ has a multiplicity-free
decomposition into a direct sum of bimodules $\S_n^\lambda \ot \P_{k,n}^\lambda$  for $\S_n \times \P_k(n)$.  
Therefore, if  
\begin{equation} \Psi_{k,n}: \S_n \rightarrow \End(\modu^{\ot k})\end{equation}
is the representation of $\S_n$ on $\modu^{\ot k}$ given by the diagonal action,
 the projection map from $\modu^{\otimes k}$ onto the isotypic component 
$\S_n^\lambda \ot \P_{k,n}^\lambda$ corresponding to $\lambda \vdash n$  is given by 
\begin{equation}\label{def:projector}
 \Psi_{k,n}(\ve_{\lambda}) = \frac{f^\lambda}{n!} \sum_{\sigma \in \S_n}\chi^\lambda_{\S_n}\!(\sigma^{-1})\, \Psi_{k,n}(\sigma).
\end{equation}
Now  $\left(\Psi_{k,n}(\ve_{\lambda})\right)^2 =  \Psi_{k,n}((\ve_{\lambda})^2) = \Psi_{k,n}(\ve_\lambda),$ 
and  $\Psi_{k,n}(\ve_{\lambda})$  commutes with $\Psi_{k,n}(\tau)$ for all $\tau \in \S_n$, because it is the image 
of the central element $\proj_\lambda$ of $\CC \S_n$.    Thus,  the idempotent $ \Psi_{k,n}(\ve_{\lambda})$ lives  in the centralizer 
$ \End_{\S_n}(\modu^{\ot k})$, \emph{which we identify  with $\P_k(n)$, whenever  $n \geq 2k$}.   With this identification, 
 $\Psi_{k,n}(\ve_\lambda)$ for any $\lambda \vdash n$  can be expressed in the diagram basis $\{d_\pi \mid \pi \in \Pi_{2k}\}$ or in
 the orbit basis $\{x_\pi \mid \pi \in \Pi_{2k}\}$ of $\P_k(n)$.   Our interest  is in one specific
 choice of $\lambda$, namely, the  two-part partition $[n-k,k] \vdash n$.  When $n \ge 2k$, the corresponding irreducible $\P_k(n)$-module
 $\P_{k,n}^{[n-k,k]}$ is one-dimensional  (see \cite[Cor.~5.14]{BHH}).  The diagram basis elements $d_\pi$ with $\mathsf{pn}(\pi) < k$ act
 as $0$ on $\P_{k,n}^{[n-k,k]}$, and the permutation diagrams act as the identity element.  Thus, the character of $\P_{k,n}^{[n-k,k]}$ is
\begin{equation}\label{character-value}
\chi_{\P_k(n)}^{[n-k,k]}(d_\pi) = \begin{cases}
1 & \qquad \hbox{if} \ \   {\pn(d_\pi) = k,}\\
0 & \qquad \hbox{if} \ \  {\pn(d_\pi) <k,}
\end{cases}
\end{equation}
and when the $\P_k(n)$-module $\P_{k,n}^{[n-k,k]}$ restricted to the subalgebra $\CC \S_k$, it gives the trivial representation \cite[Cor 4.2.3(b)]{H}.
For this particular choice of partition,  we show the following:
  
\begin{thm}\label{IdempotentOrbitBasis} Assume $k,n \in \ZZ_{\ge 1}$ and $n \ge 2k$.   
The expression for  $\Psi_{k,n}(\ve_{[n-k,k]})$ in the orbit basis is given by
\begin{equation}\label{eq:idem}
 \Psi_{k,n}(\ve_{[n-k,k]})= f^{[n-k,k]} \sum_{\pi \in \mathcal{R}_{2k}} c(\pi,n) x_\pi,
\end{equation}
where the sum is over the rook partitions $\pi$ of  $\{1, 2,\ldots, 2k\}$,  
\begin{align}
c(\pi,n) &= \frac{ (-1)^{k - \pn(\pi)} (k-\pn(\pi))!}{(n)_{|\pi|}} =\frac{ (-1)^{k - \pn(\pi)} (k-\pn(\pi))!}{(n)_{2k-\pn(\pi)}},  \quad \text{and} \label{projection-coeff}  \\
f^{[n-k,k]} &=  \frac{1}{k!}(n-2 k + 1)  (n)_{k-1} = \frac{(n-2 k + 1) n!}{k! (n-k+1)!} = \frac{(n-2 k + 1)}{(n-k+1) } \binom{n}{k}.
\label{hook-dimension}
 \end{align}
Since  $|\pi| = 2k - \pn(\pi)$ for $\pi \in \mathcal{R}_{2k}$, the coefficient $c(\pi,n)$ depends only on $n, k,$ and the
propagating number  $\pn(\pi)$. \end{thm}

The expression for  $f^{[n-k,k]}$ in \eqref{hook-dimension} follows directly from the hook-length formula.    We suppose that $\Xie_{k,n}$ is the right-hand expression in \eqref{eq:idem}, so that 
\begin{equation}\label{eq:Up}  \Xie_{k,n} := f^{[n-k,k]}  \sum_{\pi \in \mathcal{R}_{2k}} c(\pi,n) x_\pi  =   \frac{1}{k!}(n-2 k + 1)  (n)_{k-1}  \sum_{\pi \in \mathcal{R}_{2k}} c(\pi,n) x_\pi. \end{equation}
Before beginning the proof of Theorem \ref{IdempotentOrbitBasis}, we make a few comments and present some examples. 

\begin{remark}\label{R:Xi}  For  $k \ge 1$, $\Xie_{k,n}$  is a well-defined element of the partition algebra $\P_k(n)$ for all  but finitely many values of $n$ which depend on $k$. When $\Xie_{k,n}$ is defined, it is nonzero,  since  $\{x_\pi\mid\pi\in\Pi_{2k}\}$ is a basis of $\P_k(n)$,  and not all the $c(\pi,n)$ are zero.   Our aim
in proving Theorem \ref{IdempotentOrbitBasis} is to show that when $n \geq 2k$, 
$\Xie_{k,n}$  equals
$\Psi_{k,n}(\ve_{[n-k,k]})$, hence, $\Xie_{k,n}$ is a central idempotent of $\P_k(n)$.  As a result of Theorem \ref{IdempotentOrbitBasis}, we will be able to identify $\P_{k,n}^{[n-k,k]}$  with the one-dimensional ideal of
$\P_k(n)$ spanned by  $\Xie_{k,n} =\Psi_{k,n}(\ve_{[n-k,k]})$  when $n\geq 2k$.  

The element $\Xie_{k,n}$ is defined when $n = 2k-1$, as the terms $(n-2k+1)c(\pi,n)$ are exactly 0, except for the set partition $\pi$ consisting of all singletons,
as illustrated in the examples below.    If $n=2k-1$, then 
$\dimm \Phi_{k,n}\left(\P_k(n)\right)= \dimm \P_k(n) \, -1  =  \mathsf{B}(2k)-1$, and it is exactly the one-dimensional subspace spanned by $\Xie_{k,2k-1}$ that gives the kernel of the representation $\Phi_{k,n}$.  As we show in Theorem \ref{T:Xief} below,   $\Xie_{k,2k-1}$ equals  $\frac{(-1)^k}{k!}\ef_{k,2k-1}$
for all $k \in \ZZ_{\ge 1}$, where $\ef_{k,2k-1}$ is the essential idempotent from \eqref{eq:ef} corresponding to the set partition with all singleton blocks.    \end{remark}

\begin{examples}\label{ex:xikn}  Here we display $\Xi_{k,n}$ for $k = 1, 2, 3$.     In each case, all of the orbit diagrams  occurring in $\Xi_{k,n}$  have a coefficient containing $n-(2k-1)$ as a factor (and so are 0 when $n = 2k-1$)  except for the very last one, which is
the orbit basis diagram  $\ef_{k,2k-1}$ and  which has coefficient $(-1)^k/k!$.
\begin{align*}
\Proj_{1,n}= &  (n-1) \bigg(\frac{1}{n}
\begin{array}{c} 
\begin{tikzpicture}[xscale=.5,yscale=.5,line width=1.25pt] 
\foreach \i in {1}  { \path (\i,1.25) coordinate (T\i); \path (\i,.25) coordinate (B\i); } 
\filldraw[fill= black!12,draw=black!12,line width=4pt]  (.9,1.25) -- (1.1,1.25) -- (1.1,.25) -- (.9,.25) -- (.9,1.25);
\draw[blue] (T1) -- (B1);
\foreach \i in {1}  { \filldraw[fill=white,draw=black,line width = 1pt] (T\i) circle (4pt); \filldraw[fill=white,draw=black,line width = 1pt]  (B\i) circle (4pt); } 
\end{tikzpicture}
\end{array} 
+
 \frac{-1}{n(n-1)}
\begin{array}{c} 
\begin{tikzpicture}[xscale=.5,yscale=.5,line width=1.25pt] 
\foreach \i in {1}  { \path (\i,1.25) coordinate (T\i); \path (\i,.25) coordinate (B\i); } 
\filldraw[fill= black!12,draw=black!12,line width=4pt]  (.9,1.25) -- (1.1,1.25) -- (1.1,.25) -- (.9,.25) -- (.9,1.25);
\foreach \i in {1}  { \filldraw[fill=white,draw=black,line width = 1pt] (T\i) circle (4pt); \filldraw[fill=white,draw=black,line width = 1pt]  (B\i) circle (4pt); } 
\end{tikzpicture}
\end{array} 
\bigg)\, ,\\ 
\Proj_{2,n}= & \frac{n(n-3)}{2}   \Bigg(   \frac{1}{n(n-1)}
\left[
\begin{array}{c} 
\begin{tikzpicture}[xscale=.5,yscale=.5,line width=1.25pt] 
\foreach \i in {1,2}  { \path (\i,1.25) coordinate (T\i); \path (\i,.25) coordinate (B\i); } 
\filldraw[fill= black!12,draw=black!12,line width=4pt]  (T1) -- (T2) -- (B2) -- (B1) -- (T1);
\draw[blue] (T1) -- (B1);
\draw[blue] (T2) -- (B2);
\foreach \i in {1,2}  { \filldraw[fill=white,draw=black,line width = 1pt] (T\i) circle (4pt); \filldraw[fill=white,draw=black,line width = 1pt]  (B\i) circle (4pt); } 
\end{tikzpicture}
\end{array} 
+
\begin{array}{c} 
\begin{tikzpicture}[xscale=.5,yscale=.5,line width=1.25pt] 
\foreach \i in {1,2}  { \path (\i,1.25) coordinate (T\i); \path (\i,.25) coordinate (B\i); } 
\filldraw[fill= black!12,draw=black!12,line width=4pt]  (T1) -- (T2) -- (B2) -- (B1) -- (T1);
\draw[blue] (T1) -- (B2);
\draw[blue] (T2) -- (B1);
\foreach \i in {1,2}  { \filldraw[fill=white,draw=black,line width = 1pt] (T\i) circle (4pt); \filldraw[fill=white,draw=black,line width = 1pt]  (B\i) circle (4pt); } 
\end{tikzpicture}
\end{array}
\right] \\ & + 
 \frac{-1}{n(n-1)(n-2)}
 \bigg[
\begin{array}{c} 
\begin{tikzpicture}[xscale=.5,yscale=.5,line width=1.25pt] 
\foreach \i in {1,2}  { \path (\i,1.25) coordinate (T\i); \path (\i,.25) coordinate (B\i); } 
\filldraw[fill= black!12,draw=black!12,line width=4pt]  (T1) -- (T2) -- (B2) -- (B1) -- (T1);
\draw[blue] (T1) -- (B1);
\foreach \i in {1,2}  { \filldraw[fill=white,draw=black,line width = 1pt] (T\i) circle (4pt); \filldraw[fill=white,draw=black,line width = 1pt]  (B\i) circle (4pt); } 
\end{tikzpicture}
\end{array}  +
\begin{array}{c} 
\begin{tikzpicture}[xscale=.5,yscale=.5,line width=1.25pt] 
\foreach \i in {1,2}  { \path (\i,1.25) coordinate (T\i); \path (\i,.25) coordinate (B\i); } 
\filldraw[fill= black!12,draw=black!12,line width=4pt]  (T1) -- (T2) -- (B2) -- (B1) -- (T1);
\draw[blue] (T1) -- (B2);
\foreach \i in {1,2}  { \filldraw[fill=white,draw=black,line width = 1pt] (T\i) circle (4pt); \filldraw[fill=white,draw=black,line width = 1pt]  (B\i) circle (4pt); } 
\end{tikzpicture}
\end{array} +
\begin{array}{c} 
\begin{tikzpicture}[xscale=.5,yscale=.5,line width=1.25pt] 
\foreach \i in {1,2}  { \path (\i,1.25) coordinate (T\i); \path (\i,.25) coordinate (B\i); } 
\filldraw[fill= black!12,draw=black!12,line width=4pt]  (T1) -- (T2) -- (B2) -- (B1) -- (T1);
\draw[blue] (T2) -- (B1);
\foreach \i in {1,2}  { \filldraw[fill=white,draw=black,line width = 1pt] (T\i) circle (4pt); \filldraw[fill=white,draw=black,line width = 1pt]  (B\i) circle (4pt); } 
\end{tikzpicture}
\end{array} +
\begin{array}{c} 
\begin{tikzpicture}[xscale=.5,yscale=.5,line width=1.25pt] 
\foreach \i in {1,2}  { \path (\i,1.25) coordinate (T\i); \path (\i,.25) coordinate (B\i); } 
\filldraw[fill= black!12,draw=black!12,line width=4pt]  (T1) -- (T2) -- (B2) -- (B1) -- (T1);
\draw[blue] (T2) -- (B2);
\foreach \i in {1,2}  { \filldraw[fill=white,draw=black,line width = 1pt] (T\i) circle (4pt); \filldraw[fill=white,draw=black,line width = 1pt]  (B\i) circle (4pt); } 
\end{tikzpicture}
\end{array} 
\bigg] \\ & + 
\frac{2}{n(n-1)(n-2)(n-3)} 
\begin{array}{c} 
\begin{tikzpicture}[xscale=.5,yscale=.5,line width=1.25pt] 
\foreach \i in {1,2}  { \path (\i,1.25) coordinate (T\i); \path (\i,.25) coordinate (B\i); } 
\filldraw[fill= black!12,draw=black!12,line width=4pt]  (T1) -- (T2) -- (B2) -- (B1) -- (T1);
\foreach \i in {1,2}  { \filldraw[fill=white,draw=black,line width = 1pt] (T\i) circle (4pt); \filldraw[fill=white,draw=black,line width = 1pt]  (B\i) circle (4pt); } 
\end{tikzpicture}
\end{array}\Bigg)\, , \\ 
\Proj_{3,n} = & \frac{n(n-1)(n-5)}{6}  \Bigg(   \frac{1}{n(n-1)(n-2)}
\left[
\begin{array}{c} 
\begin{tikzpicture}[xscale=.5,yscale=.5,line width=1.25pt] 
\foreach \i in {1,2,3}  { \path (\i,1.25) coordinate (T\i); \path (\i,.25) coordinate (B\i); } 
\filldraw[fill= black!12,draw=black!12,line width=4pt]  (T1) -- (T3) -- (B3) -- (B1) -- (T1);
\draw[blue] (T1) -- (B1);
\draw[blue] (T2) -- (B2);\draw[blue] (T3) -- (B3);
\foreach \i in {1,2,3}  { \filldraw[fill=white,draw=black,line width = 1pt] (T\i) circle (4pt); \filldraw[fill=white,draw=black,line width = 1pt]  (B\i) circle (4pt); } 
\end{tikzpicture}
\end{array} 
+
\begin{array}{c} 
\begin{tikzpicture}[xscale=.5,yscale=.5,line width=1.25pt] 
\foreach \i in {1,2,3}  { \path (\i,1.25) coordinate (T\i); \path (\i,.25) coordinate (B\i); } 
\filldraw[fill= black!12,draw=black!12,line width=4pt]  (T1) -- (T3) -- (B3) -- (B1) -- (T1);
\draw[blue] (T1) -- (B2);
\draw[blue] (T2) -- (B1);\draw[blue] (T3) -- (B3);
\foreach \i in {1,2,3}  { \filldraw[fill=white,draw=black,line width = 1pt] (T\i) circle (4pt); \filldraw[fill=white,draw=black,line width = 1pt]  (B\i) circle (4pt); } 
\end{tikzpicture}
\end{array} + \left( {{\text{4 more} }\atop{\text{with $\pn=3$}}}\right)
\right]  \\
& + 
 \frac{-1}{n(n-1)(n-2)(n-3)}
 \bigg[
\begin{array}{c} 
\begin{tikzpicture}[xscale=.5,yscale=.5,line width=1.25pt] 
\foreach \i in {1,2,3}  { \path (\i,1.25) coordinate (T\i); \path (\i,.25) coordinate (B\i); } 
\filldraw[fill= black!12,draw=black!12,line width=4pt]  (T1) -- (T3) -- (B3) -- (B1) -- (T1);
\draw[blue] (T1) -- (B1);\draw[blue] (T2) -- (B2);
\foreach \i in {1,2,3}  { \filldraw[fill=white,draw=black,line width = 1pt] (T\i) circle (4pt); \filldraw[fill=white,draw=black,line width = 1pt]  (B\i) circle (4pt); } 
\end{tikzpicture}
\end{array}  + 
\begin{array}{c} 
\begin{tikzpicture}[xscale=.5,yscale=.5,line width=1.25pt] 
\foreach \i in {1,2,3}  { \path (\i,1.25) coordinate (T\i); \path (\i,.25) coordinate (B\i); } 
\filldraw[fill= black!12,draw=black!12,line width=4pt]  (T1) -- (T3) -- (B3) -- (B1) -- (T1);
\draw[blue] (T1) -- (B1);\draw[blue] (T2) -- (B3);
\foreach \i in {1,2,3}  { \filldraw[fill=white,draw=black,line width = 1pt] (T\i) circle (4pt); \filldraw[fill=white,draw=black,line width = 1pt]  (B\i) circle (4pt); } 
\end{tikzpicture}
\end{array} + (\text{16 more with $\pn= 2$})
\bigg] \\ & + 
\frac{2}{n(n-1)(n-2)(n-3)(n-4)} 
\bigg[
\begin{array}{c} 
\begin{tikzpicture}[xscale=.5,yscale=.5,line width=1.25pt] 
\foreach \i in {1,2,3}  { \path (\i,1.25) coordinate (T\i); \path (\i,.25) coordinate (B\i); } 
\filldraw[fill= black!12,draw=black!12,line width=4pt]  (T1) -- (T3) -- (B3) -- (B1) -- (T1);
\draw[blue] (T1) -- (B1);
\foreach \i in {1,2,3}  { \filldraw[fill=white,draw=black,line width = 1pt] (T\i) circle (4pt); \filldraw[fill=white,draw=black,line width = 1pt]  (B\i) circle (4pt); } 
\end{tikzpicture}
\end{array} +
\begin{array}{c} 
\begin{tikzpicture}[xscale=.5,yscale=.5,line width=1.25pt] 
\foreach \i in {1,2,3}  { \path (\i,1.25) coordinate (T\i); \path (\i,.25) coordinate (B\i); } 
\filldraw[fill= black!12,draw=black!12,line width=4pt]  (T1) -- (T3) -- (B3) -- (B1) -- (T1);
\draw[blue] (T1) -- (B2);
\foreach \i in {1,2,3}  { \filldraw[fill=white,draw=black,line width = 1pt] (T\i) circle (4pt); \filldraw[fill=white,draw=black,line width = 1pt]  (B\i) circle (4pt); } 
\end{tikzpicture}
\end{array} 
+ \left( {{\text{7 more} }\atop{\text{with $\pn=1$}}}\right)\bigg]\\ & + 
\frac{-6}{n(n-1)(n-2)(n-3)(n-4)(n-5)} 
\begin{array}{c} 
\begin{tikzpicture}[xscale=.5,yscale=.5,line width=1.25pt] 
\foreach \i in {1,2,3}  { \path (\i,1.25) coordinate (T\i); \path (\i,.25) coordinate (B\i); } 
\filldraw[fill= black!12,draw=black!12,line width=4pt]  (T1) -- (T3) -- (B3) -- (B1) -- (T1);
\foreach \i in {1,2,3}  { \filldraw[fill=white,draw=black,line width = 1pt] (T\i) circle (4pt); \filldraw[fill=white,draw=black,line width = 1pt]  (B\i) circle (4pt); } 
\end{tikzpicture}
\end{array} 
\Bigg).  
\end{align*} 
\end{examples}  

\noindent \emph{Proof of Theorem \ref{IdempotentOrbitBasis}.} As noted earlier, the element $\Xi_{k,n}$ in \eqref{eq:Up}  is nonzero and well-defined for all 
$n \ge 2k-1$.
The partition algebra $\P_k(n)$ is generated by the elements $\mathfrak{s}_i, 1 \le i \le k-1$, $\mathfrak{b}_1$, and $\mathfrak{p}_1$ 
defined in \eqref{s-gen}--\eqref{b-gen}.  We claim it suffices to show that 
\begin{align}\begin{split}\label{eq:steps}
& \hspace{-2.7cm} \text {\rm (1)}  \qquad \, \,  \mathfrak{s}_i \Xie_{k,n} =\Xie_{k,n}\mathfrak{s}_i =  \Xie_{k,n} \  \text{for} \ 1 \leq i \leq k-1;  \\
& \hspace{-2.7cm} \text {\rm (2)}  \qquad \,   \mathfrak{p}_1 \Xie_{k,n} = \Xie_{k,n} \mathfrak{p}_1 = \mathfrak{b}_1 
\Xie_{k,n}=  \Xie_{k,n} \mathfrak{b}_1 = 0;  \\
& \hspace{-2.7cm} \text {\rm (3)} \qquad \! \! \left(\Xie_{k,n}\right)^2 = \Xie_{k,n} \quad \  \text{($\Xie_{k,n}$ \  is an idempotent}).\end{split}
  \end{align} 
  
Indeed, once we establish (1) and (2), we will know that  
  $\Xie_{k,n}$  is central in $\P_k(n)$,  and  $\Xie_{k,n}$
spans a one-dimensional submodule (which is also an ideal)  of $\P_k(n)$,  so that $\left(\Xie_{k,n}\right)^2 
= c\, \Xie_{k,n}$ for some $c \in \CC$.    
If  $\chi$ is the character of that submodule, then using the fact that $\chi(ab) = \chi(ba)$ along with (1) and (2),  we 
determine that  $\chi(d_{\pi}) = 0$ unless $\pi$ is a permutation, in which case
$\chi(d_{\pi}) = 1$ by (1).  By comparing $\chi$ 
with the character of the irreducible module $\P_{k,n}^{[n-k,k]}$ in \eqref{character-value}, 
we conclude $\CC\, \Xie_{k,n} \cong  \P_{k,n}^{[n-k,k]}$  when $n \ge 2k$.  As a result, we know that  
$\Xie_{k,n}$ maps onto the $[n-k,k]$-isotypic component
 $\S_n^{[n-k,k]} \ot \P_{k,n}^{[n-k,k]}\cong \S_n^{[n-k,k]}$ (as vector spaces) when applied to $\modu^{\ot k}$.   However,  $\Psi_{k,n}(\ve_{[n-k,k]})$ is an idempotent projecting onto that space.  Since $\left(\Xie_{k,n}\right)^2 = c \,\Xie_{k,n}$ for $c \in \CC$,   
 once  we prove (3)  that $c = 1$ holds and $\Xie_{k,n}$ is an idempotent,  we will know that the two elements
$\Psi_{k,n}(\ve_{[n-k,k]})$ and $\Xie_{k,n}$  must be equal  when $n \ge 2k$,  as asserted in Theorem \ref{IdempotentOrbitBasis}. 

 We will prove that (1) and (2) of \eqref{eq:steps} hold for all $k,n$ for which $\Xie_{k,n}$ is defined. In the calculation below,
$f^{[n-k,k]}$ should be interpreted as the right-hand expression in (\ref{hook-dimension}).
 
(1) The relation  $\mathfrak{s}_i \Xie_{k,n} =\Xie_{k,n} \mathfrak{s}_i =  \Xie_{k,n}$ comes easily from \eqref{eq:symmetry}, which 
says $\sigma' x_{\pi} \sigma = x_{\sigma' \ast \pi \ast \sigma}$ for all $\sigma, \sigma' \in \S_k$, and from the fact that the coefficients satisfy $c(\sigma' \ast \pi \ast \sigma,n) = c(\pi,n)$, i.e., the coefficients depend only on the propagating number and are therefore constant up to permutation of the vertices.  

(2) The relation $\mathfrak{b}_1 \Xie_{k,n} = \Xie_{k,n} \mathfrak{b}_1 = 0$ in (2)  follows from the fact that rook partition diagrams have no horizontal edges, and $\mathfrak{b}_1$ contains the block $\{1,2,k+1,k+2\}$,  so $\mathfrak{b}_1 x_\pi = x_\pi \mathfrak{b}_1 = 0$ for any rook partition $\pi$, since these diagrams never match in the middle. 

The fact that $\mathfrak{p}_1 \Xie_{k,n} =  \Xie_{k,n} \mathfrak{p}_1 = 0$ is the most interesting case and makes use of the specific values of the coefficients $c(\pi,n)$.   To show this, it suffices to work with the element $\wtp :=
n \mathfrak{p}_1$, which enables us to ignore the scalar factor in the definition (\ref{p-gen}) of $\mathfrak{p}_1$.    First,  we prove the following property describing the product of a rook orbit diagram with $\wtp$. 

\underline{Claim}: For $\pi \in \mathcal{R}_{2k}$, let  $\widehat{x_\pi}$ be the sum of orbit diagrams obtained from $x_\pi$ by removing  (if it exists) the edge containing vertex 1 in the bottom row of $\pi$ and then 
adding an edge from vertex 1 to any other block of $\pi$.
We claim that $x_\pi \wtp = \mathsf{\ell} \, \widehat{x_\pi}$, where $\mathsf{\ell}  = n-|\pi| + 1$ if vertex 1 is its own block, and $\mathsf{\ell}  = 1$ otherwise. Here are two  examples of such products, which involve the multiplication of an orbit basis diagram with a diagram basis element,
as pictured below. 
\begin{align}\begin{split}\label{eq:1stex}
\begin{array}{c} 
\begin{tikzpicture}[xscale=.5,yscale=.5,line width=1.25pt] 
\foreach \i in {1,2,3}  { \path (\i,1.25) coordinate (T\i); \path (\i,.25) coordinate (B\i); } 
\filldraw[fill= black!12,draw=black!12,line width=4pt]  (T1) -- (T3) -- (B3) -- (B1) -- (T1);
\draw[blue] (T1) -- (B3);
\draw[blue] (T2) -- (B1);\draw[blue] (T3) -- (B2);
\foreach \i in {1,2,3}  { \filldraw[fill=white,draw=black,line width = 1pt] (T\i) circle (4pt); \filldraw[fill=white,draw=black,line width = 1pt]  (B\i) circle (4pt); } 
\end{tikzpicture} \\
\begin{tikzpicture}[xscale=.5,yscale=.5,line width=1.25pt] 
\foreach \i in {1,2,3}  { \path (\i,1.25) coordinate (T\i); \path (\i,.25) coordinate (B\i); } 
\filldraw[fill= black!12,draw=black!12,line width=4pt]  (T1) -- (T3) -- (B3) -- (B1) -- (T1);
\draw[blue] (T2) -- (B2);
\draw[blue] (T3) -- (B3);
\foreach \i in {1,2,3}  { \fill (T\i) circle (4pt); \fill  (B\i) circle (4pt); } 
\end{tikzpicture}
\end{array}  
& =
\begin{array}{c} 
\begin{tikzpicture}[xscale=.5,yscale=.5,line width=1.25pt] 
\foreach \i in {1,2,3}  { \path (\i,1.25) coordinate (T\i); \path (\i,.25) coordinate (B\i); } 
\filldraw[fill= black!12,draw=black!12,line width=4pt]  (T1) -- (T3) -- (B3) -- (B1) -- (T1);
\draw[blue] (T1) -- (B3);\draw[blue] (T3) -- (B2);
\foreach \i in {1,2,3}  { \filldraw[fill=white,draw=black,line width = 1pt] (T\i) circle (4pt); \filldraw[fill=white,draw=black,line width = 1pt]  (B\i) circle (4pt); } 
\end{tikzpicture}
\end{array}
+
\begin{array}{c} 
\begin{tikzpicture}[xscale=.5,yscale=.5,line width=1.25pt] 
\foreach \i in {1,2,3}  { \path (\i,1.25) coordinate (T\i); \path (\i,.25) coordinate (B\i); } 
\filldraw[fill= black!12,draw=black!12,line width=4pt]  (T1) -- (T3) -- (B3) -- (B1) -- (T1);
\draw[blue] (T1) -- (B3);\draw[blue] (T3) -- (B2);
\draw[blue] (T1) -- (B1);
\foreach \i in {1,2,3}  { \filldraw[fill=white,draw=black,line width = 1pt] (T\i) circle (4pt); \filldraw[fill=white,draw=black,line width = 1pt]  (B\i) circle (4pt); } 
\end{tikzpicture}
\end{array}
+
\begin{array}{c} 
\begin{tikzpicture}[xscale=.5,yscale=.5,line width=1.25pt] 
\foreach \i in {1,2,3}  { \path (\i,1.25) coordinate (T\i); \path (\i,.25) coordinate (B\i); } 
\filldraw[fill= black!12,draw=black!12,line width=4pt]  (T1) -- (T3) -- (B3) -- (B1) -- (T1);
\draw[blue] (T1) -- (B3);\draw[blue] (T3) -- (B2);
\draw[blue] (T2) -- (B1);
\foreach \i in {1,2,3}  { \filldraw[fill=white,draw=black,line width = 1pt] (T\i) circle (4pt); \filldraw[fill=white,draw=black,line width = 1pt]  (B\i) circle (4pt); } 
\end{tikzpicture}
\end{array}
+
\begin{array}{c} 
\begin{tikzpicture}[xscale=.5,yscale=.5,line width=1.25pt] 
\foreach \i in {1,2,3}  { \path (\i,1.25) coordinate (T\i); \path (\i,.25) coordinate (B\i); } 
\filldraw[fill= black!12,draw=black!12,line width=4pt]  (T1) -- (T3) -- (B3) -- (B1) -- (T1);
\draw[blue] (T1) -- (B3); \draw[blue] (T3) -- (B2);
\draw[blue] (B1) -- (B2); 
\foreach \i in {1,2,3}  { \filldraw[fill=white,draw=black,line width = 1pt] (T\i) circle (4pt); \filldraw[fill=white,draw=black,line width = 1pt]  (B\i) circle (4pt); } 
\end{tikzpicture}
\end{array}
\\ \\
\begin{array}{c} 
\begin{tikzpicture}[xscale=.5,yscale=.5,line width=1.25pt] 
\foreach \i in {1,2,3}  { \path (\i,1.25) coordinate (T\i); \path (\i,.25) coordinate (B\i); } 
\filldraw[fill= black!12,draw=black!12,line width=4pt]  (T1) -- (T3) -- (B3) -- (B1) -- (T1);
\draw[blue] (T1) -- (B3);
\draw[blue] (T3) -- (B2);
\foreach \i in {1,2,3}  { \filldraw[fill=white,draw=black,line width = 1pt] (T\i) circle (4pt); \filldraw[fill=white,draw=black,line width = 1pt]  (B\i) circle (4pt); } 
\end{tikzpicture} \\
\begin{tikzpicture}[xscale=.5,yscale=.5,line width=1.25pt] 
\foreach \i in {1,2,3}  { \path (\i,1.25) coordinate (T\i); \path (\i,.25) coordinate (B\i); } 
\filldraw[fill= black!12,draw=black!12,line width=4pt]  (T1) -- (T3) -- (B3) -- (B1) -- (T1);
\draw[blue] (T2) -- (B2);
\draw[blue] (T3) -- (B3);
\foreach \i in {1,2,3}  { \fill (T\i) circle (4pt); \fill  (B\i) circle (4pt); } 
\end{tikzpicture}
\end{array}  
& = (n-3) \left( 
\begin{array}{c} 
\begin{tikzpicture}[xscale=.5,yscale=.5,line width=1.25pt] 
\foreach \i in {1,2,3}  { \path (\i,1.25) coordinate (T\i); \path (\i,.25) coordinate (B\i); } 
\filldraw[fill= black!12,draw=black!12,line width=4pt]  (T1) -- (T3) -- (B3) -- (B1) -- (T1);
\draw[blue] (T1) -- (B3);
\draw[blue] (T3) -- (B2);
\foreach \i in {1,2,3}  { \filldraw[fill=white,draw=black,line width = 1pt] (T\i) circle (4pt); \filldraw[fill=white,draw=black,line width = 1pt]  (B\i) circle (4pt); } 
\end{tikzpicture}
\end{array}
+
\begin{array}{c} 
\begin{tikzpicture}[xscale=.5,yscale=.5,line width=1.25pt] 
\foreach \i in {1,2,3}  { \path (\i,1.25) coordinate (T\i); \path (\i,.25) coordinate (B\i); } 
\filldraw[fill= black!12,draw=black!12,line width=4pt]  (T1) -- (T3) -- (B3) -- (B1) -- (T1);
\draw[blue] (T1) -- (B3);\draw[blue] (T1) -- (B1);
\draw[blue] (T3) -- (B2);
\foreach \i in {1,2,3}  { \filldraw[fill=white,draw=black,line width = 1pt] (T\i) circle (4pt); \filldraw[fill=white,draw=black,line width = 1pt]  (B\i) circle (4pt); } 
\end{tikzpicture}
\end{array}
+
\begin{array}{c} 
\begin{tikzpicture}[xscale=.5,yscale=.5,line width=1.25pt] 
\foreach \i in {1,2,3}  { \path (\i,1.25) coordinate (T\i); \path (\i,.25) coordinate (B\i); } 
\filldraw[fill= black!12,draw=black!12,line width=4pt]  (T1) -- (T3) -- (B3) -- (B1) -- (T1);
\draw[blue] (T1) -- (B3); 
\draw[blue] (T3) -- (B2);
\draw[blue] (B1) -- (T2); 
\foreach \i in {1,2,3}  { \filldraw[fill=white,draw=black,line width = 1pt] (T\i) circle (4pt); \filldraw[fill=white,draw=black,line width = 1pt]  (B\i) circle (4pt); } 
\end{tikzpicture}
\end{array}
+
\begin{array}{c} 
\begin{tikzpicture}[xscale=.5,yscale=.5,line width=1.25pt] 
\foreach \i in {1,2,3}  { \path (\i,1.25) coordinate (T\i); \path (\i,.25) coordinate (B\i); } 
\filldraw[fill= black!12,draw=black!12,line width=4pt]  (T1) -- (T3) -- (B3) -- (B1) -- (T1);
\draw[blue] (T1) -- (B3); 
\draw[blue] (T3) -- (B2);
\draw[blue] (B1) -- (B2); 
\foreach \i in {1,2,3}  { \filldraw[fill=white,draw=black,line width = 1pt] (T\i) circle (4pt); \filldraw[fill=white,draw=black,line width = 1pt]  (B\i) circle (4pt); } 
\end{tikzpicture}
\end{array}
\right)
\end{split}
\end{align} 

We prove the claim for all $n \ge 2k$ by using labeled diagrams as in Section \ref{sec:LabeledDiagrams}. The bottom row of the product  $x_\pi \wtp$ is labeled with $r_1,r_2, \ldots, r_k \in [1,n]$. The middle row is then labeled with $z, r_2, \ldots, r_k$  as forced by the diagram $\wtp$ and \eqref{ex:LabeledDiagram}. The top row is labeled using the rule:  any vertices in the top row of $\pi$ connected to the middle row are given the corresponding middle row labels, and the remaining vertices are labeled with $q_1,q_2, \ldots, q_t \in [1,n]$ (where possibly $t=0$). These labelings can be seen in the examples \eqref{eq:2ndex} below. The labelings of $x_\pi$ must be distinct on distinct parts of $\pi$,  but the labelings on $\wtp$ need not be distinct.  
\begin{align}\begin{split}\label{eq:2ndex}
\begin{array}{c} 
\begin{tikzpicture}[xscale=.5,yscale=.5,line width=1.25pt] 
\foreach \i in {1,2,3}  { \path (\i,1.25) coordinate (T\i); \path (\i,.25) coordinate (B\i); } 
\filldraw[fill= black!12,draw=black!12,line width=4pt]  (T1) -- (T3) -- (B3) -- (B1) -- (T1);
\draw[blue] (T1) -- (B3);
\draw[blue] (T3) -- (B2);
\foreach \i in {1,2,3}  { \filldraw[fill=white,draw=black,line width = 1pt] (T\i) circle (4pt); \filldraw[fill=white,draw=black,line width = 1pt]  (B\i) circle (4pt); } 
\draw  (B1)  node[black,below=0.08cm]{${z}$};
\draw  (B2)  node[black,below=0.06cm]{${r_2}$};
\draw  (B3)  node[black,below=0.06cm]{${r_3}$};
\draw  (T1)  node[black,above=0.01cm]{${r_3}$};
\draw  (T2)  node[black,above=0.01cm]{${q_1}$};
\draw  (T3)  node[black,above=0.01cm]{${r_2}$};
\foreach \i in {1,2,3}  { \path (\i,-2.0) coordinate (BB\i); \path (\i,-1.0) coordinate (TT\i); } 
\filldraw[fill= black!12,draw=black!12,line width=4pt]  (TT1) -- (TT3) -- (BB3) -- (BB1) -- (TT1);
\draw[blue] (TT2) -- (BB2);
\draw[blue] (TT3) -- (BB3);
\foreach \i in {1,2,3}  { \fill (TT\i) circle (4pt); \fill  (BB\i) circle (4pt); }
\draw  (BB1)  node[black,below=0.03cm]{$r_1$};
\draw  (BB2)  node[black,below=0.01cm]{$r_2$};
\draw  (BB3)  node[black,below=0.03cm]{$r_3$};
\end{tikzpicture} 
\end{array}  
 &= (n-3) \left( 
\begin{array}{c} 
\begin{tikzpicture}[xscale=.5,yscale=.5,line width=1.25pt] 
\foreach \i in {1,2,3}  { \path (\i,1.25) coordinate (T\i); \path (\i,.25) coordinate (B\i); } 
\filldraw[fill= black!12,draw=black!12,line width=4pt]  (T1) -- (T3) -- (B3) -- (B1) -- (T1);
\draw[blue] (T1) -- (B3);
\draw[blue] (T3) -- (B2);
\foreach \i in {1,2,3}  { \filldraw[fill=white,draw=black,line width = 1pt] (T\i) circle (4pt); \filldraw[fill=white,draw=black,line width = 1pt]  (B\i) circle (4pt); } 
\draw  (B1)  node[black,below=0.01cm]{${r_1}$};
\draw  (B2)  node[black,below=0.01cm]{${r_2}$};
\draw  (B3)  node[black,below=0.01cm]{${r_3}$};
\draw  (T1)  node[black,above=0.01cm]{${r_3}$};
\draw  (T2)  node[black,above=0.01cm]{${q_1}$};
\draw  (T3)  node[black,above=0.01cm]{${r_2}$};
\end{tikzpicture}
\end{array}
+
\begin{array}{c} 
\begin{tikzpicture}[xscale=.5,yscale=.5,line width=1.25pt] 
\foreach \i in {1,2,3}  { \path (\i,1.25) coordinate (T\i); \path (\i,.25) coordinate (B\i); } 
\filldraw[fill= black!12,draw=black!12,line width=4pt]  (T1) -- (T3) -- (B3) -- (B1) -- (T1);
\draw[blue] (T1) -- (B3);\draw[blue] (T1) -- (B1);
\draw[blue] (T3) -- (B2);
\foreach \i in {1,2,3}  { \filldraw[fill=white,draw=black,line width = 1pt] (T\i) circle (4pt); \filldraw[fill=white,draw=black,line width = 1pt]  (B\i) circle (4pt); } 
\draw  (B1)  node[black,below=0.01cm]{${r_1}$};
\draw  (B2)  node[black,below=0.01cm]{${r_2}$};
\draw  (B3)  node[black,below=0.01cm]{${r_1}$};
\draw  (T1)  node[black,above=0.01cm]{${r_1}$};
\draw  (T2)  node[black,above=0.01cm]{${q_1}$};
\draw  (T3)  node[black,above=0.01cm]{${r_2}$};
\end{tikzpicture}
\end{array}
+
\begin{array}{c} 
\begin{tikzpicture}[xscale=.5,yscale=.5,line width=1.25pt] 
\foreach \i in {1,2,3}  { \path (\i,1.25) coordinate (T\i); \path (\i,.25) coordinate (B\i); } 
\filldraw[fill= black!12,draw=black!12,line width=4pt]  (T1) -- (T3) -- (B3) -- (B1) -- (T1);
\draw[blue] (T1) -- (B3); 
\draw[blue] (T3) -- (B2);
\draw[blue] (B1) -- (T2); 
\foreach \i in {1,2,3}  { \filldraw[fill=white,draw=black,line width = 1pt] (T\i) circle (4pt); \filldraw[fill=white,draw=black,line width = 1pt]  (B\i) circle (4pt); } 
\draw  (B1)  node[black,below=0.01cm]{${r_1}$};
\draw  (B2)  node[black,below=0.01cm]{${r_2}$};
\draw  (B3)  node[black,below=0.01cm]{${r_3}$};
\draw  (T1)  node[black,above=0.01cm]{${r_3}$};
\draw  (T2)  node[black,above=0.01cm]{${r_1}$};
\draw  (T3)  node[black,above=0.01cm]{${r_2}$};
\end{tikzpicture}
\end{array}
+
\begin{array}{c} 
\begin{tikzpicture}[xscale=.5,yscale=.5,line width=1.25pt] 
\foreach \i in {1,2,3}  { \path (\i,1.25) coordinate (T\i); \path (\i,.25) coordinate (B\i); } 
\filldraw[fill= black!12,draw=black!12,line width=4pt]  (T1) -- (T3) -- (B3) -- (B1) -- (T1);
\draw[blue] (T1) -- (B3); 
\draw[blue] (T3) -- (B2);
\draw[blue] (B1) -- (B2); 
\foreach \i in {1,2,3}  { \filldraw[fill=white,draw=black,line width = 1pt] (T\i) circle (4pt); \filldraw[fill=white,draw=black,line width = 1pt]  (B\i) circle (4pt); } 
\draw  (B1)  node[black,below=0.01cm]{${r_1}$};
\draw  (B2)  node[black,below=0.01cm]{${r_1}$};
\draw  (B3)  node[black,below=0.01cm]{${r_3}$};
\draw  (T1)  node[black,above=0.01cm]{${r_3}$};
\draw  (T2)  node[black,above=0.01cm]{${q_1}$};
\draw  (T3)  node[black,above=0.01cm]{${r_1}$};
\end{tikzpicture}
\end{array}
\right)
\\ 
\begin{array}{c} 
\begin{tikzpicture}[xscale=.5,yscale=.5,line width=1.25pt] 
\foreach \i in {1,2,3}  { \path (\i,1.25) coordinate (T\i); \path (\i,.25) coordinate (B\i); } 
\filldraw[fill= black!12,draw=black!12,line width=4pt]  (T1) -- (T3) -- (B3) -- (B1) -- (T1);
\draw[blue] (T2) -- (B1);
\draw[blue] (T1) -- (B3);
\draw[blue] (T3) -- (B2);
\foreach \i in {1,2,3}  { \filldraw[fill=white,draw=black,line width = 1pt] (T\i) circle (4pt); \filldraw[fill=white,draw=black,line width = 1pt]  (B\i) circle (4pt); } 
\draw  (B1)  node[black,below=0.08cm]{${z}$};
\draw  (B2)  node[black,below=0.06cm]{${r_2}$};
\draw  (B3)  node[black,below=0.06cm]{${r_3}$};
\draw  (T1)  node[black,above=0.01cm]{${r_3}$};
\draw  (T2)  node[black,above=0.08cm]{${z}$};
\draw  (T3)  node[black,above=0.01cm]{${r_2}$};
\foreach \i in {1,2,3}  { \path (\i,-2.0) coordinate (BB\i); \path (\i,-1.0) coordinate (TT\i); } 
\filldraw[fill= black!12,draw=black!12,line width=4pt]  (TT1) -- (TT3) -- (BB3) -- (BB1) -- (TT1);
\draw[blue] (TT2) -- (BB2);
\draw[blue] (TT3) -- (BB3);
\foreach \i in {1,2,3}  { \fill (TT\i) circle (4pt); \fill  (BB\i) circle (4pt); }
\draw  (BB1)  node[black,below=0.03cm]{$r_1$};
\draw  (BB2)  node[black,below=0.01cm]{$r_2$};
\draw  (BB3)  node[black,below=0.03cm]{$r_3$};
\end{tikzpicture} 
\end{array}  
& =
\begin{array}{c} 
\begin{tikzpicture}[xscale=.5,yscale=.5,line width=1.25pt] 
\foreach \i in {1,2,3}  { \path (\i,1.25) coordinate (T\i); \path (\i,.25) coordinate (B\i); } 
\filldraw[fill= black!12,draw=black!12,line width=4pt]  (T1) -- (T3) -- (B3) -- (B1) -- (T1);
\draw[blue] (T1) -- (B3);\draw[blue] (T3) -- (B2);
\foreach \i in {1,2,3}  { \filldraw[fill=white,draw=black,line width = 1pt] (T\i) circle (4pt); \filldraw[fill=white,draw=black,line width = 1pt]  (B\i) circle (4pt); } 
\draw  (B1)  node[black,below=0.06cm]{${r_1}$};
\draw  (B2)  node[black,below=0.06cm]{${r_2}$};
\draw  (B3)  node[black,below=0.06cm]{${r_3}$};
\draw  (T1)  node[black,above=0.01cm]{${r_3}$};
\draw  (T2)  node[black,above=0.08cm]{${z}$};
\draw  (T3)  node[black,above=0.01cm]{${r_2}$};
\end{tikzpicture}
\end{array}
+
\begin{array}{c} 
\begin{tikzpicture}[xscale=.5,yscale=.5,line width=1.25pt] 
\foreach \i in {1,2,3}  { \path (\i,1.25) coordinate (T\i); \path (\i,.25) coordinate (B\i); } 
\filldraw[fill= black!12,draw=black!12,line width=4pt]  (T1) -- (T3) -- (B3) -- (B1) -- (T1);
\draw[blue] (T1) -- (B3);\draw[blue] (T3) -- (B2);
\draw[blue] (T1) -- (B1);
\foreach \i in {1,2,3}  { \filldraw[fill=white,draw=black,line width = 1pt] (T\i) circle (4pt); \filldraw[fill=white,draw=black,line width = 1pt]  (B\i) circle (4pt); } 
\draw  (B1)  node[black,below=0.06cm]{${r_1}$};
\draw  (B2)  node[black,below=0.06cm]{${r_2}$};
\draw  (B3)  node[black,below=0.06cm]{${r_1}$};
\draw  (T1)  node[black,above=0.01cm]{${r_1}$};
\draw  (T2)  node[black,above=0.08cm]{${z}$};
\draw  (T3)  node[black,above=0.01cm]{${r_2}$};
\end{tikzpicture}
\end{array}
+
\begin{array}{c} 
\begin{tikzpicture}[xscale=.5,yscale=.5,line width=1.25pt] 
\foreach \i in {1,2,3}  { \path (\i,1.25) coordinate (T\i); \path (\i,.25) coordinate (B\i); } 
\filldraw[fill= black!12,draw=black!12,line width=4pt]  (T1) -- (T3) -- (B3) -- (B1) -- (T1);
\draw[blue] (T1) -- (B3);\draw[blue] (T3) -- (B2);
\draw[blue] (T2) -- (B1);
\foreach \i in {1,2,3}  { \filldraw[fill=white,draw=black,line width = 1pt] (T\i) circle (4pt); \filldraw[fill=white,draw=black,line width = 1pt]  (B\i) circle (4pt); } 
\draw  (B1)  node[black,below=0.06cm]{${r_1}$};
\draw  (B2)  node[black,below=0.06cm]{${r_2}$};
\draw  (B3)  node[black,below=0.06cm]{${r_3}$};
\draw  (T1)  node[black,above=0.01cm]{${r_3}$};
\draw  (T2)  node[black,above=0.01cm]{${r_1}$};
\draw  (T3)  node[black,above=0.01cm]{${r_2}$};
\end{tikzpicture}
\end{array}
+
\begin{array}{c} 
\begin{tikzpicture}[xscale=.5,yscale=.5,line width=1.25pt] 
\foreach \i in {1,2,3}  { \path (\i,1.25) coordinate (T\i); \path (\i,.25) coordinate (B\i); } 
\filldraw[fill= black!12,draw=black!12,line width=4pt]  (T1) -- (T3) -- (B3) -- (B1) -- (T1);
\draw[blue] (T1) -- (B3); \draw[blue] (T3) -- (B2);
\draw[blue] (B1) -- (B2); 
\foreach \i in {1,2,3}  { \filldraw[fill=white,draw=black,line width = 1pt] (T\i) circle (4pt); \filldraw[fill=white,draw=black,line width = 1pt]  (B\i) circle (4pt); } 
\draw  (B1)  node[black,below=0.06cm]{${r_1}$};
\draw  (B2)  node[black,below=0.06cm]{${r_1}$};
\draw  (B3)  node[black,below=0.06cm]{${r_3}$};
\draw  (T1)  node[black,above=0.01cm]{${r_3}$};
\draw  (T2)  node[black,above=0.08cm]{${z}$};
\draw  (T3)  node[black,above=0.01cm]{${r_1}$};
\end{tikzpicture}
\end{array}\, .\end{split}
\end{align}
 
The constant $n-|\pi| +1$ counts the possible values of $z\in [1,n]$, which must be distinct from the other $|\pi| - 1$ labels of $x_\pi$. On the right-hand side, the connections from vertex 1 to another block account  
for the fact that  $r_1$  is allowed to equal any of the other labels, since vertex 1 is not connected in $\wtp$. 
 
To complete the proof of (2) in the theorem, we need to show that each of the orbit basis diagrams
 $x_\pi$  in $\Xie_{k,n}\wtp$  and in $\wtp\Xie_{k,n}$ has coefficient 0.   For the 
product $\Xie_{k,n}\wtp$,  each such  $x_\pi$ is a rook orbit diagram or can be obtained from a rook orbit diagram by connecting vertex 1 to another block in the diagram, as seen above.  Assume $\pn(\pi) = t$ and $|\pi| = r$, so that then $x_\pi$ has  $k - t$ isolated vertices in the top row. The first way we can obtain such a diagram $x_\pi$  is in the product $x_\vr \wtp$ where $\vr$ is the rook partition obtained from $\pi$ by removing the connection to vertex 1. These are the kinds of diagrams that appear in the first line of 
\eqref{eq:1stex}. In this case,  $x_\vr$ has coefficient $f^{[n-k,k]} (-1)^{k-t}(k-t)!/(n)_{t+1}$ in $\Xie
{\!}_{k,n}$, and in the product $x_\vr \wtp$,  the diagram $x_\pi$ has  a factor of  $n-t$ by the previous claim.  Thus, the coefficient of $x_\pi$ coming from $x_\vr \wtp$ is
 
$$
f^{[n-k,k]}\frac{(n-t)(-1)^{k-t}(k-t)!}{(n)_{t+1}} =f^{[n-k,k]} \frac{(-1)^{k-t}(k-t)!}{(n)_{t}}.
$$
The other way to obtain $x_\pi$ in $\Xie_{k,n} \wtp$ is in a product $x_\delta \wtp$, where $\pi$ is obtained from  a rook partition $\delta$ by removing the connection of vertex 1 in $\delta$ and connecting it to the block that contains vertex 1 in $\pi$.  These are  the kinds of diagrams that appear in the second line of \eqref{eq:1stex}. There are $k-t$ such diagrams $\delta$ that  give $\pi$ in this way (one for each of the isolated vertices in the top row of $\pi$).  In this case,  the multiplication factor from $x_\delta \wtp$ is 1. Thus,  there  are $k-t$ terms $x_\delta \wtp$ in $\Xie_{k,n}$  that have $x_\pi$  with 
coefficient equal to
$$
f^{[n-k,k]} \frac{(-1)^{k-t-1}(k-t-1)!}{(n)_{t}}
$$  
The sum of all the coefficients of $x_\pi$ is 0, as desired.  Checking that the coefficient of the diagrams for which vertex 1 is isolated is entirely similar.
Verifying that the product $\wtp \Xie_{k,n}= 0$ can be done by reflecting the diagrams in the argument above about their horizontal axes.   Consequently,  (2) holds. 

The proof of Theorem  \ref{IdempotentOrbitBasis} will be finished once we prove (3) of \eqref{eq:steps},  $\left(\Xie_{k,n}\right)^2 = \Xie_{k,n}$, which will be shown in part (iii) of the next proposition. 

\begin{prop}\label{IdempotentDiagramBasis} If 
$\Xie_{k,n}   = \sum_{\varrho \in \Pi_{2k}} a(\vr) d_\varrho$ is the expression for  $\Xie_{k,n}$ as a linear combination of the diagram basis elements, 
then
\begin{enumerate}
\item[{\rm (i)}]  
\begin{equation}\label{ConversionToDiagramBasis}
a(\vr) = f^{[n-k,k]}
 \sum_{r = 0}^k\frac{ (-1)^{k - r} (k-r)!}{(n)_{2k-r}}\left( 
 \sum_{\pi \in \mathcal{R}_{2k,r},  \ \pi \preceq \vr }\, \mu_{2k}(\pi,\vr)\right),
\end{equation}
where $\mathcal{R}_{2k,r}\subseteq \mathcal{R}_{2k}$ is the subset of rook partitions $\pi$  with propagating number $\pn(\pi)=r$ for $0 \le r \le k$, 
and $\mu_{2k}(\pi,\vr)$ is the M\"obius function from \eqref{eq:mobiusa}. 
\item[{\rm (ii)}] $a(\vr) =(k\,!)^{-1}$ if $\varrho$ is a permutation.
\item[{\rm (iii)}]  $\left(\Xie_{k,n}\right)^2  = \Xie_{k,n}$ so that {\rm (3)} of \eqref{eq:steps} holds, and the proof of Theorem 
\ref{IdempotentOrbitBasis}  is complete.
\end{enumerate}
\end{prop}   

\begin{proof}
(i) \  Since the coefficient $c(\pi,n)$  of $x_\pi$ in $\Xi_{k,n}$ in  \eqref{eq:Up}  depends only on $n, k$ and the propagating number $\pn(\pi)$, we can sum over  rook partitions $\pi$ with propagating number $r$ for  $0 \le r \le k$.   Then we can apply 
M\"obius inversion \eqref{eq:mobiusa} to express $x_\pi$ in terms of the diagrams $d_\vr$ to obtain
\begin{align}\begin{split}\label{collect-terms}
 \Xi_{k,n} & = f^{[n-k,k]} \sum_{r = 0}^k \sum_{\pi \in \mathcal{R}_{2k,r}} \frac{ (-1)^{k - r} (k-r)!}{(n)_{2k-r}} x_\pi  
 = f^{[n-k,k]} \sum_{r = 0}^k\frac{ (-1)^{k - r} (k-r)!}{(n)_{2k-r}}  \sum_{\pi \in \mathcal{R}_{2k,r}}  x_\pi \\
& = f^{[n-k,k]} \sum_{r = 0}^k\frac{ (-1)^{k - r} (k-r)!}{(n)_{2k-r}}\left( \sum_{\pi \in \mathcal{R}_{2k,r},\, \vr \in \Pi_{2k},\, \ \pi \preceq \varrho} \mu_{2k}(\pi,\varrho) d_\varrho\right),
\end{split}
\end{align}
implying  \eqref{ConversionToDiagramBasis}.

(ii) \  Let $\vr$ be a permutation. There are $\binom{k}{r}$ rook partitions $\pi$ in $\Pi_{2k}$ with propagating number $r$ and $\pi \preceq \vr$. For each such $\pi$, we have $\mu_{2k}(\pi,\varrho) = (-1)^{k-r}$, since in going from $\pi$ to $\vr$ we join $k-r$ pairs of vertices by propagating edges. Thus, the coefficient of $d_\vr$ is
\begin{align*}
a(\vr) & = f^{[n-k,k]} \sum_{r = 0}^k\frac{ (-1)^{k - r} (k-r)!}{(n)_{2k-r}}  \binom{k}{r} (-1)^{k-r}  
= f^{[n-k,k]} \sum_{r = 0}^k\frac{  (k-r)! k!}{(n)_{2k-r} r! (k-r)!}   \\
& = \frac{f^{[n-k,k]}k!}{n!} \sum_{r = 0}^k\frac{  (n-2k+r)! }{ r! }    
= \frac{f^{[n-k,k]}k!}{n!} \frac{  (n-k + 1)! }{ (n-2k+1) k! }  
= \frac{1}{k!},
\end{align*}
where the fourth equality comes from applying the diagonal sum identity below for binomial coefficients in Pascal's triangle  with $\ell$ set equal  $n-2k$:  
$$\sum_{r=0}^k \frac{(\ell+r)!}{r!} = 
\ell!\sum_{r=0}^k \binom{\ell+r}{r} =
\ell!\binom{\ell+k+1}{k} =
\frac{(\ell+k+1)!}{(\ell+1) k!}.$$ 

(iii) \ We know that $\left(\Xie_{k,n}\right)^2 = c\,  \Xie_{k,n}$ for some constant $c$, and to determine $c$ we compare the coefficient of the identity diagram
$\mathsf{I}_k$ in $\P_k(n)$ (see \eqref{id-def}) on both sides of $\left(\Xie_{k,n}\right)^2 = c\, \Xie_{k,n}$. By Proposition \ref{IdempotentDiagramBasis}\,(ii), the coefficient of $\mathsf{I}_k$ in $\Xie
{\!}_{k,n}$ is $(k\,!)^{-1}$.     Since
$\left(\Xie_{k,n}\right)^2 =  \sum_{\varrho, \pi  \in \Pi_{2k}} a(\vr)a(\pi) d_\vr d_\pi$,  the only way to get $\mathsf{I}_k$ in a product $d_\vr d_\pi$ is if $\vr$ is a permutation and $\pi = \varrho^{-1}$. Thus, again using  Proposition \ref{IdempotentDiagramBasis}\,(ii), we determine that the coefficient of  $\mathsf{I}_k$ in $\left(\Xie_{k,n}\right)^2$ is  $\sum_{\pi \in \S_k} (k\,!)^{-2} = (k\,!)^{-1}.$  Therefore,  $c = 1$.    
\end{proof}    

\begin{remark}
In \cite{MW},  Martin and Woodcock give a recursive construction for what they term a ``splitting idempotent''  $\varphi_k^k$ of $\P_k(n)$ when  $n \ge 2k$. 
The idempotent $\varphi_k^k$ is characterized by the property that $\varphi_k^k y = 0$ for all $y$ in the ideal $J$ spanned by the diagram basis
elements with propagating number less than $k$, and $\varphi_k^k \sigma \equiv \sigma$  modulo $J$ for $\sigma \in \S_k$,  where we are identifying
a permutation $\sigma$ with its
corresponding permutation diagram basis element in $\P_k(n)$ as we have done, for example,  in \eqref{eq:symmetry}.     Letting
$e_\mu = \frac{f^\mu}{k!} \sum_{\sigma \in \S_k} \chi^\mu_{\S_k}(\sigma^{-1})\sigma$ for $\mu$ a partition of $k$, they show that $\varphi_k^k e_\mu$ is the primitive
central idempotent in $\P_k(n)$ projecting onto the irreducible $\P_k(n)$-module labeled by the partition of $n$ obtained by adjoining a part of size $n-k$ to $\mu$.
When $\mu$ is the one-part partition $[k]$ of $k$,  then $e_{[k]} = \frac{1}{k!} \sum_{\sigma \in \S_k} \sigma$ is the idempotent corresponding to the trivial
one-dimensional $\S_k$-module.   In this case,  $\varphi_k^k e_{[k]}$ is the primitive central idempotent projecting onto the irreducible 
module $\P_{k,n}^{[n-k,k]}$.      
Since $\Xi_{k,n}$ has that same property,  it follows that $\Xi_{k,n} =  \varphi_k^k e_{[k]}$ when $n \ge 2k$.   In \eqref{collect-terms}, we have displayed an explicit expression for 
$\varphi_k^k e_{[k]} = \Xi_{k,n}$ in both the diagram and orbit bases of $\P_k(n)$. 
 \end{remark}

\subsection{The irreducible module labeled by the partition $[n-1-k,k]$} \label{S:half}
Next,  we extend the results of the previous section to the partition algebras $\P_{k+\half}(n)$.   

Suppose $n \in \ZZ_{>1}$ and $k \in \ZZ_{\ge 1}$,  and define   
\begin{equation}
\begin{aligned} \label{eq:Xihalffirst}
\Xi_{\half,n} &:= \vertedge,   \\
\Xi_{k+\half,n} &:=\Xi_{k,n-1}\,\vertedge
=\frac{1}{k!} (n-2k)(n-1)_{(k-1)} \sum_{\pi \in \mathcal{R}_{2k}}  \frac{ (-1)^{k - \pn(\pi)}
(k-\pn(\pi))!} {(n-1)_{2k-\pn(\pi)}}  x_\pi\,\vertedge, 
\end{aligned}
\end{equation}   
where by the second expression,  we mean append an orbit vertical edge $\vertedge$ on the right of each orbit diagram in $\Xi_{k,n-1}$ so that  the adjoined edge is  not connected to any of the first $k$ vertices on the top or bottom, as in the examples below.  (We are not using a $\ot$ symbol here as we did in Section \ref{S:funthms},  as only one vertical edge is adjoined.)

\begin{examples}
 \begin{align*}
\Proj_{1\half,n}=&  (n-2) \bigg(   \frac{1}{n-1}
\begin{array}{c} 
\begin{tikzpicture}[xscale=.5,yscale=.5,line width=1.25pt] 
\foreach \i in {1,2}  { \path (\i,1.25) coordinate (T\i); \path (\i,.25) coordinate (B\i); } 
\filldraw[fill= black!12,draw=black!12,line width=4pt]  (T1) -- (T2) -- (B2) -- (B1) -- (T1);
\draw[blue] (T1) -- (B1);
\draw[blue] (T2) -- (B2);
\foreach \i in {1,2}  { \filldraw[fill=white,draw=black,line width = 1pt] (T\i) circle (4pt); \filldraw[fill=white,draw=black,line width = 1pt]  (B\i) circle (4pt); } 
\end{tikzpicture}
\end{array} 
+ 
 \frac{-1}{(n-1)(n-2)}
\begin{array}{c} 
\begin{tikzpicture}[xscale=.5,yscale=.5,line width=1.25pt] 
\foreach \i in {1,2}  { \path (\i,1.25) coordinate (T\i); \path (\i,.25) coordinate (B\i); } 
\filldraw[fill= black!12,draw=black!12,line width=4pt]  (T1) -- (T2) -- (B2) -- (B1) -- (T1);
\draw[blue] (T2) -- (B2);
\foreach \i in {1,2}  { \filldraw[fill=white,draw=black,line width = 1pt] (T\i) circle (4pt); \filldraw[fill=white,draw=black,line width = 1pt]  (B\i) circle (4pt); } 
\end{tikzpicture}
\end{array} 
\bigg)\, , \\ 
\Proj_{2\half,n} = & \frac{(n-1)(n-4)}{2}  \Bigg(   \frac{1}{(n-1)(n-2)}
\left[
\begin{array}{c} 
\begin{tikzpicture}[xscale=.5,yscale=.5,line width=1.25pt] 
\foreach \i in {1,2,3}  { \path (\i,1.25) coordinate (T\i); \path (\i,.25) coordinate (B\i); } 
\filldraw[fill= black!12,draw=black!12,line width=4pt]  (T1) -- (T3) -- (B3) -- (B1) -- (T1);
\draw[blue] (T1) -- (B1);
\draw[blue] (T2) -- (B2);\draw[blue] (T3) -- (B3);
\foreach \i in {1,2,3}  { \filldraw[fill=white,draw=black,line width = 1pt] (T\i) circle (4pt); \filldraw[fill=white,draw=black,line width = 1pt]  (B\i) circle (4pt); } 
\end{tikzpicture}
\end{array} 
+
\begin{array}{c} 
\begin{tikzpicture}[xscale=.5,yscale=.5,line width=1.25pt] 
\foreach \i in {1,2,3}  { \path (\i,1.25) coordinate (T\i); \path (\i,.25) coordinate (B\i); } 
\filldraw[fill= black!12,draw=black!12,line width=4pt]  (T1) -- (T3) -- (B3) -- (B1) -- (T1);
\draw[blue] (T1) -- (B2);
\draw[blue] (T2) -- (B1);\draw[blue] (T3) -- (B3);
\foreach \i in {1,2,3}  { \filldraw[fill=white,draw=black,line width = 1pt] (T\i) circle (4pt); \filldraw[fill=white,draw=black,line width = 1pt]  (B\i) circle (4pt); } 
\end{tikzpicture}
\end{array}
\right]  \\
& \qquad \quad  + 
 \frac{-1}{(n-1)(n-2)(n-3)}
 \bigg[
\begin{array}{c} 
\begin{tikzpicture}[xscale=.5,yscale=.5,line width=1.25pt] 
\foreach \i in {1,2,3}  { \path (\i,1.25) coordinate (T\i); \path (\i,.25) coordinate (B\i); } 
\filldraw[fill= black!12,draw=black!12,line width=4pt]  (T1) -- (T3) -- (B3) -- (B1) -- (T1);
\draw[blue] (T1) -- (B1);\draw[blue] (T3) -- (B3);
\foreach \i in {1,2,3}  { \filldraw[fill=white,draw=black,line width = 1pt] (T\i) circle (4pt); \filldraw[fill=white,draw=black,line width = 1pt]  (B\i) circle (4pt); } 
\end{tikzpicture}
\end{array}  +
\begin{array}{c} 
\begin{tikzpicture}[xscale=.5,yscale=.5,line width=1.25pt] 
\foreach \i in {1,2,3}  { \path (\i,1.25) coordinate (T\i); \path (\i,.25) coordinate (B\i); } 
\filldraw[fill= black!12,draw=black!12,line width=4pt]  (T1) -- (T3) -- (B3) -- (B1) -- (T1);
\draw[blue] (T1) -- (B2);\draw[blue] (T3) -- (B3);
\foreach \i in {1,2,3}  { \filldraw[fill=white,draw=black,line width = 1pt] (T\i) circle (4pt); \filldraw[fill=white,draw=black,line width = 1pt]  (B\i) circle (4pt); } 
\end{tikzpicture}
\end{array} +
\begin{array}{c} 
\begin{tikzpicture}[xscale=.5,yscale=.5,line width=1.25pt] 
\foreach \i in {1,2,3}  { \path (\i,1.25) coordinate (T\i); \path (\i,.25) coordinate (B\i); } 
\filldraw[fill= black!12,draw=black!12,line width=4pt]  (T1) -- (T3) -- (B3) -- (B1) -- (T1);
\draw[blue] (T2) -- (B1);\draw[blue] (T3) -- (B3);
\foreach \i in {1,2,3}  { \filldraw[fill=white,draw=black,line width = 1pt] (T\i) circle (4pt); \filldraw[fill=white,draw=black,line width = 1pt]  (B\i) circle (4pt); } 
\end{tikzpicture}
\end{array} +
\begin{array}{c} 
\begin{tikzpicture}[xscale=.5,yscale=.5,line width=1.25pt] 
\foreach \i in {1,2,3}  { \path (\i,1.25) coordinate (T\i); \path (\i,.25) coordinate (B\i); } 
\filldraw[fill= black!12,draw=black!12,line width=4pt]  (T1) -- (T3) -- (B3) -- (B1) -- (T1);
\draw[blue] (T2) -- (B2);\draw[blue] (T3) -- (B3);
\foreach \i in {1,2,3}  { \filldraw[fill=white,draw=black,line width = 1pt] (T\i) circle (4pt); \filldraw[fill=white,draw=black,line width = 1pt]  (B\i) circle (4pt); } 
\end{tikzpicture}
\end{array} 
\bigg] \\ & \qquad  \quad + 
\frac{2}{(n-1)(n-2)(n-3)(n-4)} 
\begin{array}{c} 
\begin{tikzpicture}[xscale=.5,yscale=.5,line width=1.25pt] 
\foreach \i in {1,2,3}  { \path (\i,1.25) coordinate (T\i); \path (\i,.25) coordinate (B\i); } 
\filldraw[fill= black!12,draw=black!12,line width=4pt]  (T1) -- (T3) -- (B3) -- (B1) -- (T1);
\draw[blue] (T3) -- (B3);
\foreach \i in {1,2,3}  { \filldraw[fill=white,draw=black,line width = 1pt] (T\i) circle (4pt); \filldraw[fill=white,draw=black,line width = 1pt]  (B\i) circle (4pt); } 
\end{tikzpicture}
\end{array} 
\Bigg)\, .   
\end{align*}  
\end{examples}

Suppose $k \in \ZZ_{\ge 0}$, $n \in \ZZ_{>1}$, and $n \ge 2k+1$.   For  the partition $[n-1-k,k]$ of $n-1$ consider the corresponding primitive central idempotent $\ve_{[n-1-k,k]}$ in $\CC \S_{n-1}$ given by
\begin{equation}
\ve_{[n-1-k,k]} = \frac{f^{[n-1-k,k]}}{(n-1)!} \sum_{\sigma \in \S_{n-1}}\chi^{[n-1-k,k]}_{\S_{n-1}}\!(\sigma^{-1})\, \sigma.
\end{equation}
When $k = 0$,  this is the idempotent corresponding to the trivial representation of $\S_{n-1}$, and the image of this idempotent  under the representation  $\Psi_{\half,n}: \CC \S_{n-1}  \rightarrow  \End(\CC \mathsf{u_n})$ is the identity mapping.   Since 
$\Xi_{\half,n} = \vertedge \in \P_\half(n)$ acts as the identity map on $\CC \mathsf{u}_n$,  the two are equal,   $\Psi_{\half,n}(\ve^{[n-1,0]}) = \Xi_{\half,n}$. 

When $k > 0$, we show that the image of the idempotent $\ve_{[n-1-k,k]}$ under the representation, 
\begin{equation}\label{eq:Psihalf} \Psi_{k+\half,n}: \CC\S_{n-1}  \rightarrow  \End(\modu^{\ot k} \ot \mathsf{u}_n), \qquad 
\Psi_{k+\half,n} = \Psi_{k,n} \ot  \mathsf{Id}_{\M_n}  \end{equation}  
is $\Xi_{k+\half,n}$,  where $\Psi_{k,n}: \CC \S_n \rightarrow \modu^{\ot k}$ is restricted to $\CC \S_{n-1}$.    More precisely, we prove  the following:

\begin{thm}\label{T:IdOB}  Assume $k,n \in \ZZ_{\ge 1}$ and $n \ge 2k+1$.   Then 
\begin{equation}\label{eq:idemhalf}
\Xi_{k+\half,n} =  \Psi_{k+\half,n}\left(\ve_{[n-1-k,k]}\right) = \frac{f^{[n-1-k,k]}}{(n-1)!} \sum_{\sigma \in \S_{n-1}}\chi^{[n-1-k,k]}_{\S_{n-1}}\!(\sigma^{-1})\, \Psi_{k+\half,n}(\sigma), 
\end{equation}
where $\Psi_{k+\half,n}$ is as in \eqref{eq:Psihalf}. \end{thm}

\begin{proof} For each $i=1,2,\dots,k$,  let $\mathsf{U}_i$ be the subspace of $\modu^{\ot k} \ot \mathsf{u}_n$ spanned by the simple basis tensors  
$\us_{r_1} \ot \us_{r_2} \ot \cdots \ot \us_{r_k} \ot \us_n$ for which $\us_{r_i} = \us_n$.  The operator $\Xi_{k+\half,n}$ acts as 0 on any simple basis tensor in $\mathsf{U}_i$,  because the orbit diagrams in the expansion of $\Xi_{k+\half,n}$ in the orbit basis are rook orbit diagrams, so none of them has an edge connecting the $i$th column to the $(k+1)$st column.    The remaining basis tensors not in some $\mathsf{U}_i$ span the subspace $\M_{n-1}^{\ot k} \ot \us_n$, and  by definition, $\Xi_{k+\half,n}$ acts exactly as the operator $\Xi_{k,n-1}\vertedge$ on that space. 
   
Each space $\mathsf{U}_i$  is an $\S_{n-1}$-submodule isomorphic as an $\S_{n-1}$-module  to 
$\M_n^{\ot k-1}$.   The $\S_{n-1}$-module labeled by  the partition $[n-1-k,k]$ does not appear as a summand in
$\mathsf{U}_i$, because the partition shape $[n-1-k,k]$ first appears 
after $k$ tensor powers, i.e., in $\M_n^{\ot k} \ot \us_n$. Thus,  $\Psi_{k+\half,n}\left(\ve_{[n-1-k,k]}\right)$ acts as 0 on $\mathsf{U}_i$ for $i=1,2,\dots,k$  
and has a nonzero action given by $\Psi_{k,n-1} \ot \mathsf{Id}_{\M_n}$ on the submodule $\M_{n-1}^{\ot k} \ot \us_n$ of $\M_{n}^{\ot k} \ot \us_n$,
where $\Psi_{k,n-1}: \CC [\S_{n-1}] \rightarrow \End(\M_{n-1}^{\ot k} \ot \mathsf{u}_n)$ is the representation afforded by  $\M_{n-1}^{\ot k} \ot \us_n$.
Since $\Xi_{k,n-1} = \Psi_{k,n-1}\left(\ve_{[n-1-k,k]}\right)$ on $\M_{n-1}^{\ot k}$ for $n \geq 2k+1$ by Theorem  \ref{IdempotentOrbitBasis},  we have that  the actions of $\Xi_{k+\half,n}$ and
$\Psi_{k+\half,n}\left(\ve_{[n-1-k,k]}\right)$ on $\M_{n}^{\ot k} \ot \us_n$ are identical.   But $\P_{k+\half}(n)$ is represented faithfully
on that space when $n \geq 2k+1$ by Theorem \ref{T:Phi}\,(b),  so the two must be equal.    \end{proof}

  \subsection{Summary of our results}
We have proven the multiplication rule for orbit basis elements of the partition algebra $\P_k(n)$.
For $n,k \in \ZZ_{\ge 1}$  such that $2k > n$, we have exhibited an orbit basis element $\ef_{k,n}$
in \eqref{eq:ef} and have shown that the two-sided ideal of $\P_k(n)$ generated by $\ef_{k,n}$ is the
kernel of the representation $\Phi_{k,n}:  \P_k(n) \rightarrow \End_{\S_n}(\mathsf{M}_n^{\ot k})$ (Theorem \ref{thm:generator})
and  that $\ef_{k,n}$ is an essential idempotent (Theorem  \ref{T:secfund}).
From  those results,  we also know that 
$\ef_{k-\half,n} = \ef_{k,n}$ generates the kernel of  the representation $\Phi_{k-\half,n}: \P_{k-\half}(n)  \rightarrow \End_{\S_{n-1}}(\mathsf{M}_n^{\ot k-1})$.    

In  Section \ref{S:funthms}, in particular in Theorem \ref{T:2ndfund},  we have established the Second Fundamental Theorem of Invariant Theory
for the symmetric group $\S_n$ using the partition algebra results proven in this paper.  

According to Theorem \ref{IdempotentOrbitBasis},  for $n \ge 2k$  the image
$\Xi_{k,n} = \Psi_{k,n}\left(\ve_{[n-k,k]}\right)$ of  the central idempotent  $\ve_{[n-k,k]} \in\CC \S_n$ corresponding
to the two-part partition $[n-k,k]$ of $n$  
under the representation $\Psi_{k,n}: \CC \S_n \rightarrow \End(\mathsf{M}_n^{\ot k})$  lives in $\P_k(n) \cong
\End_{\S_n}(\M_n^{\ot k})$ and has the following  expression:
 \begin{align}\begin{split}\label{eq:idems} \Xie_{k,n} &=f^{[n-k,k]} \sum_{\pi \in \mathcal{R}_{2k}} c(\pi,n) x_\pi=\frac{1}{k!}(n-2k+1)(n)_{k-1} \sum_{\pi \in \mathcal{R}_{2k}}  \frac{ (-1)^{k - \pn(\pi)} (k-\pn(\pi))!}{(n)_{2k-\pn(\pi)}}  x_\pi \\
 &= \frac{1}{k!} (n-2k+1)(n)_{k-1} \sum_{\varrho \in \Pi_{2k}} \Bigg( \sum_{\pi \in \mathcal{R}_{2k}, \pi \preceq \varrho}  
c(\pi,n) \mu_{2k}(\pi,\varrho)\Bigg) d_\varrho. \end{split} \end{align}  

 Analogously,  in Theorem  \ref{T:IdOB},  we have established that 
under the representation
\begin{equation}\label{eq:Psihalf} \Psi_{k+\half,n}: \CC \S_{n-1} \rightarrow  \End(\modu^{\ot k} \ot \mathsf{u}_n),
 \qquad 
\Psi_{k+\half,n} = \Psi_{k,n} \ot  \mathsf{Id}_{\M_n}, \end{equation}
 the image of the central idempotent $\ve_{[n-1-k,k]}  \in \CC \S_{n-1}$ corresponding to
 the partition $[n-1-k,k]$ of $n-1$ has the following expression:  
 \begin{equation} \Psi_{k+\half,n}\left(\ve_{[n-1-k,k]}\right) = \frac{f^{[n-1-k,k]}}{(n-1)!} \sum_{\sigma \in \S_{n-1}}\chi^{[n-1-k,k]}_{\S_{n-1}}\!(\sigma^{-1})\, \Psi_{k+\half,n}(\sigma).
\end{equation} 
Moreover, we know that  $\Psi_{k+\half,n}\left(\ve_{[n-1-k,k]}\right)  = \Xie_{k,n}\,\vertedge$, which
we have denoted $\Xie_{k+\half,n}$.      Therefore,
\begin{align}\begin{split}\label{eq:Xihalf}  
 \Xie_{k+\half,n} &= \Xie_{k,n-1} \vertedge
= \frac{1}{k!} (n-2k)(n-1)_{k-1} \sum_{\pi \in \mathcal{R}_{2k}} c(\pi,n)x_\pi\vertedge \\
&=\frac{1}{k!} (n-2k)(n-1)_{k-1} \sum_{\pi \in \mathcal{R}_{2k}}  \frac{ (-1)^{k - \pn(\pi)}
(k-\pn(\pi))!} {(n-1)_{2k-\pn(\pi)}}  x_\pi\,\vertedge, \\
& = \frac{1}{k!} (n-2k)(n-1)_{k-1} \sum_{\varrho \in \Pi_{2k+1}} \Bigg( 
\sum_{\pi \in \mathcal{R}_{2k+1}, \pi \preceq \varrho}  
c(\pi,n) \mu_{2k}(\pi,\varrho)\Bigg) d_\varrho,
\end{split}\end{align}
where $\mathcal{R}_{2k+1} \subseteq \mathcal{R}_{2(k+1)}$ is the set of rook partitions that contain the block $\{k+1,2k+2\}$. 

The coefficient $f^{[n-k,k]}c(\pi,n)$ of a rook orbit basis element $x_\pi$ in $\Xie_{k,n}$ equals
\begin{equation} \frac{(n-2 k + 1) n!}{k! (n-k+1)!} \, \ \frac{ (-1)^{k - \pn(\pi)} (k-\pn(\pi))!}{(n)_{2k-\pn(\pi)}}.
\end{equation}  
We observed in Remark \ref{R:Xi} that this coefficient makes sense for $n=2k-1$, 
and for that particular value of $n$, this expression is 0 if the propagating number  $\pn(\pi) > 0$,  and when $\pn(\pi) = 0$,
it equals  $(-1)^k/ k!$.   There is a unique $\pi \in \mathcal{R}_{2k}$ with $\pn(\pi) = 0$, and  the associated rook orbit basis element $x_{\pi}$
has $2k$ singleton blocks.  Thus, it equals the essential idempotent 
$\ef_{k,2k-1}$ from \eqref{eq:ef}.  (Compare Theorem \ref{T:secfund} with $n=2k-1$, which shows that 
$(\ef_{k,2k-1})^2 = (-1)^k k!\, \ef_{k,2k-1}$.)    Consequently, we know that the following relations hold:

\newpage
\begin{thm}\label{T:Xief}  For $k \in \ZZ_{\ge 1}$,  
\begin{itemize}\item[{\rm (a)}] \  $\Xie_{k,2k-1}   =  \displaystyle{\frac{(-1)^k}{k!}  \ef_{k,2k-1}}$. 
\item[{\rm (b)}]  \ $\Xie_{k+\half,2k}  = \Xie_{k,2k-1}\, \vertedge  = \displaystyle{  \frac{(-1)^k}{k!}  \ef_{k,2k-1}\vertedge =\frac{(-1)^k}{k!}\ef_{k+1,2k} = \frac{(-1)^k}{k!}\ef_{k+\half,2k}}$.
\end{itemize}  \end{thm}

\begin{remark}\label{R:ek2k-1} For all but finitely many values of $\xi$,
\begin{equation}\label{eq:cpixi}
\Xie_{k,\xi} :=  \frac{1}{k!}(\xi-2 k + 1)  (\xi)_{k-1}  \sum_{\pi \in \mathcal{R}_{2k}} c(\pi,\xi) x_\pi, \ \text{where} \  
c(\pi,\xi)  =\frac{ (-1)^{k - \pn(\pi)} (k-\pn(\pi))!}{(\xi)_{2k-\pn(\pi)}},
\end{equation} is defined.   The calculations we have done to establish (1) and (2) of \eqref{eq:steps}  
show that whenever $\Xi_{k,\xi}$ is defined, it commutes with the generators $\mathfrak{s}_i$ for $1 \le i \le k-1$,  $\mathfrak{b}_1$, and $\mathfrak{p}_1$  of
$\P_k(\xi)$.  Consequently,  $\Xi_{k,\xi}$ is central in $\P_k(\xi)$ whenever it is defined.  In that case, (1) and (2) of
\eqref{eq:steps} imply  $\Xi_{k,\xi}^2 = \Xi_{k,\xi}$  and $\FF \Xi_{k,\xi}$ is a one-dimensional
$\P_k(\xi)$-module with $d_\pi \Xi_{k,\xi} = 0$ for all $\pi$ such that $\mathsf{pn}(\pi) < k$ and $\sigma \Xi_{k,\xi} =  \Xi_{k,\xi}$
for all $\sigma \in \S_k$.    In particular, we know when $\xi = 2k-1$ that
 $\Xi_{k,2k-1}$ is defined,  and  therefore, $\ef_{k,2k-1} = (-1)^k k ! \,\Xi_{k,2k-1}$ is central in $\P_k(2k-1)$.  The essential idempotent
$\ef_{k,2k-1}$ is unique with that property, as $\ef_{k,n}$ is never central when $2k-1 > n$ (compare Remark \ref{R:central}).  \end{remark}

\subsection{Some connections with the work of Rubey and Westbury}  

In \cite[Sec.~8]{RW1},  Rubey and Westbury consider a one-dimensional representation for
$D_k: = \End_{\S_n}(\modu^{\ot k})$ having the property that it is the trivial representation when
restricted to the elements of $D_k$ with propagating number $k$ and is equal to $0$ on the
other basis elements.  If $n \geq 2k$, the diagrams with propagating number $k$ are exactly
the permutation diagrams in $D_k = \P_k(n)$, and the representation is $\P_{k,n}^{[n-k,k]}$.   In \cite{RW1},
the central idempotent in $D_k$ associated to this representation is called $E(k)$.     Therefore, it follows from 
 (\ref{collect-terms})   and \eqref{eq:idems} that the central idempotent in \cite{RW1} has the following expression
 \begin{equation}\label{eq:Ek}  E(k) = \Xie_{k,n}  \qquad \text{when} \ n \ge 2k, \end{equation}
 where the expansions of $\Xie_{k,n}$ in the orbit and diagram bases of $\P_k(n)$ are given in \eqref{eq:idems}. 

\begin{remark}\label{R:nottrue} 
Rubey and Westbury go on to say  that $E(k)$ is up to scalar multiple
the orbit diagram with $2k$ singleton blocks.  We have shown that this is true only  if the 
expression for $\Xie_{k,n}$ in \eqref{eq:idems} is evaluated at $n = 2k-1$, and in that case,
$\Xie_{k,2k-1} = \frac{(-1)^k}{k!}\ef_{k,2k-1}$ by Theorem \ref{T:Xief}\,(a).  The idempotent does not correspond to 
a one-dimensional representation for  $\End_{\S_{2k-1}}(\mathsf{M}_{2k-1}^{\ot k})$ then.   The  kernel of the surjection  $\P_k(2k-1) \rightarrow  \End_{\S_{2k-1}}(\modu^{\ot k})$ is one-dimensional, 
spanned  by $\Xi_{k,2k-1} =\frac{(-1)^k}{k!} \ef_{k,2k-1}$, which is the orbit diagram in $\Pi_{2k}$ with $2k$ singleton blocks (see \eqref{eq:ef}).  \end{remark}

When $k \in \ZZ_{\ge 1}$,  the restriction of the one-dimensional representation of $D_{k+1}$
 to  $D'_{k+1}$ (the subalgebra of
$D_{k+1}$ of diagrams with $k+1$ and $2(k+1)$ in  the same block),  corresponds to a central idempotent of $D'_{k+1}$,
 which is denoted by $E'(k+1)$ in \cite{RW1}.    
The algebra $\P_{k+\half}(n)$ acts faithfully on $\modu^{\ot k}$ when $n \ge 2k+1$ (see Theorem \ref{T:Phi}\,(b))  and has
a one-dimensional module  $\P_{k+\half,n}^{[n-1-k,k]}$, with corresponding central
idempotent $\Xie_{k+\half,n} = \Xie_{k,n-1} \vertedge$ by Theorem \ref{T:IdOB}. Thus,
\begin{equation} \label{eq:Ek+1}  
E'(k+1) =  \Xie_{k+\half,n} = \Xie_{k,n-1} \vertedge \qquad \text{when} \ n \ge 2k+1, \end{equation}
where the expressions for $\Xie_{k+\half,n}$ in the orbit and diagram bases are displayed in \eqref{eq:Xihalf}. 
When $n = 2k$, the kernel of the surjection  $\P_{k+\half}(2k) \rightarrow \End_{\S_{2k-1}}(\M_{2k}^{\ot k})$ is generated by
$\ef_{k+\half, 2k} = \ef_{k,2k-1}\, \vertedge\!\!,$ as in Theorem
\ref{T:Xief}\,(b).   In fact,  we know from \cite[Thm.~5.5]{BHH} that the kernel
is one-dimensional, since the dimension of  $ \End_{\S_{2k-1}}(\M_{2k}^{\ot k})$ is $\mathsf{B}(2k+1)-1
= \dimm \P_{k+\half}(2k) -1$,  so the kernel is the $\CC$-span of  $\Xie_{k+\half,2k} =\Xie_{k,2k} \vertedge$
(see Theorem \ref{T:Xief}(b)).  By Theorem \ref{T:Xief}\,(b),  $\Xi_{k+\half,2k} = \frac{(-1)^k}{k!}\ef_{k+\half,2k}$, when $n = 2k$ .

We have not verified  the conjectured recurrence relations in \cite[Conjecture 8.4.9]{RW1} for $E(k)$ and $E'(k+1)$,
except for some small values of $k$;  however, in this work 
we have given explicit expressions for these idempotents in terms of orbit and diagram  basis elements.


\begin{thebibliography}{[FHHH]} 
 
 \bibitem[BCHLLS]{BCHLLS} G.~Benkart, M.~Chakrabarti, T.~Halverson, R.~Leduc, C.~Lee, and J.~Stroomer,
 Tensor product representations of general linear groups and their connections with Brauer algebras,
J. Algebra \textbf{166} (1994), no. 3, 529--567.
  
 \bibitem[BHH]{BHH} G.~Benkart, T.~Halverson, and N.~Harman,  \emph{Dimensions of  irreducible  modules  for partition algebras
and tensor power multiplicities for symmetric and alternating groups},  
 J. Algebraic Combin. {\bf 46} no. 1 (2017), 77-108;  arXiv \#1605.06543.
 
 
 \bibitem[BDO1]{BDO1} C.~Bowman, M.~DeVisscher, and R.~Orellana,   \emph{The partition algebra and
 the Kronecker coefficients}, Discrete Math. Theor. Comput. Sci. Proc. AS (2013), 321--332. 
 
 \bibitem[BDO2]{BDO2} C.~Bowman, M.~DeVisscher, and R.~Orellana, \emph{The partition algebra and the Kronecker coefficients}, Trans. Amer. Math. Soc. {\bf 367} (2015), 3647--3667.
 
 \bibitem[BEG]{BEG} C.~Bowman, J.~Enyang, and F.W.~Goodman, \emph{The cellular second fundamental theorem of invariant theory for
 classical groups},  arXiv \#1610.09009. 
 
 \bibitem[Br]{Br} R.~Brauer,  \emph{On algebras which are connected with the semisimple continuous groups}, Ann. of Math. {\bf 38} no. 4 (1937),  857--872.
 
\bibitem[CR]{CR} C.~Curtis and I.~Reiner, {\em Methods of Representation Theory -- With Applications to Finite Groups and Orders}, Pure and Applied Mathematics, vols. I and II, Wiley \& Sons, Inc., New York, 1987. 

 \bibitem[E]{E} J.~East,  \emph{Generators and relations for partition monoids and algebras},  J. Algebra {\bf 339} (2011), 1--26.

\bibitem[FH]{FH} W.~Fulton and J.~Harris, \emph{Representation Theory,   A First Course},  
Graduate Texts in Mathematics, {\bf 129}, Springer-Verlag, New York, 1991.

 \bibitem[GW]{GW} R.~Goodman and N.R.~Wallach,  \emph{Representations and Invariants of the Classical Groups}, 
Encyclopedia of Mathematics and Its Applications {\bf 68}, Cambridge University Press, Cambridge, 1998;
3rd corrected printing, Cambridge University Press (2003).
 
\bibitem[H]{H} T.~Halverson,  \emph{Characters of the partition algebras},  J. Algebra {\bf 238} (2001),  502--533.
       
\bibitem[HR]{HR}
T.~Halverson and A.~Ram,  \emph{Partition algebras}, European J.\ Combin. {\bf 26} (2005), 869--921.

\bibitem[HX]{HX} J.~Hu and Z.~Xiao,  \emph{On tensor spaces for Birman-Murakami-Wenzl algebras}, J. Algebra {\bf 324} (2010), 2893--2922.

\bibitem[J]{J} V.F.R. Jones, \emph{The Potts model and the symmetric group}, in:  Subfactors: Proceedings of the Taniguchi
Symposium on Operator Algebras (Kyuzeso, 1993), World Scientific Publishing, River Edge, NJ, 1994,
pp. 259--267.  

\bibitem[LZ1]{LZ1}  G.~Lehrer and R.~Zhang, R. B.
   \emph{The second fundamental theorem of invariant theory for the orthogonal group},
   {Ann. of Math.} {\bf 176}  (2012) 2031--2054.
   
\bibitem[LZ2]{LZ2}  G.~Lehrer  and R.~Zhang, R. B.
   \emph{The Brauer category and invariant theory},
   {J. Eur. Math. Soc.}  {\bf 17}, (2015) {2311--2351}.

\bibitem[M1]{M1} P. Martin, \emph{Representations of graph Temperley-Lieb algebras},  Publ. Res. Inst. Math. Sci. {\bf 26} (1990), no. 3, 485--503. 
 
\bibitem[M2]{M2}  P. Martin, \emph{Temperley-Lieb algebras for non-planar statistical mechanics--the partition algebra construction}, J. Knot Theory Ramifications {\bf 3} (1994) 51--82.

\bibitem[M3]{M3}  P. Martin, \emph{The structure of the partition algebra},  J. Algebra {\bf 183} (1996) 319--358.

\bibitem[M4]{M4} P. Martin, \emph{The partition algebra and the Potts model transfer matrix spectrum in high dimensions},  
J. Phys. A: Math. Gen. {\bf 33} (2000) 3669--3695.

\bibitem[MR]{MR} P.~Martin and G.~Rollet, \emph{The Potts model representation and a Robinson-Schensted correspondence for the
partition algebra}, Compositio Math.  {\bf 112} (1998) 237--254.

\bibitem[MW]{MW} P.~Martin and D.~Woodcock, \emph{On central idempotents in the partition algebra}, J.~Algebra {\bf 217}, 156--169.

\bibitem[OZ]{OZ} R.~Orellana and M.~Zabrocki,  \emph{Symmetric group characters as symmetric functions},
arXiv \#1605.06672v2.

\bibitem[RW1]{RW1} M.~Rubey and B.~Westbury, \emph{A combinatorial approach to classical representation theory},  arXiv:1408.3592.  

\bibitem[RW2]{RW2} M.~Rubey and B.~Westbury, \emph{Combinatorics of symplectic invariant tensors}, Proceedings of FPSAC 2015, 285--296, Discrete Math. Theor. Comput. Sci. Proc., Assoc. Discrete Math. Theor. Comput. Sci., Nancy, 2015.


\bibitem[St]{St}  R.P.~Stanley, \emph{Enumerative Combinatorics}, Vol. 1. Cambridge, England,  1997.

\end{thebibliography}
\end{document}